%% file: AFKMMNT_Infinitely_many_ends-rev3.tex
\documentclass[11pt]{article}

\RequirePackage[l2tabu, orthodox]{nag}
\usepackage[T1]{fontenc}
\usepackage[utf8]{inputenc}
\usepackage{lmodern}
\usepackage[french,english]{babel}
\usepackage{color} \usepackage[usenames,dvipsnames]{xcolor}
\usepackage{amsrefs,amsthm,amssymb,amsmath}
\usepackage[pdftex]{graphicx}
\usepackage[all]{xy}
\usepackage{enumerate}
\numberwithin{equation}{section}
\numberwithin{figure}{section}

\def\comment#1{}

\newcommand{\dfcn}[5]{\setlength{\arraycolsep}{1.5pt}\begin{array}{cccc}#1:&#2&\longrightarrow&#3\\&#4&\longmapsto&#5\end{array}}

\newcommand{\R}{\mathbb{R}}
\newcommand{\N}{\mathbb{N}}
\newcommand{\C}{\mathbb{C}}
\newcommand{\PP}{\mathbf{P}}
\newcommand{\Z}{\mathbb{Z}}
\newcommand{\T}{\mathbf{S}^1} %circle
\newcommand{\cE}{\mathcal{E}}
\newcommand{\cG}{\mathcal{G}}
\newcommand{\mI}{\mathcal{I}}
\newcommand{\stab}[2]{\mathrm{Stab}_{#1}(#2)}	%stabilizer
\newcommand{\eps}{\varepsilon}
\newcommand{\fhi}{\varphi}
\newcommand{\Sch}{\mathrm{Sch}}
\def\to{\mathop{\rightarrow}}
\def\dans{\mathop{\subset}}
\newcommand{\moins}{\setminus}
\newcommand{\pstar}{(\star)}

\newcommand{\Diff}{\mathrm{Diff}}
\newcommand{\NE}{\mathrm{NE}}
\newcommand{\mN}{\mathcal{N}}
\newcommand{\PSL}{\mathrm{PSL}}
\newcommand{\cD}{\mathsf{dev}}
\DeclareMathOperator{\rel}{rel}

\newcommand{\wR}{\widetilde{\mathcal{R}}}
\newcommand{\cR}{\mathcal{R}}
\newcommand{\cC}{\mathcal{C}}

\newtheorem{thm}{Theorem}[section]
\newtheorem{thmA}{Theorem}
\newtheorem*{thm*}{Theorem}

\newtheorem*{pdfn}{Proposition - Definition}
\newtheorem{lem}[thm]{Lemma}
\newtheorem{claim}{Claim}
\newtheorem*{claim*}{Claim}
\newtheorem{prop}[thm]{Proposition}
\newtheorem{cor}[thm]{Corollary}

\newtheorem*{mpconj}{``Missing Piece'' Conjecture}
\newtheorem{conj}[thm]{Conjecture}
\theoremstyle{definition}
\newtheorem{dfn}[thm]{Definition}
\newtheorem{notation}[thm]{Notation}
\newtheorem{conv}[thm]{Convention}
\theoremstyle{remark}
\newtheorem{rem}[thm]{Remark}
\newtheorem{ex}[thm]{Example}

\usepackage[margin=2.5cm]{geometry}

\usepackage{indentfirst}
\usepackage{microtype}

\usepackage[colorlinks=true,linkcolor=blue,citecolor=magenta]{hyperref}

\title{
Groups with infinitely many ends acting analytically on the circle
}

\date{}
\author{\begin{tabular}{ccc}
S\'ebastien Alvarez &  Dmitry Filimonov &
Victor Kleptsyn \\
Dominique Malicet & Carlos Meni\~{n}o & Andr\'es Navas \\ &
 Michele Triestino &
\end{tabular}
}

\begin{document}

\maketitle
\selectlanguage{french}
\begin{center}
\emph{Dedi\'e \`a \'Etienne Ghys \`a l'occasion de son 60\` eme anniversaire}
\end{center}

\selectlanguage{english}
\begin{abstract}
This article is inspired by two milestones in the study of non-minimal group actions on the circle: Duminy's theorem about the number of ends of semi-exceptional leaves, and Ghys' freeness result in real-analytic regularity. Our first result concerns groups of real-analytic diffeomorphisms with infinitely many ends: if the action is non expanding, then the group is virtually free. The second result is a Duminy type theorem for minimal codimension-one foliations: either non-expandable leaves have infinitely many ends, or the holonomy pseudogroup preserves a projective structure.
\end{abstract}

\tableofcontents

\section{Foreword and results}

The projective linear group $\PSL(2,\R)$ is the main source of inspiration for understanding groups of circle diffeomorphisms. Although not as huge as $\Diff_+(\T)$ -- it is only a three-dimensional Lie group versus an infinite dimensional group -- it is a good model to study several important aspects of subgroups of $\Diff_+(\T)$.

To begin, recall that $\PSL(2,\R)$ naturally acts on the circle $\T$ viewed either as the projective real line $\R\mathbf{P}^1$ or as the boundary of the hyperbolic plane. This action is clearly real analytic, thus we can see $\PSL(2,\R)$ as a subgroup of the group $\Diff^\omega_+(\T)$ of orientation-preserving real-analytic circle diffeomorphisms.

Several works have already described ``non-discrete'' (more precisely, \emph{non locally discrete}) subgroups of $\Diff^\omega_+(\T)$, if not thoroughly, at least in a very satisfactory way (see Ghys \cite{proches}, Shcherbakov \textit{et al.}~\cite{shcherbakov}, Nakai \cite{nakai}, Loray and Rebelo \cite{rebelo-ergodic,loray-rebelo,rebelo-transitive}, Eskif and Rebelo \cite{eskif-rebelo}, etc.). Morally, they resemble \emph{non-discrete} subgroups in $\PSL(2,\R)$, in the sense that, because of the presence of the so-called \emph{local flows}, their dynamics approximate \emph{continuous} dynamics (see \S~\ref{ssc:dynamics}).

A decade ago or so, some of the authors, in collaboration with Bertrand Deroin, started a systematic study of \emph{locally discrete} subgroups of $\Diff_+^\omega(\T)$  \cite{DKN2009,FK2012_C_eng,DKN2014,FKone,tokyo}. They introduced an auxiliary property, named $\pstar$ (and $(\Lambda\star)$, but we do not make a distinction here), under which groups behave roughly like \emph{Fuchsian groups}, \textit{i.e.}~discrete subgroups of $\PSL(2,\R)$. Informally speaking, property $\pstar$ requires that the action is  \emph{non-uniformly hyperbolic}: points at which hyperbolicity is lost must be \emph{parabolic} fixed points (or more generally the fixed point of some element with derivative~$1$). This is indeed the case for non-elementary Fuchsian groups (see Example~\ref{ex:fuchsian}). 

Starting from this, the final purpose of the aforementioned works is to show that property $\pstar$ is always satisfied. Conjecturally, it should be satisfied even for groups of $C^2$ circle diffeomorphisms, as $C^2$ is the lowest regularity setting where one always disposes of tools of control of affine distortion. However, the attention should be focused first on real-analytic actions, where arguments are often less technical.

The progresses obtained so far ensure property $\pstar$ by relating the dynamics with the algebraic structure of the group. The program proceeds by distinction of the number of \emph{ends} of the group. Extending the previous work \cite{DKN2014} on \emph{virtually free groups} (\textit{i.e.}~groups containing free subgroups of finite index), our first main result proves that property $\pstar$ holds for groups with \emph{infinitely many ends}:

\begin{thmA}\label{mthm:1}
Let $G$ be a finitely generated, locally discrete subgroup of $\mathrm{Diff}^\omega_+(\T)$. If $G$ has infinitely many ends, then 
it satisfies property $\pstar$, and it is virtually free.
\end{thmA}

Our second result goes in the reverse direction: property $\pstar$ determines the structure of the group. As we already mentioned, classical examples of locally discrete groups with property $\pstar$ are Fuchsian groups. Similarly, one can consider \emph{finite central extensions} of Fuchsian groups (\textit{i.e.}~discrete subgroups of a $k$-fold cover $\PSL^{(k)}(2,\R)$ of $\PSL(2,\R)$). A discrete group $\Gamma\subset \PSL^{(k)}(2,\R)$ is \emph{cocompact} if the quotient $\PSL^{(k)}(2,\R)/\Gamma$ is compact.
Cocompact discrete groups have only one end.

\begin{thmA}\label{mthm:2}
Let $G$ be a finitely generated, locally discrete subgroup of $\mathrm{Diff}^\omega_+(\T)$ satisfying property~$\pstar$. Then
\begin{itemize}
\item either $G$ is $C^\omega$-conjugate to a finite central extension of a cocompact Fuchsian group
\item or it is virtually free.
\end{itemize}
\end{thmA}

An exhaustive description of virtually free, locally discrete subgroups of $\mathrm{Diff}_+^\omega(\T)$ will be the object of a forthcoming work \cite{markov-partition-vfree}. The first possibility in Theorem~\ref{mthm:2} is actually due to Bertrand Deroin (Theorem~\ref{t:Deroin2}).

\paragraph{Motivations --}
Recall that if a group $G$ acts (continuously) on the circle $\T$ and there is no finite orbit, then the group  admits a unique \emph{minimal invariant compact set}, which 
can be the whole circle or a Cantor set. In the latter case, one says that the action has an \emph{exceptional} minimal set.
The most interesting dynamics takes place on this minimal set. For example, any semi-conjugacy restricts to a conjugacy on minimal sets, so that one can think of it as the ``incompressible'' part of the dynamics, from the \emph{topological} point of view. Is it the same from the \emph{measure-theoretical} point of view?

For this, notice that the notion of ergodicity can be naturally extended to transformations with \emph{quasi-invariant} measures (as for example the Lebesgue measure for any $C^1$ action) as saying that any $G$-invariant subset of the circle has either full or zero Lebesgue measure. Now, going back to the 80s, it was observed by Shub and Sullivan \cite{SS} that \emph{expanding} actions of subgroups $G\subset \Diff^{1+\alpha}_+(\T)$ have nice \emph{ergodic} properties: if the action is minimal then it is also ergodic with respect to the Lebesgue measure, whereas if the action has an exceptional minimal set~$\Lambda$, then the Lebesgue measure of $\Lambda$ is zero and the complementary set $\T\setminus\Lambda$ splits into finitely many distinct orbits of intervals (or \emph{gaps}). An analogous result was known for $\Z$ actions by $C^2$ circle diffeomorphisms: in case of minimality (which, according to 
Denjoy's theorem, is equivalent to that nontrivial elements have irrational rotation number \cite{denjoy}), the action is Lebesgue ergodic (this was independently proven by Katok \cite{KH} and Herman \cite{herman}).

One of the motivations for studying local flows for non locally discrete groups (see for instance \cite{rebelo-ergodic}) was to extend the method of Katok and Herman to more general actions. Indeed, the group generated by a minimal circle diffeomorphism $f$ is the most natural example of a non-discrete group (and thus non locally discrete): if $(q_n)$ is the sequence of denominators of the rational approximations of the rotation number of~$f$, then the sequence $f^{q_n}$ tends to the identity in the $C^1$ topology (see \cite[Ch.~VII]{herman} and also \cite{NT}).

One of the key ingredients behind the aforementioned results is the technique of control of the affine distortion of the action (highly exploited throughout this paper as well). In the 70-80s, this suggested the conjecture that the picture above should hold as soon as control of distortion can be sought (see~\cite{Schweitzer}).

\begin{conj}[Ghys, Sullivan]\label{conj:GS}
	Let $G\subset \Diff^2_+(\T)$ be a finitely generated subgroup whose action on the circle is minimal. Then the action is also Lebesgue ergodic.
\end{conj}

\begin{conj}[Ghys, Sullivan; Hector]\label{conj:GSH}
	Let $G\subset \Diff^2_+(\T)$ be a finitely generated subgroup whose action on the circle has an exceptional minimal set $\Lambda$. Then the Lebesgue measure of $\Lambda$ is zero, and the complementary set $\T\setminus\Lambda$ splits into finitely many orbits of intervals.
\end{conj}

Property $\pstar$ was indeed introduced in \cite{DKN2009} as a property under which these conjectures can be established by somewhat standard techniques. 
%Roughly, as we already mentioned, from the set $\NE\cap\Lambda$ of non-expandable points it is possible to define an expansion procedure.
More precisely, as done in \cite{FK2012_C_eng}, one defines Markov partition of the minimal set, with a  non-uniformly expanding map encoding the dynamics of $G$ (see \S~\ref{sc:Markovp}). This allows to extend the technique of Shub and Sullivan and prove the Conjectures~\ref{conj:GS} and~\ref{conj:GSH} for groups with property $\pstar$.

\paragraph{State of the art --} It is strongly believed that property $\pstar$ holds for any (finitely generated) subgroup of $\Diff^\omega_+(\T)$. 
Although property $\pstar$ is always satisfied by non locally discrete groups, whether it holds or not is a challenging question for locally discrete groups. 
This has already been verified for certain classes of groups: virtually free groups \cite{DKN2014} and finitely presented one-ended groups of bounded torsion \cite{FKone}. Theorem \ref{mthm:1} enlarges this list. 
In the real-analytic framework, we are still left with one class of groups.

\begin{mpconj}
	Let $G\subset \Diff^\omega_+(\T)$ be a finitely generated, one-ended subgroup. Assume that $G$ is neither finitely presented nor 
	has a sequence of torsion elements of unbounded order. Then $G$ cannot be locally discrete.
\end{mpconj}

For a brief summary, see also Table~\ref{table:summary}. This simplified conjecture needs further comments. Our impression is that if any counter-example existed, it should be very pathological. The feeling is that a locally discrete subgroup of $\Diff_+^\omega(\T)$ should be \emph{Gromov-hyperbolic}. Finitely generated Gromov-hyperbolic groups are always finitely presented and have bounded torsion (see \cite[Ch.~III.$\Gamma$]{bridson-haefliger}). Even if we are still not able to prove Gromov-hyperbolicity for general locally discrete groups, this has been done in one particular case:

\begin{thm}[Deroin]
	Let $G\subset \Diff^\omega_+(\T)$ be a locally discrete, finitely generated subgroup whose action on the circle is minimal and expanding. Then $G$ is Gromov-hyperbolic.
\end{thm}

The theorem above is actually an intermediate step for a much stronger result, which suggests that locally discrete subgroups of $\Diff^\omega_+(\T)$ are strongly related to Fuchsian geometry:

\begin{thm}[Deroin]\label{t:Deroin2}
	Let $G\subset \Diff^\omega_+(\T)$ be a locally discrete, finitely generated subgroup whose action on the circle is minimal and expanding. Then $G$ is analytically conjugate to a finite central extension of a cocompact Fuchsian group.
\end{thm}
These results appear in \cite{Deroin}. In the statements, \emph{expanding} means that for every $x\in\T$ there exists $g\in G$ such that $g'(x)>1$. 

The interested reader may consult the survey \cite{tokyo} for getting an idea of the landscape growing around the study of locally discrete groups.

\begin{table}[ht]
	\begin{center}
		\begin{tabular}{|c|c|c|c|}
			\hline
			\textbf{infinitely many ends} & \textbf{two ends} & \textbf{one end} & \textbf{one end}\\
			& & \textbf{expanding} & \textbf{non expanding}\\
			\hline
			virtually free  & virtually $\Z$ & finite central extension of a & conjectured to be \\
			&& cocompact Fuchsian group & impossible\\
			(Theorem \ref{mthm:1}) & (Corollary \ref{cor:finite-orbits}) & (Theorem~\ref{t:Deroin2}) & (partial result in \cite{FKone})\\
			\hline
		\end{tabular}
	\end{center}
	\caption{Classification of locally discrete subgroups of $\Diff_+^\omega(\T)$.}\label{table:summary}
\end{table}

\paragraph{Outline of the paper --}
In Sections \ref{s:basic} and \ref{s:technical} we introduce the main definitions and objects, which are both of dynamical and group-theoretical nature. The proof of Theorem~\ref{mthm:1} is worked out in Section~\ref{s:mthma}, combining dynamics with Bass-Serre theory. On the other hand, the proof of Theorem~\ref{mthm:2} is more involved, and is obtained by joining two intermediate results, here Theorems~\ref{t:duminy} and \ref{t:ends_germs}, together with Theorem~\ref{mthm:1}. Theorem~\ref{t:duminy} is discussed in Section~\ref{s:thmc}, and the proof is a nice interplay between geometry and dynamics. Section~\ref{s:thmd} is devoted to the proof of Theorem~\ref{t:ends_germs}, which is definitely the most technical part of this paper.

%%%%%%%%%%%%%%%%%%%%%%%%%%%%%%%%%%%%%

\section{Basic definitions and preliminaries}
\label{s:basic}

Let us introduce the main ingredients, which are both of dynamical and group-theoretical nature. This will be also the occasion for a better discussion on the background.

\subsection{Ends and groups}

\begin{dfn}
	Let $X$ be a connected topological space. Let $(K_n)_{n\in\N}$ be an increasing sequence 
	of compact subsets $K_n\subset X$, such that the union of their interiors covers $X$. An \emph{end} of $X$ is a decreasing sequence
	\[
	\cC_1\supset\cC_2\supset\ldots\supset \cC_n\supset\ldots,
	\]
	where $\cC_n$ is a connected component of $X\setminus K_n$.
	We denote by $e(X)$ the \emph{space of ends} of $X$; it does not depend on the choice of $(K_n)$.
\end{dfn}

Note the cardinality of $e(X)$, called the \emph{number of ends} of $X$, is the least upper bound, possibly infinite, for the number of unbounded connected components of the complementary sets $X\setminus K$, where $K$ runs through the compact subsets of $X$.

The space of ends carries a natural \emph{topology}:
an open subset $V$ in $X$ induces an open subset in $e(X)$ given by the subset of sequences $(\cC_n)$ so that $\cC_n\subset V$ for all but finitely many $n$.
Also the topology does not depend on the choice of $(K_n)$.
For nice topological spaces (connected and locally connected) the space of ends defines a compactification of $X$.

\begin{dfn}
	A sequence of points $(x_n)$ in $X$ \emph{goes to an end} if for every compact subset $K\subset X$, there exist $n_0$ and a connected component $\mathcal C$ of $X\setminus K$ such that $x_n\in \mathcal C$ for all $n\ge n_0$.
\end{dfn}

If $G$ is a group generated by a finite set $\mathcal G$, we define the space of ends $e(G)$ of $G$ to be the space of ends of the \emph{Cayley graph} of $G$ relative to $\mathcal G$. This is the graph whose vertices are the elements of $G$, and two elements $g,h\in G$ are joined by an edge if $g^{-1} h\in\mathcal G$. The graph metric induces the \emph{length metric} on $G$ given by the following expression:
\[d_\mathcal G(g,h)=\min\{\ell\mid g^{-1}h=s_1\cdots s_\ell,\,s_j\in \mathcal G\cup\mathcal G^{-1}\}.\]
The \emph{length of an element} $g \in G$ is defined as $\|g\| = d_{\mathcal{G}} (id,g)$. We will denote as usual by $B(n)=\{g\in G\mid \|g\|\le n\}$ the ball of radius $n$ centered at the identity.
Here a \emph{graph} will always be a one-dimensional  complex, endowed with any metric which is compatible with the graph metric defined on vertices.

It is a classical fact \cite[\S~8.30]{bridson-haefliger} that the space of ends, and hence the number of ends, of a group does not depend on the choice of the finite generating set (this easily follows from the fact that Cayley graphs associated with different finite generating systems are bilipschitz equivalent). Moreover, the number of ends does not change when passing to finite extensions or finite-index subgroups.
Furthermore, finitely generated groups can only have $0$, $1$, $2$ or infinitely many ends. Groups with $0$ or $2$ ends are not of particular interest: they are respectively finite or virtually infinite cyclic, \textit{i.e.}~they contain $\Z$ as a finite index subgroup (we refer to \cite[\S~8.32]{bridson-haefliger} for further details). Although they represent a broader class, groups with infinitely many ends may also be algebraically characterized, according to the celebrated Stallings' theorem \cite{stallings} (due to Bergman \cite{bergman} in the case of torsion). Before stating it, we recall two basic operations on groups.

\begin{dfn}
	Let $G_1$ and $G_2$ be two groups, and denote by $\rel G_i$ the set of relations in $G_i$. 
	Let $Z$ be a group which embeds in both $G_1$ and $G_2$ via morphisms $\phi_i:Z\hookrightarrow G_i$, $i=1,2$. The \emph{amalgamated product} $G_1*_ZG_2$ of $G_1$ and $G_2$ over the group $Z$ is defined by the presentation
	\[
	\langle G_1,G_2\mid \rel G_1,\rel G_2\text{ and }\phi_1(z)=\phi_2(z)\text{ for every }z\in Z\rangle.
	\]
\end{dfn}
Amalgamated products arise, for example, in the classical van Kampen theorem.
It is clear that if $G_1$ and $G_2$ are finitely generated, then any amalgamated product $G_1*_ZG_2$ is also finitely generated. Conversely, 
if $Z$ and $G_1*_ZG_2$ are finitely generated, then $G_1$ and $G_2$ are also finitely generated.

\begin{dfn}
	Let $H$ be a group and $Z$ another group that embeds in two (possibly equal) ways into $H$ via morphisms $\phi_i:Z\hookrightarrow H$, $i=1,2$. The \emph{HNN extension} $H*_Z$ of $H$ over $Z$ is defined by the presentation
	\[
	\langle H,\sigma \mid \rel H,\text{ and }\phi_1(z)=\sigma \phi_2(z)\sigma^{-1}\text{ for every }z\in Z\rangle.\]
	The generator $\sigma$ is usually called the \emph{stable letter} of the extension.
\end{dfn}

The most basic examples are the Baumslag-Solitar groups $\mathrm{BS}(m,n)=\langle t,\sigma\mid t^n=\sigma t^m\sigma^{-1}\rangle$, which correspond to HNN extensions 
of the type $\Z*_\Z$ (here the embeddings $\phi_i:\Z\hookrightarrow \Z$ are the multiplications by $m$ and~$n$, respectively).

From an algebraic point of view, an HNN extension $H*_Z$ is isomorphic to the semi-direct product of $\Z$ (generated by $\sigma$) and a bi-infinite chain of amalgamated products of copies of $H$. 
As before, if $H$ is finitely generated, then any HNN extension $H*_Z$ is also finitely generated. Conversely, if $Z$ and $H*_Z$ are finitely generated, then $H$ is also finitely generated.
We refer the reader to~\cite{serre,baumslag} for more details.

Proper HNN extensions and amalgamated products different from $\Z_{2k}*_{\Z_k}\Z_{2k}$ have infinitely many ends. The converse is also true:

\begin{thm}[Stallings]\label{thm:stallings}
	Let $G$ be a finitely generated group with infinitely many ends. Then $G$ is either an amalgamated product $G_1*_ZG_2$ over a finite group $Z$ (different from $G_1$ and $G_2$) or an HNN extension $H*_Z$ over a finite group $Z$ (different from $H$).
\end{thm}

Given  a finitely generated group $G$ with infinitely many ends, we shall call \emph{Stallings' decomposition} any possible decomposition of $G$ as an amalgamated product or as an HNN extension over a finite group. 

\medskip

In the second part of this work we study the geometry of orbits. To this extent, we recall the notion of \emph{Schreier graph}, which is nothing but the generalization of 
Cayley graphs to group actions.

\begin{dfn}
	Let $G$ be a finitely generated group acting on a space, let $\mathcal{G}$ be a finite generating set and $X$ an orbit for the action. The \emph{Schreier graph} of the orbit $X$, denoted by $\Sch(X,\mathcal G)$, is the graph whose vertices are the elements of $X$, and two vertices $x,y\in X$ are joined by an edge if there exists $s\in\mathcal G$ such that $s(x)=y$. The graph metric on $X$ is induced by the length metric on $G$:
	\[d^X_{\mathcal G}(x,y)=\min\left \{d_{\mathcal G}(id,g)\mid g(x)=y\right \}.\]
\end{dfn}

\begin{rem}
\label{r:indep_generators}
As for Cayley graphs, the space of ends of Schreier graphs, and hence the number of ends, do not depend on the choice of the finite generating set.
\end{rem}
We will simply write $e(X)$ instead of $\Sch(X,\mathcal G)$ for the space of ends (this is justified by the remark above). However, a Schreier graph might not have the same number of ends as the Cayley graph, even in nice cases like faithful actions, with ``small'' point stabilizers. For example, Thompson's group $T$ is one-ended, it acts on the circle by $C^\infty$ diffeomorphisms \cite{thompson}, and there are Schreier graphs associated with this action that have infinitely many ends, as Duminy's theorem (Theorem~\ref{Duminy}) guarantees.

\smallskip

Finally, we introduce a graph structure for the \emph{groupoid of germs} $G_{x_0}$ defined at a point~$x_0$. Fix a finite generating set $\mathcal G$ for $G$. Recall that two diffeomorphisms $f$ and $g$ define the same \emph{germ} at a point $x_0$ if there exists a neighbourhood $U$ of $x_0$ such that the restrictions of $f$ and $g$ to $U$ coincide. In the following, we identify a germ with any diffeomorphism representing it. The germs usually do not define a group, but rather a groupoid. For our purposes, it is enough to consider $G_{x_0}$ simply as a metric space as follows:
$G_{x_0}$ is formed by all the germs defined at $x_0$ and equipped with the distance
\[
d_{\mathcal{G},x_0}(g,h)=\min\left\{\ell\in \N\,\middle\vert\, g^{-1}h\vert_U=s_1\cdots s_\ell\vert_U,\, s_j\in \mathcal G\cup\mathcal G^{-1},\text{ for some neighbourhood }U\ni x_0  \right\}.
\]
We define the connected graph $\widehat{\Sch}(X,\mathcal G)$, called the \emph{holonomy covering} of $\Sch(X,\cG)$ whose set of vertices are the elements of $G_{x_0}$, and two vertices $g,h$ are joined by an edge if there exists $s\in \cG$ such that $g^{-1}h\vert_U=s\vert_U$, for some neighbourhood $U$ of $x_0$.
As for Caley graphs and Schreier graph, the space of ends of the holonomy covering does not depend on the finite generating set, so that we can simply refer to the space of ends of the groupoid of germs.
\begin{rem}\label{r:holcover}
	The natural map $g\in G_{x_0}\mapsto g(x_0)\in X$ extends to a covering map $\widehat{\Sch}(X,\mathcal G)\to \Sch(X,\cG)$, which corresponds to the classical holonomy covering of a leaf in foliation theory. (Remark also that there is another natural covering, from the Cayley graph of $G$ to the holonomy covering.)
\end{rem}

\subsection{Dynamics}\label{ssc:dynamics}

\paragraph{Locally discrete groups of real-analytic circle diffeomorphisms --}

Let $G$ be a group acting  (continuously) on the circle $\T$, with no finite orbit, and let $\Lambda\subset \T$ be its minimal invariant set. Because of the minimality of the action on the minimal set $\Lambda$, the local dynamics around a point $x\in\Lambda$ is essentially the same as the local dynamics around any other point $y\in\Lambda$. Roughly speaking, the dynamics of $G$ on $\Lambda$ is encoded in the restriction of the action of $G$ to any open interval $I$ intersecting~$\Lambda$.

\begin{dfn}
A subgroup $G\subset \mathrm{Diff}_+^1(\T)$ is \emph{$C^1$ locally discrete} if for any interval $I\subset \T$ intersecting a minimal set, the restriction of the identity to $I$ is isolated in the $C^1$ topology among the set of restrictions to $I$ of the diffeomorphisms in $G$.

In what follows, we will simply refer to this property as \emph{locally discrete}.
\end{dfn}

Even if the previous definition is given for subgroups of $\mathrm{Diff}_+^1(\T)$, 
we focus our interest on subgroups of $\mathrm{Diff}_+^\omega(\T)$. The huge difference between $C^\omega$ and lower regularity is the following:

\begin{thm}[see Proposition~3.7 of \cite{Matsuda2009}]\label{t:hector}
Let $G\subset \Diff_+^{\omega}(\T)$ be a finitely generated, locally discrete subgroup. Then the stabilizer in $G$ of every point is either trivial or infinite cyclic.
\end{thm}

The next corollary essentially describes locally discrete groups with finite orbits. 

\begin{cor}\label{cor:finite-orbits}
Let $G\subset \Diff_+^{\omega}(\T)$ be a finitely generated, locally discrete subgroup with a finite orbit. Then $G$ is either cyclic or 
contains an index-2 subgroup which is the direct product of an infinite cyclic group with a finite cyclic group.
\end{cor}

\begin{rem}
Notice that the index-$2$ subgroup above arises when a rotation conjugates 
an element with fixed points into its inverse (as it is the case of involution $x \to -1/x$ with respect to the hyperbolic M\"obius transformation $x \to \lambda x$, 
with $\lambda \neq 1$, both viewed as maps of the circle $\T \sim \mathbb{R} \mathbf{P}^1$). Indeed, such a group $G$ is either cyclic or isomorphic to a semi-direct product $\Z\rtimes \Z_m$, where $\Z_m$ is isomorphic to $\mathrm{rot}(G)\subset\R/\Z$ (in the case of	 a finite orbit, the rotation number defines a homomorphism).
\end{rem}

Theorem~\ref{t:hector} is a consequence of a well-known result due to Hector, and we refer to it as ``Hector's lemma'' (see \cite[Th\'eor\`eme 2.9]{euler} and \cite{proches,navas2006}). Generalizing  Hector's lemma to lower regularity is a longstanding major problem in codimension-one foliations \cite[pp.~448\ndash 449]{dippolito}. It is also the major reason why our results hold in this wide generality only for subgroups of $\Diff_+^{\omega}(\T)$.

\paragraph{Non locally discrete groups of analytic circle diffeomorphisms --}
If a subgroup $G\subset \mathrm{Diff}_+^\omega(\T)$ is locally discrete, then it is also discrete (with respect to the $C^1$ topology). As a matter of fact, there is no deep reason for privileging local discreteness above discreteness: we believe that the two notions coincide, but this would be a consequence of our aimed classification. 
Indeed, appropriate dynamical tools are known only when working with (non) local discreteness.

As we mentioned at the beginning, non locally discrete groups have been studied in several works, mainly by Shcherbakov, Nakai, Loray and Rebelo.
The fundamental tool, which goes back to \cite{shcherbakov,nakai,loray-rebelo}, is the following result that establishes the existence of \emph{local flows in the local closure} of the group. We state it in the form of~\cite[Proposition 2.8]{DKN2014}:

\begin{prop}\label{p:localvf}
Let $I$ be an interval on which nontrivial real-analytic diffeomorphisms $f_k \in \Diff^\omega(I,\T)$ are defined. Suppose that the sequence $f_k$ 
converges to the identity in the $C^1$
topology on $I$, and let $f$ be another $C^\omega$ diffeomorphism having a hyperbolic fixed point
on $I$. Then there exists a (local) $C^1$ change of coordinates $\phi : I\longrightarrow [-1, 2]$ after which the pseudogroup $G$ generated by the $f_k$'s and $f$ contains
in its $C^1([0, 1], [-1, 2])$-closure a (local) translation sub-pseudogroup:
\[\overline{\big\{\phi g\phi^{-1}\vert_{[0,1]}\mid g\in G\big\}}\supset \big\{x\mapsto x+s\mid s\in [-1,1]\big\}.\]
\end{prop}

%\vspace{2cm}

Under mild assumptions, existence of elements with hyperbolic fixed points is guaranteed by the classical Sacksteder's theorem (\cite{Sacksteder}, see also \cite{conf,DKN2007,Navas2011}). 
We state a more general version (in class $C^1$) due to Deroin, Kleptsyn and Navas, inspired by a similar result of Ghys in the $C^2$ context.

\begin{thm} \label{t:sacksteder}
Let $G$ be a finitely generated group of $C^1$ circle diffeomorphisms. If $G$ admits no invariant probability measure on $\T$, 
then it contains an element that has a hyperbolic fixed point in the minimal invariant set of $G$.
\end{thm}

Observe that a group with an invariant measure either is semi-conjugate to a group of rotations or has a finite orbit. 
Joining Proposition~\ref{p:localvf} and Sacksteder's theorem together, we have that if a finitely generated group 
$G\subset \Diff_+^\omega(\T)$ acts minimally with no invariant measure and is non locally discrete, then it has local vector flows in its local closure.

\paragraph{Non-expandable points --}

The existence of local flows in the closure of a group of circle diffeomorphisms 
yields rich dynamics. For instance, the action either has a periodic orbit, or is minimal and Lebesgue ergodic \cite{rebelo-ergodic}. If, besides, 
there is no invariant probability measure, one deduces from Sacksteder's theorem that the action must be \emph{expanding}, in the following sense:

\begin{dfn}
A point $x\in\T$ is \emph{non expandable} for the action of a group $G$ of circle diffeomorphisms 
if for every $g\in G$, the derivative of $g$ at $x$ is not greater than $1$. We denote by $\mathrm{NE}=\mathrm{NE}(G)$ the set of non-expandable points of $G$.
The action of a group of circle diffeomorphisms is \emph{expanding} if $\mathrm{NE}=\emptyset$.
\end{dfn}

Since we have $\NE=\left\{x\mid g'(x)\le 1\text{ for every }g\in G\right\}=\bigcap_{g\in G}\left\{x\mid g'(x)\le 1\right\},$
the set of non-expandable points is always closed. Notice that one can define the set of non-expandable points for any group of $C^1$ circle diffeomorphisms. However, 
it is important to point out that, \emph{a priori}, the definition does not well behave under $C^1$ conjugacy: only the property $\NE=\emptyset$ is invariant under $C^1$ conjugacy. The problem is that the notion of non-expandable points is not a \emph{dynamical} one. The following definition, introduced in \cite{DKN2007}, forces a conjugacy-invariant condition (\textit{cf}.~\cite[Corollary~1.10]{DKN2007}).

\begin{dfn}[Property $\pstar$ -- $C^\omega$ case]\label{d:pstar}
Let $G\subset \mathrm{Diff}_+^\omega(\T)$ be a subgroup with no finite orbit, and let $\Lambda$ be its minimal invariant 
set. The group $G$ has \emph{property $\pstar$} if for every $x\in \mathrm{NE}\cap \Lambda$ there is $g\in G\setminus\{id\}$ 
such that $x$ is a fixed point of $g$.
\end{dfn}

\begin{ex}\label{ex:fuchsian}
In the case of Fuchsian groups, for which property $\pstar$ is satisfied, orbits of non-expandable points are geometrically interpreted as \emph{cusps}.
%(more generally, as boundary components). 
For instance, consider actions of non-uniform lattices $\Gamma$ in $\PSL(2,\R)$, that is subgroups for which the quotient $\mathbf{H}^2/\Gamma$  is not compact but has finite volume. The most classical examples are $\PSL(2,\Z)$ and its finite index free subgroups like
\[\Gamma=\left \langle
\begin{bmatrix}
1 & 2\\0&1
\end{bmatrix},\begin{bmatrix}
1 & 0\\2&1
\end{bmatrix}\right \rangle\]
(the quotient $\mathbf{H}^2/\Gamma$ is a sphere with three cusps). In these cases, the orbit of the set of non-expandable points is made of the rational numbers together with the point at infinity in $\R\mathbf{P}^1\cong \R\cup \{\infty\}$. In the quotient space $\mathbf{H}^2/\Gamma$, these points coincide with the cusps.
\end{ex}

Property $\pstar$ makes sense even for $C^1$ actions. However it turns to be a useful notion only when working with actions that are at least of class $C^2$ (as it must be combined with classical tools of control of affine distortion, as Proposition \ref{l:schwartz0}). In most issues, there is no relevant difference between $C^2$ and $C^\omega$ actions with property $\pstar$. However, the definition of property $\pstar$ in class $C^2$ is slightly more complicated, as one has to take into account that there could be elements that are the identity on some interval.

\begin{dfn}[Property $\pstar$ -- $C^2$ case]
Let $G\subset \mathrm{Diff}_+^2(\T)$ be a subgroup with no finite orbit, and let $\Lambda$ be its minimal invariant set. The group $G$ has \emph{property $\pstar$} if for every $x\in \mathrm{NE}\cap \Lambda$ there are $g_+$ and $g_-$ in $G$ such that $x$ is an isolated fixed point in $\Lambda$ from the right (resp.~from the left) for $g_+\vert_\Lambda$ (resp. $g_-\vert_\Lambda$).
\end{dfn}

Property $\pstar$, even in $C^2$ regularity, entails several strong properties for the dynamics of the group action. For a detailed discussion, the reader may consult \cite{DKN2007} or \cite[\S~3.5]{Navas2011}. Here we collect the results that are relevant to our purposes. First of all, if $\NE \neq \emptyset$, then the group is locally discrete.
Secondly, the set $\NE\cap \Lambda$ intersects only finitely many orbits (also, when $\Lambda$ is a Cantor set, there are only finitely many orbits of connected components of the complement $\T\setminus \Lambda$). This can be seen as a consequence of the work \cite{DKN2009} where an \emph{expansion procedure} was introduced, and later improved by Filimonov and Kleptsyn in \cite{FK2012_C_eng}. In this latter work the authors show that, if $\NE\neq \emptyset$, the dynamics on the minimal set can be encoded by a ``Markov'' dynamics. We will give a more precise account later in \S~\ref{sc:Markovp}, as this fact is one fundamental ingredient for the proof of Theorem~\ref{mthm:2}.

%%%%%%%%%%%%%%%%%%%%%%%%%%%%

\section{More technical preliminaries}
\label{s:technical}

\input{Section3.tex}

%%%%%%%%%%%%%%%%%%%%%%%%%%%%%

\section{Theorem \ref{mthm:1}: Property $\pstar$ for groups with infinitely many ends}
\label{s:mthma}

\input{Section4.tex}

%%%%%%%%%%%%%%%%%%%%%%%%%%%%%%%%%%%%

\section{Theorem~\ref{t:duminy}: Duminy revisited}
\label{s:thmc}

\input{Section5.tex}

%%%%%%%%%%%%%%%%%%%%%%%%%%%%%%%%%%%%

\section{Theorem~\ref{t:ends_germs}: Ends of the groupoid of germs}
\label{s:thmd}

\input{Section6.tex}

%%%%%%%%%%%%%%%%%%%%%%%%%%%%%%%%%%%%

\section*{Acknowledgements}
The authors wish to thank Pablo Barrientos and Artem Raibekas for having taken active part in the process of understanding most of the background material, during workshop sessions at UFF and PUC in 2014.
M.T. acknowledges the hospitality of USACH and the discussions with A.N. and Crist\'obal Rivas around Ghys' Theorem during the visit in December 2014.
This work was carried on during the visit of D.F., V.K. and A.N. to PUC in Rio de Janeiro in January 2015.
The presentation of the paper has been largely improved after the generous suggestions of the referees.

\smallskip

S.A., C.M., D.M. and M.T. were supported by a post-doctoral grant financed by CAPES.
M.T. was supported by PEPS -- Jeunes Chercheur-e-s -- 2017 (CNRS).
S.A. was supported by the project \emph{Geometric theory of dynamical systems and France-Brazil cooperation in mathematics}, sponsored by Marcelo Viana's prix Louis D, as well by the Université de Bourgogne.
C.M. was supported by Fundaci\'on Barri\'e de la Maza, post-doctoral fellow 2012 and MICINN, Grant MTM2014-56950-P (2014--2017) (Spain).
D.F. and V.K. were partially supported by the RFBR projects 13-01-00969-a and 16-01-00748-a and 
by the project CSF of CAPES.
V.K. was partially supported by the R\'eseau France-Br\'esil en Math\'ematiques.
A.N. was supported by the Anillo 1103 Research Project DySyRF and its compagnion Project REDES 140138. 

\bibliographystyle{plain}
\def\cprime{$'$}

\begin{flushleft}
{\scshape S\'ebastien Alvarez}\\
CMAT, Facultad de Ciencias, Universidad de la Rep\'ublica\\
Igua 4225 esq. Mataojo. Montevideo, Uruguay.\\
email: \texttt{salvarez@cmat.edu.uy}

\smallskip

{\scshape Dmitry Filimonov}\\
National Research University Higher School of Economics (HSE)\\
20 Myasnitskaya ulitsa,
101000 Moscow, Russia\\
email: \texttt{mityafil@gmail.com}

\smallskip

{\scshape Victor Kleptsyn}\\
CNRS, Institut de R\'echerche Math\'ematique de Rennes (IRMAR, UMR 6625)\\
B\^at. 22-23, Campus Beaulieu,
263 avenue du G\'en\'eral Leclerc,
35042 Rennes, France\\
email: \texttt{victor.kleptsyn@univ-rennes1.fr}

\smallskip

{\scshape Dominique Malicet}\\
Universidade Federal Fluminense (UFF)\\
Rua Prof. Marcos Waldemar de Freitas Reis, S/N -- Bloco H, 4o Andar\\
Campus do Gragoatá, Niterói, Rio de Janeiro 24210-201, Brasil\\
email: \texttt{dominique.malicet@crans.org}

\smallskip

{\scshape Carlos Meni\~no Cot\'on}\\
Rua Prof. Marcos Waldemar de Freitas Reis, S/N -- Bloco H, 4o Andar\\
Campus do Gragoatá, Niterói, Rio de Janeiro 24210-201, Brasil\\
email:  \texttt{carlos\_menino@id.uff.br}

\smallskip

{\scshape Andr\'es Navas}\\
Universidad de Santiago de Chile (USACH)\\
Alameda 3363, Estaci\'on Central, Santiago, Chile\\
email: \texttt{andres.navas@usach.cl}

\smallskip

{\scshape Michele Triestino}\\
Institut de Math\'ematiques de Bourgogne (IMB, UMR 5584)\\
9 av.~Alain Savary, 21000 Dijon, France\\
email: \texttt{michele.triestino@u-bourgogne.fr}

\end{flushleft}

\end{document}

%% file: Section3.tex
In this section we develop some technical tools required for the proofs of our main results.
For Theorem~\ref{mthm:1}, will need in particular a long discussion about group acting on trees (\S~\ref{ssc:bass-serre}), while for Theorem~\ref{mthm:2} we recall the results of \cite{FK2012_C_eng} and describe the Markov partition associated with a group with property $\pstar$.
We also take the opportunity to introduce some notation for further reference.

\subsection{Groups: Basic Bass-Serre theory for amalgamated products and actions on trees}
\label{ssc:bass-serre}

In this part we recall some elementary facts about groups acting on trees. Many of these are well-known results, and we only sketch the proofs.  The main results in this section are Proposition~\ref{l:distorted} and Proposition~\ref{l:ping-pong}, and will be important for the proof of Theorem~\ref{mthm:1}.

\paragraph{Normal forms --}
Every element in an amalgamated product can be written in a \emph{normal form} (see \cite{combinatorial,serre}).

\begin{lem}\label{l:normalform}
	Fix transversal sets of cosets $T_1\subset G_1$ and $T_2\subset G_2$ for ${}_Z\backslash{}^{G_1}$ and ${}_Z \backslash^{G_2}$ respectively, both containing the identity. Then every element $g\in G$ has a unique factorization as $g= \gamma\, t_n\cdots t_1$, with $\gamma\in Z$ and $t_j\in T_{i_j}\setminus \{id\}$, with none of two consecutive $i_j$'s equal.
\end{lem}

We sketch a geometrical proof of this lemma using Bass-Serre theory \cite{serre}.
Every amalgamated product acts isometrically on a simplicial tree without edge-inversion, namely the \emph{Bass-Serre tree}, that we denote it by $X$. Bass-Serre theory holds more generally, but for an amalgamated product $G=G_1*_ZG_2$, the tree and the action on it have a simple description: the vertices are the cosets $\{G_ig\mid g\in G,i=1,2\}$, and the edges are $\{(G_1g,G_2g)\mid g\in G\}$. The group $G$ acts by \emph{right} multiplication: $G_ig.\fhi=G_ig\fhi$. The edge $e=(G_1,G_2)$ is a \emph{fundamental domain} for the action of $G$ on $X$: each factor group $G_i$ coincides with the stabilizer of the vertex $G_i$, and $Z=G_1\cap G_2$ is the stabilizer of the edge $e$.

Remark that if $(G_1g,G_2g)$ and $(G_1g',G_2g')$ represent the same edge, then we have $G_ig=G_ig'$ for $i=1,2$. We deduce that $g'g^{-1}$ belongs to the intersection $G_1\cap G_2=Z$. So $g'=\gamma g$ for some $\gamma\in Z$.

\begin{proof}[Proof of Lemma~\ref{l:normalform}]
	If an element $g\in G$ belongs to a factor group $G_i$, then there is a unique $t\in T_i$ and $\gamma\in Z$ such that $g=\gamma t$. 	
	If an element $g$ is not in a factor group, then the fundamental domain $e$ and its image $e.g$ do not intersect. Therefore, since $X$ is a tree, there is a unique geodesic path $\pi$ connecting $e$ to $e.g$.
%	 (see Figure~\ref{fig:path}). 
The path is of the form
	\[\pi=(G_{i_1}=G_{i_1}g_1,G_{i_2}g_2,G_{i_3}g_3,\ldots,G_{i_{n}}g_{n}=G_{i_{n}}g),\]
	with the $g_k$'s verifying $G_{i_k}g_{k}=G_{i_{k}}g_{k-1}$ for every $k=2,\ldots n$, and none of two consecutive $i_j$'s equal.
	From the remark above, the $g_k$'s are uniquely defined modulo $Z$. However, if the transversal sets $T_1$ and $T_2$ are given, then we can write every $g_k$ in the form
	\begin{align*}
	g_1=\,&t_1,\\
	g_2=\,&t_2t_1,\\
	\cdots&\\
	g_n=\,&t_n\cdots t_1,\text{ with  every }t_j\in T_{i_j}\setminus \{id\},
	\end{align*}
	which is unique. 
	Since $G_{i_n}g_n = G_{i_n}g$, $G_{i_{n+1}} g_{n+1} = G_{i_{n+1}} g$ 
	and $G_{i_{n+1}} g_{n+1} = G_{i_{n+1}} g_n$, the product $g g_n^{-1}=\gamma$ belongs to $Z = G_{i_n} \cap G_{i_{n+1}}$.
\end{proof}

%\begin{figure}
%	\[\includegraphics[scale=1]{geodesics_BS_tree-0.pdf}\]
%	\caption{The geodesic path in the Bass-Serre which determines the normal form. Here is an example where we have $t_1\in G_2$ and $t_n\in G_2$.}\label{fig:path}
%\end{figure}

\begin{rem}\label{r:normal}
	Consider an element $g\in G=G_1*_ZG_2$, written in normal form as $g=\gamma\,t_n\cdots t_1$. 
	Observe that if $g$ is written differently as $g=s_k\cdots s_1$ with every $s_j\in G_{i_j}\setminus Z$ and none of two consecutive $i_j$'s equal, then $k=n$, and for every $j=1,\ldots,n$, the factor $t_j$ belongs to $G_{i_j}$. Moreover every product $t_j^{-1}s_j$ belongs to $Z$.
	
	Indeed, the length $n$ is exactly the length of the geodesic path in the Bass-Serre tree $\pi$ from the edge $e=(G_1,G_2)$ to the edge $e.g=(G_1g,G_2g)$, and the indices $i_j$'s correspond to the vertices visited by the path.
	
	We also have that the inverse $g^{-1}$ can be written in a normal form of length $n$, since the geodesic path from $e$ to $e.g^{-1}$ is the translation $\pi.g^{-1}$ (with opposite orientation), of the path $\pi$ from $e$ to $e.g$.
\end{rem}

\paragraph{Tree isometries --}
Let us study in more detail the action of $G$ on its Bass-Serre tree. The reader may consult \cite{culler-morgan} for a general description of actions on (real) trees.

As in the previous proof, we denote by $X$ the Bass-Serre tree of $G$ and by $d$ the graph metric on the tree $X$. We keep the convention of right action.

\begin{dfn}
	Given an isometry $\fhi$ of the tree $(X,d)$, we denote by $\ell(\fhi)$ its \emph{translation length}:
	\begin{equation}\label{eq:translation-length}
	\ell(\fhi)=\min\{d(x,x.\fhi)\mid x\in X\}.
	\end{equation}
	If $\ell(\fhi)=0$ then $\fhi$ is \emph{elliptic}, otherwise, $\fhi$ is \emph{hyperbolic}.
\end{dfn}

Observe that the minimum in \eqref{eq:translation-length} is attained because the distance $d$ on $X$ takes discrete values when restricted to the set of vertices. In particular, we have:
\begin{lem}\label{l:conjugate_factor}
	Let $G=G_1*_Z G_2$ be an amalgamated product and let $X$ be its Bass-Serre tree. Take an element $\fhi\in G$. The following statements are equivalent: 
	\begin{enumerate}
		\item the element $\fhi$ belongs to a conjugate factor group (\textit{i.e.}~a subgroup of $G$ of the form $gG_ig^{-1}$);
		\item $\fhi$ fixes a vertex of $X$;
		\item $\fhi$ is elliptic, that is, $\ell(\fhi)=0$.
	\end{enumerate}
\end{lem}

Any tree isometry $\fhi$ has a natural \emph{invariant set} $X(\fhi)$, which is a convex subset of $X$. This is the union of the minimal invariant sets. More explicitly, for an elliptic element, $X(\fhi)$ is defined as the set of fixed points of $\fhi$. Observe that $\fhi$ fixes more than one point if and only if $\fhi$ belongs to some conjugate of the edge group $Z$.

For hyperbolic elements, the invariant set is described as follows:
\begin{lem}\label{l:dist_image}
	If a tree isometry $\fhi:X\to X$ is hyperbolic, the invariant set $X(\fhi)$ is a \emph{translation axis}, \textit{i.e.}~an invariant bi-infinite geodesic line in $X$, on which $\fhi$ acts as a translation of displacement $\ell(\fhi)$.
	Moreover, for any vertex $x\in X$, one has 
	\begin{equation}\label{eq:dist_image}
	d(x,x.\fhi)=\ell(\fhi)+2d(x,X(\fhi)).
	\end{equation}
\end{lem}

\begin{proof}[Sketch of the proof]
	Consider a vertex $x\in X$ that minimizes the translation length: $d(x,x.\fhi)=\ell(\fhi)$. We denote by $\pi=(x=x_0,x_1,\ldots,x_{\ell(\fhi)}=x.\fhi)$ the geodesic path from $x$ to $x.\fhi$ in $X$. The segments $\pi$ and $\pi.\fhi$ only intersect at $x.\fhi$,
%Indeed, since $X$ is a tree, the intersection $\pi\cap \pi.\fhi$  is connected, and if it were not a point, then the points $x_1$ and $x_1.\fhi$ would be at distance $\ell(\fhi)-2$, contradicting the minimality (see Figure~\ref{fig:axis} on the left).
%	\begin{figure}
%		\[\includegraphics[scale=1]{geodesics_BS_tree-1.pdf}\qquad\includegraphics[scale=1]{geodesics_BS_tree-2.pdf}\]
%		\caption{Existence of the translation axis for a hyperbolic element.}\label{fig:axis}
%	\end{figure}
	therefore, the union 
	\begin{equation}\label{eq:axis}
	X(\fhi)=
	\bigcup_{n\in\Z}\pi.\fhi^n
	\end{equation}
	is a bi-infinite geodesic in $X$, on which $\fhi$ acts as a translation by $\ell(\fhi)$. %(see Figure~\ref{fig:axis} on the right).
	Uniqueness, and \eqref{eq:dist_image} are an easy consequence of the fact that $\fhi$ acts as a translation on $X(\fhi)$.
%	Let us prove that the translation axis is unique. First suppose that two translation axes intersect. Then the intersection is a (possibly infinite) interval that is $\fhi$-invariant.
%	Since $\fhi$ acts as a translation on the intersection, the only possibility is that both axes actually coincide. Secondly, if the intersection is empty, then the two axes must be parallel. More precisely, the set of shortest geodesic paths connecting the two axes is invariant under the isometry $\fhi$. However, such a geodesic path is unique in a tree, 
%	hence it must be fixed by $\fhi$ (see Figure~\ref{fig:unique}). This contradicts $\ell(\fhi)>0$.
%	\begin{figure}
%		\[
%		\includegraphics[scale=1]{geodesics_BS_tree-3.pdf}\qquad \includegraphics[scale=1]{geodesics_BS_tree-4.pdf}
%		\]
%		\caption{Uniqueness of the translation axis for a hyperbolic element.}\label{fig:unique}
%	\end{figure}
%	
%	Finally, let us prove \eqref{eq:dist_image}. Let $\gamma$ be the geodesic segment from $x$ to $X(\fhi)$, with endpoint $y\in X(\fhi)$ (so that $d(x,y)=d(x,X(\fhi))$). Then $\gamma.\fhi$ is the geodesic segment from $x.\fhi$ to $X(\fhi)$, with endpoint $y.\fhi$. Then we have
%	\[d(x,x.\fhi)=d(x,y)+d(y,y.\fhi)+d(y.\fhi,x.\fhi),\]
%	from which one easily deduces \eqref{eq:dist_image}.
\end{proof}

\begin{rem}\label{r:elliptic}
	The relation \eqref{eq:dist_image} holds even for an elliptic isometry $\fhi$, in which case we simply have
	\[
	d(x,x.\fhi)=2d(x,X(\fhi)).
	\]
	More precisely, if $\gamma$ is the geodesic segment from $x$ to $X(\fhi)$, with endpoint $y\in X(\fhi)$, then $\gamma.\fhi$ is the geodesic segment from $x.\fhi$ to $X(\fhi)$, with endpoint $y.\fhi=y$.
\end{rem}

The following result gives a geometric condition for detecting the position of the invariant set of a tree isometry. The proof is elementary and left to the reader.

\begin{prop}\label{p:position}
	Let $\fhi$ be a non-trivial isometry of a tree $X$, and $x\in X\setminus X(\fhi)$ a vertex.
	Let $\pi^+,\pi-$ the geodesic paths in $X$ connecting $x$ with $x.\fhi$ and $x.\fhi^{-1}$ respectively. Then $\pi^+$ and $\pi^-$ share the first edge $e$ and $X(\fhi)$ is contained in the connected component of $X\setminus \{x\}$ which contains $e$.
\end{prop}

\begin{rem}
	As it will appear clear from the proof (\textit{cf.}~also Remark~\ref{r:elliptic}), if the element $\fhi$ is elliptic, then it is enough to look at the geodesic path from $x$ to $x.\fhi$: if $\pi^+$ starts with the edge~$e$, so does~$\pi^-$.
\end{rem}

%\begin{proof}
%	Suppose first that the isometry $\fhi$ is elliptic. The statement is clearly empty if $x=x.\fhi$, so we suppose that we are not in this situation. Let $\pi^+$ be the geodesic segment connecting $x$ to $x.\fhi$. As described in Remark~\ref{r:elliptic}, the invariant set $X(\fhi)$ intersects the path $\pi^+$ exactly at its middle point.
%	
%	Suppose now that $\fhi$ is hyperbolic. Observe that our hypothesis guarantees that $x$ does not belong to the hyperbolic axis $X(\fhi)=X(\fhi^{-1})$, because otherwise the paths $\pi^-$ and $\pi^+$ would intersect only at $x$.
%	Then, as we saw in Lemma~\ref{l:dist_image}, the paths $\pi^\pm$ from $x$ to $x.\fhi^{\pm 1}$ decompose into three nontrivial pieces:
%	first, the path reaches the translation axis $X(\fhi)$, then it crosses the axis along a segment of length $\ell(\fhi)$, and finally it goes out of $X(\fhi)$ to reach the image $x.\fhi$. By uniqueness of geodesics in a tree, the first pieces for $\pi^-$ and $\pi^+$ must coincide. This gives the result.
%\end{proof}

\paragraph{Distorted elements --} First, we recall the following:
\begin{dfn}
	An element $\fhi$ of a finitely generated group $G$ is undistorted (in $G$) if the length of the element $\fhi^n$ grows linearly in $n$. 
	(Notice that this definition is invariant under quasi-isometries and in particular it does not depend on the finite generating system chosen for defining the length metric on $G$.)
\end{dfn}

In free groups all nontrivial elements are undistorted; by invariance under quasi-isometries one gets:

\begin{lem}\label{l:dist_free}
	Let $G$ be a finitely generated virtually free group. Then every element of infinite order is undistorted in $G$.
\end{lem}

We will need the following more general statement.

\begin{lem}\label{l:distorted}
	Let $G=G_1*_ZG_2$ be an amalgamated product and let $\fhi\in G$ be a distorted element in $G$. Then $\fhi$ is conjugate to an element into one of the two factors (and it is actually distorted in the conjugate factor with respect to the restricted metric).
\end{lem}

\begin{proof}
	Because of Lemma~\ref{l:conjugate_factor}, it is enough to prove that if the element $\fhi$ is hyperbolic, then it is undistorted.
	If $\fhi$ was distorted, given $x\in X(\fhi)$ the distance $d(x,x.\fhi^n)$ would have sublinear growth, but 
	$\fhi^n$ acts by translation by $n\ell(\fhi)$ on $X(\fhi)$.
\end{proof}

%\begin{proof}
%	 There is however a simpler, more classical, proof. Indeed, up to quasi-isometry, it is enough to prove the result for a group $G$ which is \emph{free}. For free groups there are many ways to see this, here we choose to give an argument relying on actions on trees. Indeed,
%	one of the first byproducts of the Bass-Serre theory is that a group is free if and only if it has a free action on a tree. So consider such a free action: every element acts as a hyperbolic isometry, so it is undistorted.
%\end{proof}

\paragraph{Ping-pong and free groups --}
Let us first give a statement about commutators in a free group:

\begin{lem}\label{l:CommPowers}
	In the rank-two free group $F_2$, consider two free generators $a$ and $b$. Define the sequence of iterated commutators
	\[
	\begin{cases}
	w_0=a,\\
	w_1=b,\\
	w_{k+2}=[w_k,w_{k+1}].
	\end{cases}
	\]
	Let $H$ be the free subgroup generated by $w_2$ and $w_3$. Given an element $h\in H$, the following property holds:
	when writing $h$ as a reduced word in the generating system $\{a^{\pm 1},b^{\pm 1}\}$, then the expression does not contain $a^{\pm 2}$ and $b^{\pm 3}$ as subwords.
\end{lem}

The following nice proof has been explained to us by Jarek K\c edra on MathOverflow.

\begin{proof}
	Every element in the commutator subgroup $[F_2,F_2]$ can be represented by an oriented closed path on the square grid $\Z^2$, starting at the origin: the letters $a,b$ are represented by edges going to the right and up, respectively.
	Since the subgroup generated by $w_2,w_3$ is contained in $[F_2,F_2]$, we can use this interpretation for any element in $H$. 
	
	In this interpretation, the element $w_2$ is represented by a simple square loop, while $w_3$ is represented by a loop describing a ``figure eight'', namely two vertically adjacent squares (see Figure~\ref{fig:loops}). 
	\begin{figure}
		\[
		\includegraphics[scale=1]{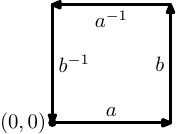}\quad\includegraphics[scale=1]{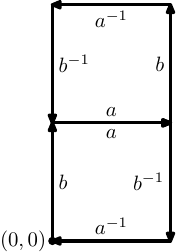}
		\]
		\caption{The paths representing the commutators $w_2$ (left) and $w_3$ (right).}\label{fig:loops}
	\end{figure}
	
	Thus every element in the group $H$ describes a closed loop that is contained in the figure eight, simply because when concatenating $w_2^{\pm 1},w_3^{\pm 1}$, the support of the loops cannot escape. In particular the reduced form for an element $h\in H$ cannot contain powers of $a^{\pm 1}$ exceeding $1$, otherwise the support of  
	the loop it represent would escape the figure eight from one of its vertical sides. Similarly we deduce that there is no power of $b^{\pm 1}$ exceeding $2$.
\end{proof}

\begin{dfn}\label{d:boundorder}
	Let $G$ be a group acting of isometries of a tree $X$. Let $\beta\in \N$ be a positive integer. We say that $G$ is $\beta$-bounded if for any isometry $\fhi\in G$ fixing an edge of $X$, then $\fhi^{n}=id$ for some $|n|\le \beta$. In other words, $\beta$ is a uniform upper bound on the order of isometries of $G$ fixing edges. 
\end{dfn}

\begin{lem}\label{l:boundorder}
	Let $G$ be a group of isometries of a tree $X$, which is $\beta$-bounded. If $\fhi\in G$ is such that there exists a positive integer $p\in \N$ such that $\fhi^{p}$ fixes an edge, then $\fhi$ has order at most $\beta p$.
\end{lem}

\begin{proof}
	It follows directly from Definition~\ref{d:boundorder} above.
\end{proof}

\begin{lem}\label{l:boundorder2}
	Let $G$ be a group of isometries of a tree $X$, which is $\beta$-bounded. Consider an isometry $\fhi\in G$ whose order is at least $5\beta$ (possibly infinite).
	
	Consider a connected component $C$ of the complement $X\setminus X(\fhi)$ of the invariant set of $\fhi$. Then for every power $p\in \{\pm 1,\ldots, \pm 4\}$, the image $\fhi^p(C)$ has empty intersection with $C$.
\end{lem}

\begin{proof}
	Since $C$ is a connected component of the complement of the invariant set $X(\fhi)$, there is a unique edge $e$ connecting $X(\fhi)$ to $C$.
	
	Suppose there is $p>0$ such that $\fhi^p\neq id$ and the intersection $\fhi^p(C)\cap C$ is not empty. The power~$\fhi^p$ must fix the edge $e$, so $\fhi^p$ fixes one edge. As $G$ is $\beta$-bounded, Lemma~\ref{l:boundorder} implies that we have $\fhi^{p\beta}=id$. Thus, by hypothesis, we must have $p\beta\ge 5\beta$. This implies $p\ge 5$.
	
	When $p<0$, considering $\fhi^{-1}$ we find similarly $p\le -5$. This ends the proof.
\end{proof}

Now we can proceed to the main result of this paragraph, which is a variation on the classical ping-pong lemma:

\begin{prop}[Ping-pong]\label{l:ping-pong}
	Let $G$ be a group acting by isometries on a tree $X$, which is $\beta$-bounded.
	Let $\fhi,\psi\in G$ be two tree isometries such that:
	\begin{enumerate}
		\item their invariant sets are disjoint,
		\item their order is at least $5\beta$ (possibly infinite).
	\end{enumerate} 
	Then $h=[\fhi,\psi]$ and $k=[\psi,[\fhi,\psi]]$ generate a free subgroup of $G$.
\end{prop}

\begin{proof}
	Let $\fhi$ and $\psi$ be two isometries with disjoint invariant sets $X(\fhi)$ and $X(\psi)$. Denote by $\pi$ the geodesic path in $X$ connecting these two sets. Let $v$ and $w$ be the vertices 
	on $\pi$ that lie on $X(\fhi)$ and $X(\psi)$ respectively.
	We consider the following two subtrees of $X$:
	\begin{enumerate}
		\item $A$ is the maximal subtree of $X$ that contains $v$ but not the rest of $\pi$;
		\item $B$ is the maximal subtree of $X$ that contains $w$ but not the rest of $\pi$.
	\end{enumerate}
	Consider an element $g$ in the group generated by $h$ and $k$. Up to cyclical rewriting (that is, up to conjugating by an element in $\langle \fhi,\psi\rangle$)\footnote{Notice that the group generated by $h,k$ is not normal, so the cyclical rewriting may transform $g$ into an element which does not belong to this group. However this has no influence on the rest of the proof.} the element $g$ decomposes as a product
	\begin{equation}\label{eq:altern}
	g=a_1b_1\cdots a_nb_n,\quad a_i\in \langle \fhi\rangle, b_i\in \langle \psi\rangle,
	\end{equation}
	which is (formally) reduced in the free group $F(\fhi,\psi)$.
	Moreover Lemma~\ref{l:CommPowers} implies that
	\[a_i\in \{\fhi^{\pm 1},\cdots \fhi^{\pm 4}\},b_i\in\{\psi^{\pm 1},\cdots,\psi^{\pm 4}\}\quad\text{for every }i=1,\ldots,n:
	\]
	indeed the lemma says initially that powers are bounded by $2$, however after a cyclical rewriting powers may increase up to $4$.
	Thus, applying Lemma~\ref{l:boundorder2}, we observe the following ping-pong dynamics:
	\[
	B.a_i\subset A\quad \text{and}\quad A.b_i\subset B\quad\text{for every }i=1,\ldots,n.
	\]
	Therefore, if we apply $g$ to $B$, we must have
	\[
	g.B\subset A.
	\]
	As $A$ and $B$ are disjoint, this implies that $g$ is not the identity in $G$.
\end{proof}

Next, we detect the translation axis of certain hyperbolic elements, as this will be needed for verifying the first condition in Proposition~\ref{l:ping-pong}.
\begin{lem}\label{l:find_axis}
	Let $G=G_1*_Z G_2$. Consider an element $\fhi\in G$ of the form
	\begin{equation}\label{eq:form_good}
	\fhi=\sigma_n t_n\sigma_{n-1} t_{n-1}\cdots \sigma_1 t_1,\quad \text{with }t_i\in G_1\setminus Z,\sigma_i\in G_2\setminus Z\text{ for every }i=1,\ldots n.
	\end{equation}
	Set $e=(G_1,G_2)$. Then $\fhi$ is hyperbolic, and its translation axis is
	\[
	X(\fhi)=\bigcup_{k\in \Z}(\pi\cup e).\fhi^k,
	\]
	where $\pi$ is the geodesic path between $e$ and the image $e.\fhi$. That is, $X(\fhi)$ is the bi-infinite geodesic path
	\begin{equation}\label{eq:form_axis}
	X(\fhi)=(\ldots,G_2t_n^{-1}\sigma_n^{-1}, G_1\sigma_n^{-1},G_2, G_1,G_2t_1,G_1\sigma_1 t_1,\ldots, G_2t_n\sigma_{n-1}\cdots t_1,G_1\fhi, G_2t_1\fhi,\ldots).
	\end{equation}
	In particular, we have $\ell(\fhi)=2n$. 
%	(See Figure~\ref{fig:form_axis}.)
\end{lem}

%\begin{figure}
%	\[\includegraphics[scale=1]{geodesics_BS_tree-7.pdf}\]
%	\caption{The translation axis of an element $\fhi=\sigma t$, with $t\in G_1\setminus Z,\sigma\in G_2\setminus Z$.}\label{fig:form_axis}
%\end{figure}

\begin{proof}
	We have to prove that the path \eqref{eq:form_axis} is geodesic. That is, we have to prove that there is no backtracking, which is the same as proving that any two vertices on it are distinct. This can be verified directly from the uniqueness of the normal form (Lemma~\ref{l:normalform} and Remark~\ref{r:normal}), noticing that the normal form of a power $\fhi^k$ is
	\[
	(\sigma_n t_n\sigma_{n-1} t_{n-1}\cdots \sigma_1 t_1)\cdots (\sigma_n t_n\sigma_{n-1} t_{n-1}\cdots \sigma_1 t_1),
	\]
	with the  $(\sigma_n t_n\sigma_{n-1} t_{n-1}\cdots \sigma_1 t_1)$ repeated $k$ times.
\end{proof}

\begin{rem}\label{r:conjugate_axis}
	For any $g\in G$ and $\fhi$ of the form  \eqref{eq:form_good}, the translation axis of the conjugate $\psi=g\fhi g^{-1}$ is $X(\psi)=X(\fhi).g^{-1}$.
\end{rem}

%%%%%%%%%%%%%%%%%%%%%%%

\subsection{Dynamics: Distortion, Markov partition and expansion procedure}\label{sc:Markovp}

\paragraph{Distortion --} Let $J\subset \T$ be an interval, the \emph{distortion coefficient} of a $C^1$ diffeomorphism $g:J\to g(J)$ on $J$ is defined as
\begin{equation}\label{eq:distortion_coeff}
\varkappa(g;J)=\sup_{x,y\in J} \left \vert \log \frac{g'(x)}{g'(y)}\right \vert.\end{equation}
This measures how far is $g$ to be an affine map. Besides, this is well behaved under composition and inversion:
\[\varkappa(gh;J)\le \varkappa(g;h(J))+\varkappa(h;J),\quad\varkappa(g;J)=\varkappa(g^{-1};g(J)).\]
Suppose now that $G\subset \Diff_+^2(\T)$ is a finitely generated subgroup. If we fix a finite generating system $\mathcal G$ of the group $G$ and 
set  $C_\mathcal G=\max_{g\in \mathcal G\cup \mathcal G^{-1}}\sup_{\T} |g''/g'|$, then
\[\varkappa(g;J)\le C_{\mathcal{G}}\,|J|\quad\text{for every }g\in\mathcal G.\]
This implies that if $g=g_n\cdots g_1$ belongs to the ball of radius $n$ in $G$, $g_i\in\mathcal G$, then 
\begin{equation}\label{eq:intermediate}
\varkappa(g_n\cdots g_1;J)\le C_{\mathcal G}\sum_{i=0}^{n-1}|g_i\cdots g_1(J)|,
\end{equation}
where $g_i\cdots g_1=id$ for $i=0$.

The inequality \eqref{eq:intermediate} suggests that the control of the affine distortion of $g$ on some small interval $J$ 
can be controlled by the \emph{intermediate compositions} $g_i\cdots g_1$. This is better explained in the following way: 
Let 
\begin{equation}\label{eq:S}
S=\sum_{i=0}^{n-1}(g_i\cdots g_1)'(x_0)
\end{equation}
denote the sum of the intermediate derivatives at some \emph{single} point $x_0\in \T$. 
Then the affine distortion of $g$ can be controlled in a neighbourhood of radius $\sim 1/S$ about $x_0$. More precisely, 
we have the following statement (which goes back to A.~Schwartz \cite{Schwartz} and, later,  to Sullivan \cite{sullivan}): 

\begin{prop}\label{l:schwartz0}
	Let $G\subset \Diff_+^2(\T)$ be a finitely generated subgroup with finite generating set $\mathcal G$.
	For a point $x_0\in \T$ and $g\in B(n)$, let $S$ be as in \eqref{eq:S} and $c=\log2/4C_{\mathcal G}$. For every $r\le c/S$, we have the following bound on the affine distortion of $g$:
	\[\varkappa(g;U_{r}(x_0))\le 4C_{\mathcal G}Sr,\]
	where $U_r(x_0)$ denotes the $r$-neighbourhood of $x_0$.
\end{prop}

In the case of groups of real-analytic circle diffeomorphisms, every element is defined, by definition, on some annular complex neighbourhood of $\T$. The control of distortion in Proposition~\ref{l:schwartz0} is then extended to a complex neighbourhood of $x_0$:

\begin{prop}\label{l:schwartz}
	Let $G\subset \Diff_+^\omega(\T)$ be a finitely generated subgroup with finite generating set $\mathcal G$.
	For a point $x_0\in \T$ and $g\in B(n)$, let $S$ be as in \eqref{eq:S} and $c=\log2/4C_{\mathcal G}$. There exists $\rho>0$ such that for every $r\le \min\{c/S,\rho\}$, we have the following bound on the affine distortion of $g$:
	\[\varkappa(g;U_{r}^{\C}(x_0))\le 4C_{\mathcal G}Sr,\]
	where $U^{\C}_r(x_0)$ denotes the complex $r$-neighbourhood of $x_0$.
\end{prop}

\paragraph{Markov partition --}
We recall one result of \cite{FK2012_C_eng} in the case of minimal actions:

\begin{thm}[Filimonov, Kleptsyn]\label{t:Markov}
	\label{t:Markov2}
	Let $G\subset \Diff_+^2(\T)$ be a finitely generated subgroup whose action is minimal and with property $\pstar$. Let $\ell$ be the number of non-expandable points of $G$, and write $\NE=\{x_1,\ldots,x_\ell\}$. Then there exist a finite subset $\Delta_0\subset \T$ and a partition of $\T\setminus \Delta_0$ into finitely many open intervals
	\[\mathcal I=\left\{I_1,\ldots,I_k,I_1^+,I^-_1,\ldots,I_\ell^+,I_\ell^-\right\},\]
	an expansion constant $\lambda>1$ and elements $g_I\in G$, $I\in \mathcal I$ such that:
	\begin{enumerate}[\bf i.]
		\item for every $I\in \mathcal I$, the image $g_I(I)$ is a union of intervals in $\mathcal I$;
		\item we have $g'_I\vert_I\ge \lambda$ for every $I=I_1,\ldots, I_k$;\label{i:markov2}
		\item the intervals $I_i^+$ and $I_i^-$ are adjacent respectively on the right and on the left to the non-expandable $x_i$, which is the unique fixed point, topologically repelling, for $g_{I_i^+}$ (resp.~$g_{I_i^-}$) on the interval $I_i^+$ (resp.~$I_i^-$); moreover $x_i$ is the unique non-expandable point in $g_{I_i^\pm}(I_i^\pm)$;
		\item\label{i:markov4} for every $I=I_1^\pm,\ldots,I_\ell^\pm$, set
		\[k_I:I\longrightarrow \N\]
		to be the function $k_I(x)=\min\{k\in\N\mid g^k_{I} (x) \not\in I\}$ and
		\[j:I\longrightarrow \{1,\ldots,k\}\]
		defined by the condition $g_I^{k_I(x)}(x)\in I_{j(x)}$. Then for every $x\in I$, $\left(g_{I_{j(x)}}\circ g_I^{k_I(x)}\right)'(x)\ge \lambda$.
	\end{enumerate}
\end{thm}

\begin{rem}\label{r:keyCw}
	If we assume moreover that $G$ is in $\Diff_+^\omega(\T)$, then \ref{i:markov4} above can be reformulated as follows:
	if $k_I(x)=\min \{k\in\N\mid g_I^k(x)\notin I\}$, then for every $x\in I$ one has $\left(g_I^{k_I(x)}\right )'(x)\ge \lambda$.
	
	Indeed, as $g_I$ is a parabolic stabilizer of one of the endpoints $x_I$ (say the leftmost one) of the interval~$I$, there exist $A,B>0$ and $n\ge 1$ an integer such that
	\begin{equation}\label{eq:ParNormalForm}
	g_I(x)=x\left (1+A(x-x_I)^{n}+o((x-x_I)^n)\right )\quad \text{for every }x\in I, \text{ as }x\to x_I
	\end{equation}
	and
	\begin{equation}\label{eq:ParNormalForm2}
	g'_I(x)=1+B(x-x_I)^{n}+o((x-x_I)^n)\quad \text{for every }x\in I, \text{ as }x\to x_I.
	\end{equation}
	Therefore the derivative of $g_I$ is never less than one on a small right neighbourhood of $x_I$.
	This fact will be crucial in our proof of Theorem~\ref{t:ends_germs} and hence Theorem~\ref{mthm:2}.
\end{rem}

\begin{rem}
	It is worthwhile to observe that Theorem \ref{mthm:2} was first conjectured in \cite{FKone} as a moral consequence of Theorem~\ref{t:Markov}: the (non-uniformly) expanding maps $g_I$'s give a way to decompose the Schreier graphs of all but finitely many orbits into a finite number of trees \cite{FK2012_C_eng}, thus suggesting freeness in the structure.
\end{rem}

\paragraph{Magnification maps --} 
From now on, we fix a finite subset $\Delta_0\subset \T$, a Markov partition \[\mathcal I=\left\{I_1,\ldots,I_k,I_1^+,I^-_1,\ldots,I_\ell^+,I_\ell^-\right\},\]
an expansion constant $\lambda>1$ and elements $g_I\in G$, $I\in \mathcal I$ given by Theorem~\ref{t:Markov2}. 
We introduce a first \emph{magnification map} $\mathcal R:\T\setminus \Delta_0\to \T$ defined as
\begin{equation}\label{eq:expR}
\mathcal R\vert_I =g_I \quad\text{for any }I\in\mathcal I,
\end{equation}
and its modification $\widetilde{\mathcal R}:\T\setminus \Delta_0\to \T$  defined as
\begin{equation}\label{eq:expR2}
\widetilde{\mathcal R}\vert_I:x\in I\mapsto\begin{cases}
g_I(x)&\text{if }I\in\left \{I_1,\ldots,I_k\right \}\\ 
g_{I_{j(x)}} g_I^{k_I(x)}&\text{if }I\in \left \{I_1^{\pm},\cdots, I_\ell^{\pm}\right \}
\end{cases},\quad\text{for any }I\in\mathcal{I},
\end{equation}
which, after Theorem \ref{t:Markov2}.\ref{i:markov4} above, is \emph{uniformly expanding}: $\wR'(x)\ge \lambda$ for any $x\in\T\setminus \Delta_0$.

The following result will be very helpful during the proof of Theorem~\ref{t:ends_germs}:
\begin{lem}\label{l:keyCw}
	Assume we are under the hypotheses of Theorem~\ref{t:Markov2} and suppose moreover that $G\subset \Diff_+^\omega(\T)$. Then the magnification map $\mathcal R$ can be chosen to be everywhere expanding:
	\[
	\mathcal R'(x)>1\quad \text{for every }x\in \T\setminus \Delta_0.
	\]
\end{lem}

\begin{proof}
	The magnification map is piecewise defined by \eqref{eq:expR},
	however, it depends on the construction of the collection $\mathcal I$ in Theorem~\ref{t:Markov2}.
	The proof in \cite{FKone} starts first by fixing neighbourhoods $I_j^\pm$ of the non expandable points $\{x_1,\ldots,x_\ell\}$, then subdividing the rest of the circle into intervals $I_j$. Taking smaller neighbourhoods $I_j^\pm$ has usually the result of decreasing the expansion constant $\lambda>1$.
	
	If $I$ is one of the $I_j^\pm$, then we have seen in Remark~\ref{r:keyCw} that $\mathcal{R}\vert_I=g_I\vert_I$ is of the form \eqref{eq:ParNormalForm}, and its derivative of the form \eqref{eq:ParNormalForm2}. Hence, shrinking $I$ a little in Theorem~\ref{t:Markov2}, we may assure $(\mathcal{R})'\vert_I=g_I'\vert_I>1$.
	
	On the other hand, if $I\in \mathcal I$ is one of the $I_j$, then we already have a good expansion by construction: $\mathcal R'\vert_I\ge \lambda$ after Theorem~\ref{t:Markov2}.\ref{i:markov2}.
\end{proof}

\paragraph{Partitions of higher level --}
In order to encode the dynamics within the orbit of the set of non-expandable points, it is appropriate to define subpartitions of $\mathcal I$.

\begin{notation}\label{def:Deltak}
	We define the endpoints of the atoms of the partition of level $k$ by the following inductive procedure, starting from the set $\Delta_0$ of endpoints of atoms of the partition $\mI$.
	If $\Delta_k$ is constructed, 
	consider $\Delta_k(I)=\Delta_k\cap I$, where $I\in\mI$, so that $\Delta_k=\bigcup_{I\in\mI}\Delta_k(I)$. We distinguish two possibilities:
	\begin{itemize}
		\item if $I$ is not adjacent to a non-expandable point, set
		\[\Delta_{k+1}(I)=g_I^{-1}(\Delta_k\cap g_I(I));\]
		\item for $I\in\mI$ adjacent to one of the non-expandable points, set
		\[\Delta_{k+1}(I)=\bigcup_{j=1}^{\infty}g_I^{-j}(\Delta_k\cap(g_I(I)\moins I)).\]
	\end{itemize}
\end{notation}

\begin{dfn}
	The connected components of $\T\moins\Delta_k$ form a partition called the \emph{partition of level $k$} that we denote by $\mI_k$.
\end{dfn}

\paragraph{Expansion of a non-expandable point --} 
We start by the following result describing the orbits of non-expandable points (see for instance \cite[Lemma 3.5.14]{Navas2011}).

\begin{lem}
	\label{characNE}
	Let $G\subset \Diff_+^2(\T)$ be a finitely generated subgroup whose action is minimal and satisfies property $\pstar$. Then a point $x\in\T$ belongs to the orbit of a non-expandable point  if and only if the set $\lbrace g'(x)\mid g\in G\rbrace$ is bounded.
\end{lem}

One of our main tools is a process of expansion that we describe below. Assume that $x\in G\cdot\NE$. There exists $k(x)\in\N\cup\{\infty\}$ and a sequence of $k(x)$ points $(x_i)_{i=0}^{k(x)}\subset G\cdot \NE$, that we call the \emph{expansion sequence} of $x$ and is defined recursively as follows. First, set $x_0=x$. Now assume that $x_i$ has been constructed. Then there exists $I\in\mI$ such that $x_i\in\bar{I}$ (if $x_i$ is one of the endpoints of $I$, one can always require that it is the left one). Then we have three \emph{mutually exclusive} possibilities:

\begin{itemize}
	\item if $x_i\in\NE$, then the procedure stops and $k(x)=i$;
	\item if $I$ is not adjacent to a non-expandable point, we set $x_{i+1}=g_{i+1}(x_i)$, where $g_{i+1}=g_I$;
	\item if one endpoint of $I$ is a non-expandable point we set $x_{i+1}=g_{i+1}(x_i)$, where $g_{i+1}=g_{I_{j(x_i)}} g_I^{k_I(x_i)}$. Here $k_I$ and $j$ are the numbers defined in Theorem \ref{t:Markov2}.
\end{itemize}
In other words, if the point $x_i$ is not non-expandable, we set $x_{i+1}=\widetilde{\mathcal R}(x_i)$, where $\widetilde{\mathcal R}$ is the expanding magnification map introduced at \eqref{eq:expR2}.

If the procedure never stops we can set $k(x)=\infty$, though it turns out that this possibility never occurs:

\begin{prop}
	\label{Expprocedure}
	Let $G\subset \Diff_+^2(\T)$ be a finitely generated subgroup whose action is minimal, satisfies property $\pstar$ and such that $\NE\neq\emptyset$. Let $x\in G\cdot \NE$. Then the following assertions hold true.
	\begin{enumerate}[\bf i.]
		\item There exists a finite integer $k=k(x)$, called the \emph{level} of $x$, such that the procedure stops after $k$ steps. 
		
		\item \label{ass:level}
		Let $\mathbf g_x$ denote the composition $g_{k}\, g_{k-1}\cdots g_1$ (locally equal to $\wR^k$).
		By construction $\mathbf g_x(x)=x_k$ belongs to $\NE$ and is the leftmost point of some $I^+_{j(x_k)}$. Define the interval $J_x^+=\mathbf g_x^{-1}(I^+_{j(x_k)})$, whose leftmost point is $x$. Then there exists a number $\kappa=\kappa(x)\geq k$ such that  $J_x^+$ is an atom of $\mI_{\kappa}$, the partition of level $\kappa$.

		\item \label{ass:dist} There exists a constant $C_0>0$ which does not depend on $x\in G\cdot\NE$ such that the distortion coefficient verifies the inequality $\varkappa(\mathbf g_x,J_x^+)\le~C_0$.
	\end{enumerate}
\end{prop}

\begin{proof}
	We observe that the expanding property of the magnification map $\wR$ imply that the derivatives of the compositions $g_{j}\, g_{j-1}\cdots g_1=\wR^j$ are always larger than $\lambda^j$. Since $x\in G\cdot\NE$, by Lemma \ref{characNE},  $(\wR^j)'(x)$ has to be bounded. This is possible if and only if the expansion procedure described above stops at some step $k$.
	
	That the intervals $J_x^+$ are atoms of the partition of some level $\kappa$ is clear from the definition of the two procedures.
	
	The map $\mathbf g_x$ is precisely the \emph{expansion map} $\wR^{k(x)}$ in restriction to $J_x^+$, in the sense of \cite[Definition~7]{FK2012_C_eng}. Thus, the third assertion follows from \cite[Proposition 2]{FK2012_C_eng} and because the size of the intervals $g_j\cdots g_1(J_x^+)=\wR^j(J_x^+)$ is uniformly bounded from below.
\end{proof}

\begin{lem}
	\label{distortionlem}
	With hypotheses and notations as in Proposition~\ref{Expprocedure}, the following assertions hold true.
	\begin{enumerate}[\bf i.]
		\item The family $(J^+_x)_{k(x)=k}$ consists of disjoint intervals.
		\item There exists a constant $C>1$ which does not depend on $x\in G\cdot \NE$ such that
		\[\frac{C^{-1}}{|J^+_x|}\leq \mathbf g_x'(x)\leq \frac{C}{|J^+_x|}.\]
	\end{enumerate}
\end{lem}

\begin{proof}
	By Proposition \ref{Expprocedure}.\ref{ass:level}, each interval $J_x^+$ is an atom of some partition of level $\kappa(x)$. This implies that two different intervals $J_x^+$ either are disjoint, or one is contained into the other.
	
	Assume for example that $J_x^+$ contains $J_y^+$ for some $x,y\in G\cdot \NE$. Then we claim that $k(x)<k(y)$. Indeed, the maps $g_i$ defined by the expansion procedure of $x$ and $y$ must coincide at least before the procedure stops for $x$. It stops for $x$ when $i=k$, and $x=x_k$. Then $\mathbf g_x(y)=y_k$ lies strictly inside $I^+_{j(x_k)}$, which contains no non-expandable point. Hence, the expansion procedure of $y$ must continue after the $k$-th step, and we have $k(x)<k(y)$ as desired.
	
	The second assertion directly follows from Proposition \ref{Expprocedure}.\ref{ass:dist}.
\end{proof}

In the final part of the proof of Theorem~\ref{t:ends_germs}, we will also need a second important result from~\cite{FK2012_C_eng}, which is presented  as a ``Structure Theorem''. It says that elements of $G$, upon magnification, are constructed from finitely many bricks.

\begin{thm}[Structure Theorem \cite{FK2012_C_eng}]\label{t:thompsonlike}
	Let $G\subset \Diff_+^2(\T)$ be a finitely generated subgroup whose action is minimal, with property $\pstar$ and such that $\NE\neq \emptyset$. Let $\Delta_0$ and $\mathcal{I}$ be the finite set and partition of $\T\setminus \Delta_0$ given by Theorem~\ref{t:Markov}, with the associated expanding maps $g_I$'s. There exists a finite number of intervals $L_1,\ldots L_N,L_1',\ldots L'_N\subset \T$ and finitely many elements $h_1,\ldots, h_N\in G$ defining diffeomorphisms $h_i:L_i\to L_i'$, such that any element $g\in G$ admits the following representation:
	\begin{enumerate}[\bf i.]
		\item there exist a finite subset $\Sigma_g\subset \T$ containing $\Delta_0\cup g^{-1}(\Delta_0)$, and a partition of $\T\setminus \Sigma_g$ into intervals $J_1,\ldots, J_q$ (which depends on $g$);
		\item for any $p=1,\ldots,q$ there exist intervals $L_{i_p}$, $L'_{i_p}$ in the expansion sequences of the intervals $J_p$ and $g(J_p)$ respectively. In other words for some $n_p,n_p'$ one has
		\[
		\cR^{n_p}(J_p)=L_{i_p},\quad \cR^{n_p'}(g(J_p))=L_{i_p}';
		\] 
		\item The map $g$ equals $h_{i_p}$ under magnification:
		\begin{equation}\label{eq:magnification}
		\cR^{n_p'}g\vert_{J_p}=h_{i_p}\cR^{n_p}\vert_{J_p}.
		\end{equation}
	\end{enumerate}
	Moreover, the finite subset $\Sigma_g$ and the partition $J_1,\ldots, J_q$ can be chosen to be the same for any finite subset of $G$.
\end{thm}

\begin{rem}\label{rem:thompsonlike}
	In the original statement in \cite{FK2012_C_eng} it is not specified that the maps $h_i:L_i\to L_i'$ are the restrictions of elements in $G$, however the elements $h_i$ are given by \cite[Lemma 5]{FK2012_C_eng}, where they appear as so.
\end{rem}

%% file: Section4.tex
\subsection{Preliminaries}\label{s:stallings}

\paragraph{A previous result --} Virtually free groups are the typical examples of groups with infinitely many ends. In \cite{DKN2014} Deroin, Kleptsyn and Navas succeeded in showing that virtually free groups have property $\pstar$:

\begin{thm}[Deroin, Kleptsyn, Navas]\label{thm:DKNfree}
	Let $G\subset \Diff_+^{\omega}(\T)$ be a virtually free subgroup acting minimally on the circle. Then $G$ has property $\pstar$.
\end{thm}
Hence, Theorem~\ref{mthm:1} extends the main result of \cite{DKN2014}. In fact, the proof of Theorem~\ref{mthm:1} relies on an interplay between the proof of Theorem~\ref{thm:DKNfree} and Stallings' theorem, following ideas of Hector and Ghys \cite{euler} that we sketch in \S~\ref{s:stallings}.

\paragraph{Stallings' theorem and virtually free groups --}

We will use what we know about the action of $G$ to restrict the possible Stallings' decompositions of a group $G$ acting by real-analytic diffeomorphisms of the circle and admitting an exceptional minimal set. This idea  can be traced back to Hector (and Ghys) \cite{euler}.
As a first illustrative example, let us sketch an argument by Hector under the additional assumption of no torsion \cite[Proposition~4.1]{euler}.

\begin{thm}[Hector]
Let $G\subset \Diff_+^\omega(\T)$ be a finitely generated, torsion-free subgroup acting with an exceptional minimal set. Then $G$ is free.
\end{thm}

\begin{proof}
Duminy's theorem (Theorem~\ref{Duminy}) implies that $G$ has infinitely many ends, so has a Stallings' decomposition. Since the group is torsion free, the Stallings' decomposition must be a free product $G=G_1*G_2$ of finitely generated groups $G_1$ and $G_2$. Now, neither factor acts minimally (otherwise $G$ does). If one of the factors acts with an exceptional minimal set, then we can expand the free product $G_1*G_2$ until the moment we get $G=H_1*\ldots*H_n$ with every $H_i$ acting with some periodic orbit. Indeed, this procedure has to stop in a finite number of steps, for the rank (the least number of generators) of the factors is less than the rank of the group (this follows from a classical formula of Grushko; see  \cite{combinatorial}). Now we use that the action is by real-analytic diffeomorphisms. As the action has an exceptional minimal set, the group must be locally discrete (easy consequence of Proposition~\ref{p:localvf}).
Corollary~\ref{cor:finite-orbits} implies that the subgroups $H_i$'s must be either cyclic or semi-direct products of an infinite cyclic group with a finite group. Since the group $G$ is torsion-free, the only possibility is that every $H_i$ is infinite cyclic. Thus, $G$ is free, as claimed.
\end{proof}

In \cite{euler}, Ghys proved that the same holds for any group $G\subset \Diff_+^\omega(\T)$ acting with an exceptional minimal set:

\begin{thm}[Ghys]\label{thm:Ghys}
	Let $G\subset \Diff_+^{\omega}(\T)$ be a finitely generated subgroup acting  with an exceptional minimal set. Then $G$ is virtually free.
\end{thm}

We can sketch the proof of Ghys' Theorem~\ref{thm:Ghys} under the assumption that the group $G$ acting on the circle with an exceptional minimal set verifies a certain hypothesis, called \emph{Dunwoody's accessibility}. Finitely generated groups with $0$ or $1$ ends are accessible (by definition) and, in general, accessible groups are all those groups that can be obtained as amalgamated products or HNN extensions of accessible groups over finite groups. Dunwoody proved that finitely presented groups are accessible \cite{dunwoody1}, but there are 
finitely generated groups that are not accessible \cite{dunwoody2}. 

\begin{thm}[Ghys]
Let $G\subset \Diff_+^\omega(\T)$ be a finitely generated, accessible subgroup acting with an exceptional minimal set. Then $G$ is virtually free.
\end{thm}

\begin{proof}
Starting with a Stallings' decomposition of $G$, say $G=G_1*_Z G_2$ or $H*_Z$, we argue as before that the groups $G_1$ and $G_2$ or $H$ cannot act minimally. If the action of one of these groups has a finite orbit, then the group is virtually cyclic (Corollary~\ref{cor:finite-orbits}). Otherwise, it acts with an exceptional minimal set and Duminy's Theorem~\ref{Duminy} applies, so we can take a Stallings' decomposition and keep repeating this argument. Accessibility guarantees that this process stops after a finite number of steps, so the group $G$ is obtained by a (finite) combination of amalgamated products and HNN extensions over finite groups, with virtually cyclic groups as basic pieces. Finally, these groups are \emph{virtually free}, as one deduces from the following classical theorem \cite{virtually_free}:

\begin{thm}[Karrass, Pietrowski, Solitar]\label{t:KPS}
Let $G_1$, $G_2$ and $H$ denote finitely generated, virtually free groups and $Z$ a finite group. Then the amalgamated product $G_1*_Z G_2$ and the HNN extension $H*_Z$ are also virtually free.
\end{thm}
\end{proof}

\subsection{Proof of Theorem~\ref{mthm:1}: Outline}
\label{ssc:minimal}

The rest of this section is dedicated to the proof of Theorem~\ref{mthm:1}.

Let $G\subset \Diff_+^{\omega}(\T)$ be a locally discrete, finitely generated subgroup with infinitely many ends acting minimally on the circle. 
By Stallings' theorem, we know that either $G=G_1*_ZG_2$ or $G=H*_Z$, with $Z$ a finite 
group. For the proof of Theorem \ref{mthm:1}, we  analyse the factors appearing in Stallings' decompositions, as in the previous subsection. 
From now on, we shall assume that $G$ admits non-expandable points, otherwise $\pstar$ is trivially satisfied.

%%%%%%%%%%%%%%%%%%%%%%%%%%%%%%%%%%%%%%%%%%%%%

\paragraph{First (possible) case: No Stallings' factor acts minimally -- }

If such a factor has a finite orbit, then it is virtually cyclic by Corollary~\ref{cor:finite-orbits}. Otherwise, it acts with an exceptional minimal set, and Ghys' Theorem \ref{thm:Ghys} implies that it is virtually free. Therefore, $G$ is either an amalgamated product of virtually free groups over a finite group or an HNN extension of a virtually free group over a finite group. By the already mentioned theorem of Karrass, Pietrowski and Solitar (Theorem~\ref{t:KPS}), the group $G$ itself is virtually free. We deduce that the group satisfies $\pstar$ by Theorem~\ref{thm:DKNfree}.

%%%%%%%%%%%%%%%%%%%%%%%%%%%%%%%%%%%%%%%%%%%%%%%%

\paragraph{Second (impossible) case: At least one factor acts minimally --}

Under this assumption, we will prove that $G$ is non locally discrete borrowing one of the main arguments from \cite{DKN2014}. 
To do this, remark that it is enough to study the case where $G=G_1*_ZG_2$ is an amalgamated product, since any HNN extension $H*_Z$ contains copies of $H*_ZH$ as subgroups. Indeed, if we denote by $\sigma$ the \emph{stable letter} (that is, the element conjugating the two embedded copies of $Z$) in $H*_Z$, then $H$ and $\sigma H\sigma^{-1}$ generate a subgroup isomorphic to $H*_ZH$.

\medskip

Thus, from now on, we suppose that $G$ is an amalgamated product $G_1*_ZG_2$ over a finite group $Z$, and we assume that $G_1$ acts minimally. In particular $G_1$ is infinite, while $G_2$ can possibly be finite. For simplicity, we let $\mathcal{G} = \mathcal{G}_1\sqcup\mathcal{G}_2$ be a finite system of generators for $G$, with $\mathcal{G}_i$ generating $G_i$ and symmetric. We consider the length metric on the group $G$ associated with this generating system, and for every $n\in\N$ we let $B(n)$ be the ball of radius $n$ centred at the identity. 

Let us illustrate the main lines of the proof before getting involved in technicalities. This will be also the opportunity to introduce some further notation.

\begin{notation}\label{not:outer}
	Given a finite subset $E\subset G$, let $\rho(E)$ denote the \emph{outer radius} of $E$, that is, the minimal $n\in\N$ such that $E\subset B(n)$.
\end{notation}

\begin{notation}\label{notation}
We fix a non-expandable point $x_0\in \NE$, and for any finite set $E\subset G$, we let $x_E$ denote the closest point on the right of $x_0$ among the points in the image set $E \cdot x_0$ distinct from $x_0$ (such a point exists for any $E$ which is not contained in the stabilizer of $x_0$). This point corresponds to some $g_E\in E$, that is, $x_E=g_E(x_0)$. Besides,  $g_E$ is uniquely defined modulo right multiplication with an element in $\stab{G}{x_0}$. The length of the interval $J_E=[x_0,x_E]$ will be denoted by $\ell_E$.

In order to take account of the number of elements fixing $x_0$, and hence of possible overlaps of the intervals $g(J_E)$, for $g\in E$, we define
\[c_E=\max_{h\in E} \#\left ( E\cap h\,\stab{G}{x_0}\right ).\]
\end{notation}

As in \cite{DKN2014,FKone}, the proof is carried on in three different stages, which will be exposed separately in the next paragraphs, and then executed in the following subsections.

\paragraph{Step 1. --} The first and most important step (Proposition \ref{p:sufficient_estimate}) is to describe a \emph{sufficient condition} guaranteeing that for a prescribed sequence of finite subsets $E(n)\subset G$, setting $F(n)=E(n)^{-1}E(n)$, the elements $g_{F(n)}$ ``locally converge'' (in the $C^1$ topology) to the identity.  In concrete terms, letting
\begin{equation}\label{eq:notation}S_E=\sum_{g\in E}g'(x_0),\end{equation}
we will show that, in order to ensure the desired convergence, it is enough that
\begin{equation}\label{eq:sufficient_estimate}
\rho(E(n)) \frac{c_{E(n)}}{S_{E(n)}}=o(1)\quad\text{as }n\text{ goes to infinity}.
\end{equation}
Notice, however, that this criterion does not provide directly a contradiction to the hypothesis of local discreteness of $G$, since we are only able to show that $g_{F(n)}$ is closer and closer to $id$ \emph{when restricted} to (a complex extension of) an interval depending on $n$, which is unfortunately shrinking to $x_0$.

\begin{rem}
In the following, we will deal both with $C^0$ and $C^1$ local convergence. In fact, as the elements are real-analytic, the classical Cauchy estimates imply that the two notions are equivalent. The point is that for proving that the sequence of elements $g_{F(n)}$ converges $C^0$ to the identity, we first prove that the derivatives converge to $1$ and deduce from the control of the affine distortion that the elements converge $C^0$. 
\end{rem}
%%%%%%%%%%%

\paragraph{Step 2. --} We then show that it is very easy to find examples of sequences $\left (E(n)\right )_{n\in\N}$ which satisfy the criterion above, even in a very strong way. For this, we use three key facts:
\begin{enumerate}
\item $G_1$ acts minimally, hence taking a sufficiently large integer $n\in\N$, the sum $\sum_{g\in B_1(n)}g'(x)$ can be made as large as we want, and this uniformly on $x\in \T$ (Proposition \ref{p:minimal-bound}). Here, $B_1(n)$ is the ball of radius $n$ in $G_1$ with respect to the generating set $\mathcal G_1$ .
\item Using the tree-like structure and the normal form in amalgamated products, we  move from a $G_1$-slice in $G$ to another. Doing this, 
we increase the lower bound for $S_{E(n)}$ in an exponential way (Proposition~\ref{p:exponential-bound}). As a consequence, there exists $a>1$, such that
\[S_{E(n)}\ge a^{\rho(E(n))}.\]
\item 
At the same time, studying how the stabilizer $\stab{G}{x_0}$ sits inside $G$, we prove that $c_{E(n)}$ has at most linear growth in terms of $\rho(E(n))$ (Proposition~\ref{l:cE}). For this, we use Proposition~\ref{l:distorted} about distorted elements in amalgamated products.
This estimate turns to be fine enough: since $S_{E(n)}$ grows exponentially, the quantity in \eqref{eq:sufficient_estimate} decays exponentially.
\end{enumerate} 

\paragraph{Step 3. --} The key idea here relies on a result of Ghys \cite[Proposition 2.7]{proches} (that can be traced back to Gromov \cite[\S~7.11.E1]{dambra-gromov}) about groups of analytic local  diffeomorphisms defined on the complex neighbourhood $U^{\C}_r(x_0)$ of radius $r>0$ of $x_0\in\C$:
\begin{prop}\label{p:ghys}
For any $r>0$ there exists $\eps_0>0$ with the following property: Assume that the complex analytic local diffeomorphisms $f_1$, $f_2:U_r^{\C}(x_0)\to \C$ are $\eps_0$-close (in the $C^0$ topology) to the identity, and let the sequence $f_k$ be defined by the recurrence relation
\[f_{k+2}=[f_k,f_{k+1}],\quad k=1,2,3,\ldots\]
Then all the maps $f_k$ are defined on the disc $U_{r/2}^{\C}(x_0)$ of radius $1/2$, and $f_k$ converges to the identity in the $C^{1}$ topology on $U_{r/2}^{\C}(x_0)$.
\end{prop}
The main point of this proposition is that if the sequence of iterated commutators $(f_k)_{k\in\N}$ is not eventually trivial, then $f_1$ and $f_2$ generate a group which is non locally discrete.

From the previous steps, it is not difficult to find elements $f_1$, $f_2$ of the form $g_{E(m)}$ which are very close to the identity on some neighbourhood of $x_0$, but we must exhibit explicit $f_1$ and $f_2$ for which we are able to show that the sequence of iterated commutators $f_k$ is not eventually the identity. This is certainly the case if $f_1$ and $f_2$ generate a free group: we prove in Proposition~\ref{p:step3} that it is possible to find such two elements, relying on Proposition~\ref{l:ping-pong} which allows a ping-pong argument.

\paragraph{Summary of the proof of Theorem~\ref{mthm:1} --}
We start with $G\subset \Diff_+^\omega(\T)$ a locally discrete, finitely generated subgroup with infinitely many ends, and a point $x_0\in\NE$.
% We suppose by way of contradiction that no nontrivial element in $G$ fixes $x_0$.

By Stallings' theorem, $G$ has  a Stallings' decomposition. Without loss of generality, we may suppose $G=G_1*_ZG_2$. We have seen how to rule out the case when no factor acts minimally. 
Therefore we consider the case when $G_1$ acts minimally. Under this assumption, Proposition~\ref{p:step3} ensures the existence of elements $f_1,f_2\in G$ such that:
\begin{enumerate}
	\item they are both $\eps_0$-close to the identity in the $C^0$ topology, when restricted to a certain complex neighbourhood of $x_0$,
	\item no iterated commutator $f_{k+2}=[f_k,f_{k+1}]$ is trivial.
\end{enumerate}
Then we apply Proposition~\ref{p:ghys} and get that the sequence $f_k$ converges to the identity in the $C^1$ topology when restricted to a fixed neighbourhood of $x_0$. This contradicts the hypothesis that the group $G$ is locally discrete.

%%%%%%%%%%%%%%%%%%%%%%%%%%%%%%%%%%%

\subsection{Step 1: Getting close to the identity}
\label{ssc:close}

Here we review the argument given in \cite[\S~3.2]{DKN2014} and \cite[\S~2.5]{FKone}, which explains how to find elements which are close to the identity in a neighbourhood of a non-expandable point.
The result is stated in a general form,  because of the algebraic issues that we have to overcome in \S~\ref{ssc:commutators}. The main result of this section is a variation of \cite[Lemma 3.15]{DKN2014}. For its statement and proof, we shall make use of Notations~\ref{not:outer}, \ref{notation}, and \eqref{eq:notation}.

\begin{prop}\label{p:sufficient_estimate}
Let $\left (E(n)\right )_{n\in\N}$ be a sequence of subsets of $G$ containing the identity. If
\[
\rho(E(n))\frac{c_{E(n)}}{S_{E(n)}}=o(1)\quad\text{as }n\text{ goes to infinity},
\]
then the sequence $g_{F(n)}$ for $F(n)=E(n)^{-1}E(n)$ converges to the identity in the $C^1$ topology on a complex disc of radius $o(1/\rho(E(n)))$ around $x_0$. More precisely, considering $r_n=o\left (1/\rho(E(n))\right )$ such that
\[\frac{c_{E(n)}}{S_{E(n)}}=o\left ({r_n}\right )\quad\text{as }n\text{ goes to infinity},\]
the (affinely) rescaled sequence 
\[\widetilde g_{F(n)}(t)=\frac{g_{F(n)}(x_0+r_n\,t)-x_0}{r_n}\]
converges to the identity in $C^0(U_1^{\C}(0))$ (and equivalently, in $C^1(U_1^{\C}(0))$).
%\marginpar{Why do you make distinction between the $C^0$ and $C^1$ issues ?}
\end{prop}

We avoid the (somehow technical) details of the proof and prefer to explain the relevant ideas, which mostly rely on the classical technique of \emph{control of affine distortion} (see \cite[Lemma 3.7]{DKN2014}). 

Recall from \S~\ref{sc:Markovp} that the distortion coefficient $\varkappa(g;I)$ measures the failure of some diffeomorphism $g:I\to g(I)$ to be an affine map. As it behaves sub-additively with respect to composition, the distortion coefficient of $g=g_n\cdots g_1$ is usually estimated by the sum $S=\sum_{i=0}^{n-1}(g_i\cdots g_1)'(x_0)$ introduced in \eqref{eq:S}.
The key observation in our framework (and originally of \cite{DKN2014,FKone}) is that at non-expandable points $x_0\in\NE$, we obviously have $S\le n$ for $g\in B(n)$. Therefore Proposition~\ref{l:schwartz} implies that, for a very large $n$, in a neighbourhood of size $r\ll 1/n$ about $x_0$, the maps in $B(n)$ are almost affine. In particular, the element $g_{F(n)}$ (resp.~$\widetilde g_{F(n)}$) is almost affine on a neighbourhood of radius $r_n=o\left (1/\rho(E(n))\right )$ (resp.~$1$) about $x_0$ (resp.~$0$). 

To see that the derivative of $g_{F(n)}$ (and $\widetilde g_{F(n)}$) is close to $1$, we consider the inverse map $g_{F(n)}^{-1}$, which satisfies
\[(g_{F(n)}^{-1})'(x_0)\le 1\quad\text{and}\quad (g_{F(n)}^{-1})'(x_{F(n)})=\frac{1}{g_{F(n)}'(x_0)}\ge 1.\]
The point $x_{F(n)}$ is at distance $\ell_{F(n)}$ from $x_0$ (Notation~\ref{notation}). If $\ell_{F(n)}=o(r_n)$, then the control on the affine distortion guarantees that the derivative of $g_{F(n)}^{-1}$, and hence of $g_{F(n)}$, is close to $1$ on the neighbourhood of radius $r_n$.
Indeed, for every $z \in U_r(x_0)$ one has
\[
\log (g_{F(n)}^{-1})'(z)=\log \frac{(g_{F(n)}^{-1})'(z)}{(g_{F(n)}^{-1})'(x_0)}+\log(g_{F(n)}^{-1})'(x_0)\le \sup_{x,y\in U_r(x_0)} \log \frac{(g_{F(n)}^{-1})'(x)}{(g_{F(n)}^{-1})'(y)}
\]
and
\[
\log (g_{F(n)}^{-1})'(z)=\log \frac{(g_{F(n)}^{-1})'(z)}{(g_{F(n)}^{-1})'(x_{F(n)})}+\log(g_{F(n)}^{-1})'(x_{F(n)})\ge \inf_{x,y\in U_r(x_0)} \log \frac{(g_{F(n)}^{-1})'(x)}{(g_{F(n)}^{-1})'(y)}.
\]
Thus $\sup_{U_r(x_0)}|\log (g^{-1}_{F(n)})'|\le \varkappa (g^{-1}_{F(n)},U_r^{\mathbb C}(x_0))$.

The asymptotic condition $\ell_{F(n)}=o(r_n)$  assures that also the map $\widetilde g_{F(n)}$ is almost the identity, since $\widetilde g_{F(n)}(0)=\ell_{F(n)}/r_n$. 
Therefore, we get the desired conclusion from the following key estimate:

\begin{lem}\label{l:ell_F}
Let $E\subset G$ be a finite subset of $G$ containing the identity and define $F=E^{-1}E$. Then the length $\ell_F$ verifies
\[\ell_F\le  C\,\frac{c_E}{
S_E},\]
where the constant $C>0$ does not depend on $E$.
\end{lem}

\begin{proof}[Sketch of the proof]
We observe that any two intervals $g(J_F)$ and $h(J_F)$, for $g,h\in E$, are either disjoint or have the same leftmost points, with equality if and only if $g\in h\,\stab{G}{x_0}$. Indeed, suppose that the left endpoint of $h(J_F)$ belongs to $g(J_F)$. Then $h^{-1}g(x_0)$ is closer than $x_F$ to $x_0$ on the right, and since $h^{-1}g\in E^{-1}E= F$, we must have $h^{-1}g(x_0)=x_0$, that is, $g\in h\,\stab{G}{x_0}$.

Therefore, the union of the intervals $g(J_F)$, for $g\in E$, covers the circle $\T$ at most $c_E$ times. With the (quite subtle) argument in \cite[Lemma 3.15]{DKN2014} relying on the control of the affine distortion, we find the inequality
\[\ell_F\le  C\,\frac{c_E}{
S_E},\]
as desired.
\end{proof}

%%%%%%%%%%%%%%%%%%%%%%%%%%

\subsection{Step 2: An exponential lower bound for the sum of derivatives}
\label{ssc:exponential}

Using the normal form of elements in an amalgamated product (Lemma~\ref{l:normalform}), we will use a tool developed in \cite{DKN2014} for free groups. The aim of this step is to find a sequence of subsets $A(n)$ with an exponential lower bound for the sum $S_{A(n)}$ as defined in \eqref{eq:notation}. We actually prove more: the exponential lower bound for the sum of the derivatives holds \emph{at every point $x\in\T$}. This turns out to be very useful, since it gives exponential lower bounds for the sum $S_{\psi A(n)\psi^{-1}}$ associated to each conjugate set $\psi A(n) \psi^{-1}$ of $A(n)$, 
where $\psi\in G$. 

We start by noticing that, since $G_1$ acts minimally, the proof of \cite[Proposition 2.5]{DKN2014} combined with a compactness type argument 
immediately yields:

\begin{prop}\label{p:minimal-bound}
%Let $G_1\subset \Diff^{2}_+(\T)$ be a finitely generated group whose action on $\T$ is minimal.
For every $M>0$, there exists $R_1\in\N$ such that for every $x\in S^1$ we have
\begin{equation}\label{eq:sumB1}
\sum_{g\in B_1(R_1)}g'(x)>M,
\end{equation}
where $B_1(R_1)$ is the ball of radius $R_1$ in $G_1$.
\end{prop}

As in Lemma~\ref{l:normalform}, we denote by $T_i$ a transversal set of cosets for ${}_Z\backslash{}^{G_i}$, $i=1,2$.  Using the previous proposition, we next prove:

\begin{lem}\label{p:sumB1-trans}
For every $M'>0$, there exists $R_1'\in \N$ such that
\[\sum_{t\in B_1^{\times}(R'_1)\cap T_1}t'(x) > M',\]
where $B_1^{\times}(R_1)$ is the ball $B_1(R_1)$ in $G_1$, but with the identity excluded.
\end{lem}

\begin{proof}
Let $c_0= |Z|\cdot \sup_{\gamma\in Z}\|\gamma'\|_0$. Take $M > c_0 (1 + M')$ and fix 
the associated $R_1$ given by Proposition~\ref{p:minimal-bound}. Decomposing the sum~\eqref{eq:sumB1} using the transversal set, we write
\begin{equation}\label{eq:sumB1-trans}
\sum_{g\in B_1(R_1)}g'(x) =\sum_{\gamma\in Z,\, t\in T_1\,:\,\gamma t\in B_1(R_1)}(\gamma t)'(x).
\end{equation}
Observe that, by the triangle inequality, one has the inclusion
\[
\{g=\gamma t\mid \gamma\in Z,\, t\in T_1\text{ such that }\gamma t\in B_1(R_1)\}\subset \{g=\gamma t\mid \gamma\in Z,\, t\in B_1(R_1+\rho(Z))\cap T_1\},
\]
(recall from Notation~\ref{not:outer} that $\rho(Z)$ denotes the outer radius of the set $Z$). Thus the sum~\eqref{eq:sumB1-trans} 
is bounded from above by the same sum but over the larger set:
\[
\sum_{g\in B_1(R_1)}g'(x)\le \sum_{\gamma\in Z}\left (\sum_{t\in B_1(R_1+\rho(Z))\cap T_1}(\gamma t)'(x)\right ).
\]
Next, using the chain rule and taking care of the identity element, we obtain:
\begin{align*}
M\le \sum_{g\in B_1(R_1)}g'(x) \le\,&\sum_{\gamma\in Z}\left (\sum_{t\in B_1(R_1+\rho(Z))\cap T_1}\gamma'(t(x)) t'(x)\right )\\
\le\,& |Z|\cdot \sup_{\gamma\in Z}\|\gamma'\|_0\,\left (1+\sum_{t\in B_1^{\times}(R_1+\rho(Z))\cap T_1}t'(x)\right )\\
=\,&c_0\left (1 + \sum_{t\in B_1^{\times}(R_1+\rho(Z))\cap T_1}t'(x)\right ).
\end{align*}
Setting $R_1'=R_1+\rho(Z)$, this closes the proof.
\end{proof}

It is easy now to construct a sequence of sets $A(n)$ 
with an exponential lower bound for the sum of the derivatives. Indeed, it is enough to fix an element $\sigma\in T_2 \setminus \{id\}$, 
and define the product set
\begin{equation}\label{eq:A(n)}
A(n)=\sigma \left (B_1^{\times}(R_1')\cap T_1\right )\cdots \sigma \left (B_1^{\times}(R_1')\cap T_1\right ),\end{equation}
where the product of $\sigma \left (B_1^{\times}(R_1')\cap T_1\right )$ is repeated $n$ times and $R_1'$ is appropriately chosen.

\begin{lem}\label{l:sumA}
There exists $a>1$ such that for all $n\in \N$ and every $x\in\T$,
\[\sum_{g\in A(n)}g'(x)\ge a^{\rho(A(n))}.\]
\end{lem}

\begin{proof}
Take $M'>(\inf \sigma')^{-1}$ and the associated $R_1'$ from Lemma~\ref{p:sumB1-trans}.
Let us consider all the products $\sigma t_1$, with $t_1\in B_1^{\times}(R_1')\cap T_1$. 
We define $\overline{M}=M'\cdot \inf \sigma'$, which is larger than $1$ by assumption. With this choice, we have
\begin{align*}
\sum_{g\in A(n)}g'(x)\,&= \sum_{t_1,\ldots,t_n\in B_1^{\times}(R_1')\cap T_1}(\sigma t_n\cdots \sigma t_1)'(x) \\
\,&\ge \overline{M}\cdot \sum_{t_1,\ldots,t_{n-1}\in B_1^{\times}(R_1')\cap T_1}(\sigma t_{n-1}\cdots \sigma t_1)'(x).
\end{align*}
Proceeding inductively, we get $\sum_{g\in A(n)}g'(x)\ge \overline{M}^n$. We claim that it is enough to set 
$a=\overline{M}^{1/\left (R_1'+d_{\mathcal G}(id,\sigma)\right )}$.
This is because the inequality $\rho(A(n))\le n(R_1'+d_{\mathcal G}(id,\sigma))$ hold for every $n\in\N$: by definition, the subset $A(n)$ is contained in the ball of radius $n(R_1'+d_{\mathcal G}(id,\sigma))$ in $G$, so the outer radius of $A(n)$ grows at most linearly on $n$.
\end{proof}

Finally, we have:

\begin{prop}\label{p:exponential-bound}
For any $\psi\in G$, there exists a constant $C(\psi)$ such that the sum $S_{\psi A(n)\psi^{-1}}$ defined as in \eqref{eq:notation} satisfies
\[S_{\psi A(n)\psi^{-1}}\ge C(\psi)\, a^{\rho\left (\psi A(n)\psi^{-1}\right )}.\]
\end{prop}
\begin{proof}
For $\psi\in G$, let $\lambda=\|\psi \|$ denote its length in the generating system $\mathcal G$. Then, by the triangle inequality, for any $n\in\N$, we have
\[\rho\left (\psi A(n)\psi^{-1}\right )\le \rho(A(n))+2\lambda.\]
We can easily compare the sum $S_{\psi A(n)\psi^{-1}}$ with the sum of the derivatives of elements in $A(n)$:
\begin{align*}
S_{\psi A(n)\psi^{-1}}\,&=\sum_{g\in \psi A(n)\psi^{-1}}g'(x_0)\\
\,&=\sum_{h\in A(n)}(\psi h\psi^{-1})'(x_0)\\
\,&\ge \inf \psi'\cdot\sum_{h\in A(n)}h'(\psi^{-1}(x_0)) \cdot (\psi^{-1})'(x_0).
\end{align*}
Hence, by Lemma~\ref{l:sumA}, we have the inequality
\[S_{\psi A(n)\psi^{-1}}\ge \left (\inf \psi'\cdot (\psi^{-1})'(x_0)\right )a^{\rho(A(n))}.\]
The proof is finished by letting $C(\psi)=a^{-2\lambda}\,(\psi^{-1})'(x_0)\inf \psi'$.
\end{proof}

Now, let us set $E(n)=\{id\}\cup A(n)$ and 
\begin{equation}
\label{eq:defFn}
F(n)=E(n)^{-1}E(n).
\end{equation}
In order to close the second step, it remains to estimate the quantity $c_{\psi E(n)\psi^{-1}}$ (Notation~\ref{notation}), which gives an upper bound for the number of overlaps of the intervals $g(J_{\psi F(n)\psi^{-1}})$, for $g\in \psi E(n)\psi^{-1}$.

Let us first rule out a particular (but important) case.

\begin{lem}\label{l:cEdistorted}
	Assume that the subgroup $\stab{G}{x_0}$ is cyclically generated by a distorted element $\fhi\in G$. Then the quantity
	\[
	c_{E(n)}=\max_{h\in E(n)}\#\left (E(n)\cap h\,\stab{G}{x_0}\right )
	\]
	is uniformly bounded: there exists $L>0$ such that $c_E(n)<L$ for every $n\in \N$.
\end{lem}

\begin{proof}
	In the proof, we consider the action of $G=G_1*_ZG_2$ on the Bass-Serre tree $X$. Notations  are borrowed from \S~\ref{ssc:bass-serre}.
	
	As $\fhi$ is distorted, Lemma~\ref{l:distorted} implies that $\fhi$ belongs to a conjugate factor $g^{-1}G_ig$.
	Without loss of generality, we can suppose $\fhi\in g^{-1}G_1g$, for some $g\in G$. Indeed, only a subgroup acting minimally can contain distorted elements: otherwise the subgroup would be virtually free (Corollary~\ref{cor:finite-orbits} and Theorem~\ref{thm:Ghys}), and virtually free groups do not have distorted elements of infinite order (Lemma~\ref{l:dist_free}).
	
	After Lemma~\ref{l:find_axis}, every element belonging to the subset $A(n)$ acts as a hyperbolic isometry with translation length $2n$, whose translation axis contains the common segment $(G_1\sigma^{-1},G_2,G_1)$.
	On the other hand, every element of $\stab{G}{x_0}$ acts as an elliptic isometry, fixing the vertex $G_1g\in X$.
	This already gives $E(n)\cap \stab{G}{x_0}=\{id\}$, so in the following we fix $h\in  A(n)$ and we look for a uniform upper bound for the quantity 
	$\#\left (E(n)\cap h\,\stab{G}{x_0}\right )$.
	Notice that $E(n)\cap h\,\stab{G}{x_0}= A(n)\cap h\,\stab{G}{x_0}$, as $id\in h\,\stab{G}{x_0}$ would imply $h\in\stab{G}{x_0}$.
	Therefore we want to prove that the cardinality of the set
	\begin{equation}\label{eq:setPh}
	P_h=\left\{\ell \in \Z\mid  h\fhi^\ell\in A(n) \right\}
	\end{equation}
	is uniformly bounded on $h\in A(n),n\in\N$.
	
	Let us assume that there exists $\ell\in P_h$ and write $\widetilde h= h\, \fhi^\ell\in A(n)$. As already observed, the translations axes $X(h),X(\widetilde h)$ contain a common segment, so they intersect.

		Consider the point $G_1g$ in the Bass-Serre tree, which is fixed by $\fhi\in g^{-1}G_1g$. Because of the equality $ h^{-1} \widetilde h=\fhi^\ell$, the images $G_1gh^{-1}$ and $G_1g\widetilde h^{-1}$ are the same (recall that the action on the Bass-Serre tree is naturally a right action)
		
		\begin{claim}\label{cl:vertex}
			The vertex $G_1g$ belongs to the intersection $X(h)\cap X(\widetilde{h})$.
		\end{claim}
		
		\begin{proof}[Proof of Claim]
		Applying the formula \eqref{eq:dist_image} for the distances of the images, we find
		\begin{align*}
		d(G_1g,G_1g\widetilde h^{-1})=\,&2n+2d(G_1g,X(\widetilde h)),\\
		d(G_1g,G_1g h^{-1})=\,&2n+2d(G_1g,X(h))
		\end{align*}
		and by equality of the images, we must have $d(G_1g,X(\widetilde h))=d(G_1g,X(h))$.
		
		Let us assume by way of contradiction that this distance is not zero. Since the two axes intersect, the geodesic segments from $G_1g$ to $X(\widetilde h)$ and $X(h)$ respectively, must be the same: indeed, if this was not the case, these segments would give a nontrivial geodesic path connecting $X(h)$ and $X(\widetilde h)$; then the union of such a path and the intersection of the axes would give a nontrivial loop in the tree. Call this geodesic segment $\gamma$, which goes from $G_1g$ to the intersection $X(h)\cap X(\widetilde h)$.
		
		We are assuming $d(G_1g,X(\widetilde h))>0$, so this segment $\gamma$ has more than one vertex. We then have that the product $ h^{-1} \widetilde h =\fhi^\ell $ fixes it: indeed, we repeat the previous argument and get that the geodesic paths from $G_1gh^{-1}=G_1g\widetilde h^{-1}$ to $X(h)$ and $X(\widetilde h)$ coincide, and this common path is exactly the image $\gamma.h^{-1}=\gamma.\widetilde h^{-1}$.
		We deduce that $\fhi^\ell$ belongs to a conjugate of the edge group~$Z$. However $\fhi$ has infinite order, a contradiction.
	\end{proof}

	\begin{claim}
		The intersection $X(h)\cap X(\widetilde h)$ coincides with the segment from $G_1gh^{-1}=G_1g\widetilde h^{-1}$ to $G_1g$.
	\end{claim}
	
	\begin{proof}[Proof of Claim]
		By Claim~\ref{cl:vertex} the vertex $G_1g$ belongs to both axes $X(h)$ and $X(\widetilde h)$. Since the images $G_1gh^{-1}$ and $G_1g\widetilde h^{-1}$ are the same, we have that the intersection $X(h)\cap X(\widetilde h)$ contains the whole segment of length~$2n$ between $G_1g$ and its image $G_1gh^{-1}=G_1g\widetilde h^{-1}$.
		
		On the one hand $ h^{-1}\widetilde h =\fhi^\ell$ fixes exactly one point, while on the other hand the two elements act like translations by $2n$ on their own translation axes. If the intersection contains more than $2n$ points, then we get that the product $ h^{-1}\widetilde h =\fhi^\ell$ fixes at least two points, and thus $\fhi$ has finite order. This gives a contradiction. 
	\end{proof}
	
	\begin{claim}\label{cl:bound}
		There exists a finite set $P$ such that if $h,\widetilde h\in A(n)$ and $\ell\in \Z$, are such that $\widetilde h=h\fhi^{\ell}$, then $\fhi^\ell \in P$.
	\end{claim}
	
	\begin{proof}[Proof of Claim]
		The elements $h$ and $\widetilde h$ are in $A(n)$: by its definition in \eqref{eq:A(n)}, there exist $t_i$'s and $\widetilde t_i$'s in $B_1^{\times}(R_1')\cap T_1$, $i=1,\ldots n$, such that
		\begin{equation}\label{eq:form_A(n)}
		h=\sigma t_n\cdots \sigma t_1,\quad \widetilde h=\sigma \widetilde t_n\cdots \sigma \widetilde t_1.
		\end{equation}
		Notice that the vertex $G_1$ belongs to the intersection $X(h)\cap X(\widetilde h)$ (as we said at the beginning, it contains the path $(G_1\sigma^{-1},G_2,G_1)$). The situation is cartooned in Figure~\ref{fig:cE}.
		\begin{figure}
			\[\includegraphics[scale=1]{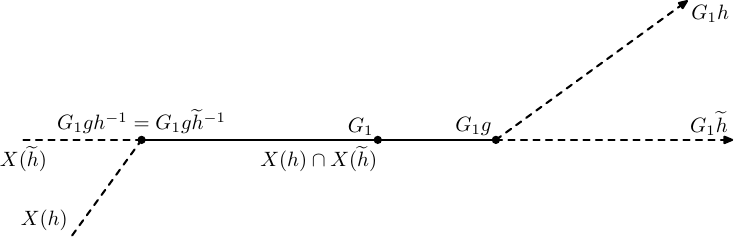}\]
			\caption{The translation axes of the two elements in the proof of Proposition~\ref{l:cE}}\label{fig:cE}
		\end{figure}
		Given the explicit expression \eqref{eq:form_axis} for the translation axes of elements in $A(n)$, we deduce that there exists $0\le k\le n$ such that
		\[
		G_1g=G_1\sigma  t_k\cdots \sigma  t_1\quad \text{or}\quad G_1g=G_1\sigma^{-1}  t^{-1}_k\cdots \sigma^{-1}  t^{-1}_n\sigma^{-1}
		\]
		(with abuse of notation, the case $k=0$ corresponds to $G_1g=G_1$).
		Let us write temporarily $\widetilde g=\sigma  t_k\cdots \sigma  t_1$ (resp.~$\widetilde{g}=\sigma^{-1}  t^{-1}_k\cdots \sigma^{-1}  t^{-1}_n\sigma^{-1}$); since $G_1g=G_1\widetilde{g}$ there exists $f\in G_1$ so that $g=f\widetilde{g}$ and thus $\widetilde{g}^{-1}G_1\widetilde{g}=g^{-1}f^{-1}G_1fg=g^{-1}G_1g$. Therefore we can suppose that $g$ is the initial (resp.~the inverse of the final) part of $h$, that is $g=\widetilde g=\sigma  t_k\cdots \sigma  t_1$ (resp.~$g=\sigma^{-1}  t^{-1}_k\cdots \sigma^{-1}  t^{-1}_n\sigma^{-1}$). We also write $\fhi^\ell=g^{-1}x^\ell g$, with $x\in G_1$.
		
		Assume first $g=\widetilde g=\sigma  t_k\cdots \sigma  t_1$. The product $ h\fhi^\ell= h g^{-1}x^\ell g$ has therefore a ``cyclic simplification'':
		\begin{equation}
		\label{eq:cyclic}
		 h\fhi^\ell=\sigma  t_n\cdots \sigma \left( t_{k+1} x^\ell\right) \sigma  t_k \cdots \sigma  t_1,
		\end{equation}
		with $\left(t_{k+1}x^\ell\right)$ belonging to $G_1$. Now, this product $h\fhi^\ell$ equals $\widetilde h$, so we compare the expression \eqref{eq:cyclic} above to the resulting expression from \eqref{eq:form_A(n)}. 
		From Remark~\ref{r:normal}, we deduce that the product $\widetilde t^{-1}_{k+1}\left( t_{k+1} x^\ell \right)$ is in $Z$ and thus $x^\ell \in B_1(2R'_1)Z$. 
		
		For the second case $g=\sigma^{-1}  t^{-1}_k\cdots \sigma^{-1}  t^{-1}_n\sigma^{-1}$, we consider the element $\widetilde h \fhi^{-\ell}=h$. We have
		\begin{align*}
		&\widetilde h \fhi^{-\ell} = \widetilde h g^{-1}x^{-\ell} g \\
		=\,& \sigma \widetilde t_n\cdots \sigma \widetilde t_1 \,\sigma t_n\cdots t_k\sigma \, x^{-\ell}\, \sigma^{-1}t_k^{-1}\cdots \sigma^{-1}t_n^{-1}\sigma^{-1},
		\end{align*}
		and the latter expression cannot be shortened, unless $g=id$ (because $x^{-\ell}\in G_1\setminus Z$). From Remark~\ref{r:normal}, as this expression equals $h$ which is of length $2n$, it can only be that $g=id$ and $t^{-1}_1(\widetilde t_1 x^{-\ell})\in Z$. Therefore we must have $x^{-\ell}=\fhi^{-\ell}\in B_1(2R_1')Z$ in this case.
		
		In the first case, the set $P= gB_1(2R_1')Zg^{-1}$ convenes, while in the second one we can take $P=ZB_1(2R_1')$.
	\end{proof}
	
	By Claim~\ref{cl:bound}, the cardinality of the set $P_h$ defined in \eqref{eq:setPh} is uniformly bounded by $L=\# P$. This gives the desired result.	
	\end{proof}
	
	\setcounter{claim}{0}

\begin{prop}
\label{l:cE}For any $\psi\in G$, the function 
\[c_{\psi E(n)\psi^{-1}}=\max_{h\in \psi E(n)\psi^{-1}}\#\left (\psi E(n)\psi^{-1}\cap h\,\stab{G}{x_0}\right )\]
grows at most linearly in terms of the outer radius $\rho\left (\psi E(n)\psi^{-1}\right )$. More precisely, there exists a constant $L\in\N$ such that $c_{\psi E(n)\psi^{-1}}\le L\, \rho\left (\psi E(n)\psi^{-1}\right )$.
\end{prop}

\begin{proof}
	Recall that, under our assumption of real-analytic regularity, the stabilizer of $x_0$ is either trivial or infinite cyclic (Theorem \ref{t:hector}).
	If the stabilizer $\stab{G}{x_0}$ is trivial, clearly $c_{E}$ is always $1$, no matter what $E$ is.
	Hence, we can suppose that the stabilizer $\stab{G}{x_0}$ is infinite cyclic and generated by some element $\fhi\in G$. Here we distinguish two cases, depending on whether $\fhi$ is \emph{undistorted} in $G$ or not.
	If $\fhi$ is undistorted, then the quantity $c_{\psi E(n)\psi^{-1}}$ grows at most linearly in terms of the outer radius $\rho(\psi E(n)\psi^{-1})$. If $\fhi$ is distorted, as in the proof of Lemma~\ref{l:cEdistorted}, we assume $\fhi = g^{-1}\fhi_1 g$, for some $\fhi_1\in G_1$ and $g\in G$. Let us show that in this case $c_{\psi E(n)\psi^{-1}}$ is linearly bounded in terms of the outer radius.
	
	Notice first that the quantity
	\[c_{\psi E(n)\psi^{-1}}=\max_{h\in \psi E(n)\psi^{-1}}\#\left (\psi E(n)\psi^{-1}\cap h\,\stab{G}{x_0}\right )= \max_{h\in \psi E(n)\psi^{-1}}\#\left (\psi E(n)\psi^{-1}\cap h\,g^{-1}\langle \fhi_1\rangle g\right )\]
	is also equal to
	\[\max_{h\in E(n)}\#\left (E(n)\cap h\,\psi^{-1}g^{-1}\langle \fhi_1\rangle g\psi\right )
	,\]
	therefore up to replacing $g$ above with $g\psi$, it is enough to find a uniform bound for 
	$c_{E(n)}$. This has been established with the previous Lemma~\ref{l:cEdistorted}.
	\end{proof}

As a consequence of the results of \S~\ref{ssc:close}, we obtain the following key fact. For notations appearing in the statement and proof, we refer the reader to the usual Notations~\ref{not:outer}, \ref{notation} and \eqref{eq:notation}.
\begin{cor}\label{c:close}
Given $\eps_0>0$ and $\psi\in G$, there exists $n=n(\psi)$ such that the element $g_{\psi F(n)\psi^{-1}}$ is locally $\eps_0$-close to the identity in the $C^0$ topology
when restricted to a certain complex neighbourhood of $x_0\in\NE$.
\end{cor}
\begin{proof}
Given $\psi\in G$, consider the constants $C=C(\psi)$ and $L$ from Propositions~\ref{p:exponential-bound} and~\ref{l:cE} respectively. Then the quantity
\[
\rho\left (\psi E(n)\psi^{-1}\right )\frac{c_{\psi E(n)\psi^{-1}}}{S_{\psi E(n)\psi^{-1}}}\le \frac{L}{C}\,\rho\left (\psi E(n)\psi^{-1}\right )^2 a^{-\rho\left (\psi E(n)\psi^{-1}\right )}
\]
is certainly $o(1)$ as $n$ goes to $\infty$. Thus Proposition~\ref{p:sufficient_estimate} applies and the sequence $g_{\psi F(n)\psi^{-1}}$ for $F(n)=E(n)^{-1}E(n)$ converges $C^0$ to the identity over a complex disc of size $o(1/\rho\left (\psi E(n)\psi^{-1})\right )$ around $x_0$.
\end{proof}

\subsection{Step 3: Chain of commutators}
\label{ssc:commutators}

\paragraph{Strategy --}
As we have already explained, Proposition~\ref{p:ghys} implies that if two diffeomorphisms $f_1,f_2$ in $G$ are $\eps_0$-close to the identity over a small interval, then the sequence of commutators $f_{k+2}=[f_{k+1},f_k]$ must be eventually trivial, since $G$ is locally discrete. We want to get a contradiction, finding two elements $f_1$ and $f_2$ which are locally $\eps_0$-close to $id$, generating a free subgroup in $G$. The main result in this third step is the following:

\begin{prop}\label{p:step3}
	Let $F(n), n\in\N$, be the family of subsets introduced in \eqref{eq:defFn}.
Given $\eps_0>0$, there exists $\psi_1,\psi_2\in G$ and $n$ such that the elements $f_1=g_{\psi_1F(n)\psi_1^{-1}}$ and $f_2=g_{\psi_2 F(n)\psi_2^{-1}}$ (Notation~\ref{notation}) satisfy the following two properties:
\begin{enumerate}
\item they are both $\eps_0$-close to the identity in the $C^0$ topology when restricted to a certain complex neighbourhood of $x_0\in\NE$,
\item the elements $f_3=[f_1,f_2]$ and $f_4=[f_2,f_3]$ generate a free group.
\end{enumerate} 
\end{prop}

Before starting the proof, let us describe the general strategy.
By Corollary~\ref{c:close}, for any $\psi_1,\psi_2\in G$ there exists $n$ such that the elements $f_1=g_{\psi_1F(n)\psi_1^{-1}}$ and $f_2=g_{\psi_2 F(n)\psi_2^{-1}}$ are both locally $\eps_0$-close to the identity in the $C^0$ topology when restricted to some complex neighbourhood of $x_0\in \NE$. 
By the ping-pong Proposition~\ref{l:ping-pong}, if $f_1$ and $f_2$ have disjoint invariant sets for the action on the Bass-Serre tree, then $f_3$ and $f_4$ generate a free subgroup in $G$, and the proof is over. 

\paragraph{Reduced forms for elements in $F(n)$ --}
Let $A(n)$ be the set defined as in \eqref{eq:A(n)}. Here we consider elements in the set
\[F(n)=A(n)\cup A(n)^{-1}\cup A(n)^{-1}A(n).\]
Each element in $A(n)$ can be written in the reduced form~\eqref{eq:form_good}. Also, if an element is in $A(n)^{-1}$, then its inverse is in $A(n)$.
It remains to describe the elements in $A(n)^{-1}A(n)$.

\begin{lem}\label{l:reducedA(n)-A(n)}
Let $g\in A(n)^{-1}A(n)$ be an element which does not belong to the ball $B_1(3 R_1')$ of radius $3R_1'$ in $G_1$. Then there exist elements $s,t\in G_1$ and an element $w\in G$ such that:
\begin{itemize}
\item $s,t\in B_1(R_1')\setminus Z$,
\item a reduced form representing $w$ starts and ends with a letter in $G_2\setminus Z$,
\item $g=swt$.
\end{itemize}
\end{lem}

\begin{proof}
As $g$ belongs to $A(n)^{-1}A(n)$, we can write $g$ as
\begin{equation}\label{eq:formAn-An}
g=s_1^{-1}\sigma^{-1}\cdots s_n^{-1}\sigma^{-1}\sigma t_n\cdots \sigma t_1,
\end{equation}
with $s_i,t_i\in B_1(R_1')$ and $\sigma\in G_2$ our fixed element. The problem is that the expression \eqref{eq:formAn-An} is not reduced: clearly the subword $\sigma^{-1}\sigma$ in the middle represents the identity, but there could be further central simplifications. For this, after erasing $\sigma^{-1}\sigma$, we look at the new middle subword $s_n^{-1}t_n$. It represents an element in $G_1$; if it does not belong to $Z$, then the expression
\[
g=s_1^{-1}\sigma^{-1}\cdots \sigma^{-1}(s_n^{-1} t_n)\sigma \cdots \sigma t_1,
\]
is already reduced; otherwise the subword $\sigma^{-1}s_n^{-1} t_n\sigma$ represents an element in $G_2$, and we have similar further cases to analyze. Proceeding in this way, we end up with a word $w$ such
that $g=s_1^{-1}wt_1$, and there are two possibilities:
\begin{enumerate}
\item the element $w$ is not in $Z$, and in this case we have that a reduced form representing it starts and ends with a letter in $G_2\setminus Z$,
\item or $w\in Z$ and thus $g=s_1^{-1}wt_1\in B_1(R_1')ZB_1(R_1')$ belongs to the ball $B_1(3R_1')$ (the choice of the radius $R_1'$ implies in particular that $B_1(R_1')\supset Z$).
\end{enumerate}
Because of our assumption on $g$, only the first possibility may happen, whence we get the properties of the statement, with $s=s_1^{-1}$ and $t=t_1$.
\end{proof}

\paragraph{Conjugation --} Here we determine good choices of $\psi$ so that elements in $\psi F(n)\psi^{-1}$ are suitable for ping-pong.

\begin{prop}\label{p:start-geod}
Fix $x\in G_1\setminus Z$ and $y\in G_1\setminus B_1(2R_1')$. Consider the element $\psi=x\sigma y$. Then for any element $g\in \psi \left (F(n)\setminus B_1(3R_1')\right )\psi^{-1}$, the first letter of $g$ is in $Zx^{-1}$.

In other words, if $\pi$ denotes the geodesic path going from the vertex $G_1$ to $G_1g$ in the Bass-Serre tree of $G$, then the first edge of $\pi$ is $(G_1,G_2 x^{-1})$.
\end{prop}

\begin{proof}
As $g\in \psi F(n)\psi^{-1}$, there exists an element $h\in F(n)$ such that
$g=\psi h\psi^{-1}$. We separate our discussion into two cases:
\begin{enumerate}
\item the element $h$ is in $A(n)\cup A(n)^{-1}$,
\item the element $h$ is in $A(n)^{-1}A(n)$.
\end{enumerate}
Suppose we are in the first situation, and suppose $h\in A(n)$ (the other case being similar). We write
\[
h=\sigma t_n\cdots \sigma t_1,
\]
thus
\begin{align*}
g&=\psi h\psi^{-1}\\
&= x\sigma y \sigma t_n\cdots \sigma t_1 y^{-1}\sigma^{-1}x^{-1}.
\end{align*}
We look at the subword $t_1y^{-1}$ appearing in the last expression: after our assumption on $y$, we have that the product $t_1y^{-1}$ is in $G_1$, but it does not belong to $Z$, otherwise we would have
$t_1y^{-1}\in Z\subset B_1(R_1')$ and thus $y^{-1}\in B_1(R_1')B_1(R_1')\subset B_1(2R_1')$, against our assumption.

Hence the writing 
\[g=x\sigma y \sigma t_n\cdots \sigma (t_1 y^{-1})\sigma^{-1}x^{-1}\]
is in reduced form, and it clearly starts with $x^{-1}$. If we consider another reduced form representing $g$, then we can replace the letter $x^{-1}$ by another letter in $Zx^{-1}$ (see Remark~\ref{r:normal}).

\smallskip

If we are in the second situation, the previous Lemma \ref{l:reducedA(n)-A(n)} says that we can write $h=swt$, with $s,t\in B_1(R_1')\setminus Z$ and $w$ such that a reduced form representing it starts and ends with a letter in $G_2\setminus Z$.
Hence
\[
g=\psi g\psi^{-1}=x\sigma y\,swt\,y^{-1}\sigma^{-1}x^{-1}. 
\]
Arguing as before, we get that both subwords $ys$, $ty^{-1}$ are in $G_1\setminus Z$. Therefore $g$ is represented by the reduced form
\[
g=\psi g\psi^{-1}=x\sigma (ys)w(ty^{-1})\sigma^{-1}x^{-1},
\]
and we conclude as in the previous situation.

\smallskip

The last statement about the geodesic $\pi$ is now a direct consequence of Remark~\ref{r:normal}.
\end{proof}

\begin{cor}\label{c:choose-psi12}
Take $y\in G_1\setminus B_1(2R_1')$.
If $x_1,x_2\in G_1\setminus Z$ are such that $G_2x_1^{-1}\neq G_2x_2^{-1}$, then letting
\[
\psi_1=x_1\sigma y,\quad \psi_2=x_2\sigma y,
\]
for any
\[
g_1\in \psi_1\left (F(n)\setminus B_1(3R_1')\right )\psi_1^{-1},\quad
g_2\in \psi_2\left (F(n)\setminus B_1(3R_1')\right )\psi_2^{-1},
\]
the invariant sets $X(g_1)$, $X(g_2)$ are disjoint.
\end{cor}

\begin{proof}
It follows directly from Propositions \ref{p:start-geod} and \ref{p:position}.
\end{proof}

\paragraph{End of the proof --} We are now in position to prove Proposition~\ref{p:step3}.

\begin{proof}[Proof of Proposition~\ref{p:step3}]
Consider two elements $\psi_1,\psi_2\in G$ given by Corollary~\ref{c:choose-psi12}.
Given $\eps_0>0$ we take $n$ such that the elements $f_1=g_{\psi_1F(n)\psi_1^{-1}}$ and $f_2=g_{\psi_2 F(n)\psi_2^{-1}}$ are both $\eps_0$-close to the identity in the $C^0$ topology when restricted to a certain complex neighbourhood of $x_0$, which exists after Corollary~\ref{c:close}. Since the sequences $g_{\psi_i F(m)\psi_i^{-1}}$
do not belong to a finite set (the lengths $\ell_{\psi_i F(m)\psi_i^{-1}}$, defined as in Notation~\ref{notation}, go to zero as $m\to\infty$), up to consider a larger $n$, we can suppose that $f_i\notin \psi_iG_1\psi_i^{-1}$, $i=1,2$: indeed it is easy to see that the intersection $F(n)\cap G_1$ is contained in $B_1(3R_1')$ and hence is finite (see Lemma~\ref{l:reducedA(n)-A(n)}).

Similarly, up to consider a larger $n$ (or $\eps_0$ smaller), we can suppose that the orders of $f_1$ and $f_2$ is at least $3|Z|$ (possibly infinite): if a periodic element locally converges to the identity, its order must go to infinity (\textit{cf.}~\cite[Lemma 10]{FKone}).

Corollary~\ref{c:choose-psi12} guarantees that the invariant sets $X(f_1)$ and $X(f_2)$ are disjoint.
 Then, by applying the ping-pong Proposition~\ref{l:ping-pong} (the group $G$ is $|Z|$-bounded, as in the action on its Bass-Serre tree, stabilizers of edges are conjugates of $Z$), we deduce that $f_3=[f_1,f_2]$ and $f_4=[f_2,[f_1,f_2]]$ generate a free group of rank two, as desired.
\end{proof}
This also completes the proof of Theorem~\ref{mthm:1}, as explained at the end of \S~\ref{ssc:minimal}.

%% file: Section5.tex
\subsection{Ends of the group vs.~ends of Schreier graph}

\paragraph{Duminy's and Ghys' theorems --} Theorem~\ref{mthm:1} generalizes Ghys' Theorem~\ref{thm:Ghys}, that describes groups acting with an exceptional minimal set, to minimal actions with non-expandable points. 
Our second result, Theorem~\ref{mthm:2},  also goes in this direction. As it will appear clear from the proof, the orbit of a non-expandable point plays the role of the gaps associated with an exceptional minimal set.
In this analogy, the non-expandable point is identified with a maximal gap which cannot be expanded.

\begin{ex}
	If we think of classical Fuchsian groups, actions with an exceptional minimal set (usually called Fuchsian groups \emph{of the second kind}) are semi-conjugate to minimal actions (Fuchsian groups \emph{of the first kind}). Geometrically, given a infinite volume hyperbolic surface $\mathbf{H}^2/\Gamma_0$, the semi-conjugacy is realized by contracting all infinite volume ends (topologically circular boundary components) to cusps, so to obtain a new hyperbolic surface $\mathbf{H}^2/\Gamma$ of finite volume. Here, the groups $\Gamma_0$ and $\Gamma$ are isomorphic (and free). The deformation also goes in the reverse way: given a non compact hyperbolic surface of finite volume, we can deform it by making cusps become infinite volume ends.
\end{ex}
In this perspective,  Theorem \ref{mthm:2} is the natural analogue of the celebrated Duminy's theorem \cite{navas2006}:

\begin{thm}[Duminy]
	\label{Duminy}
	Let $G\subset \Diff_+^{2}(\T)$ be a finitely generated subgroup acting on $\T$ with an exceptional minimal set $\Lambda$. Consider a connected component (a gap) $J_0$ of $\T\setminus\Lambda$. Then the Schreier graph of the orbit $X=G \cdot J_0$ has infinitely many ends.
	
	In the particular case where $G\subset \Diff_+^{\omega}(\T)$, this implies that the group $G$ itself has infinitely many ends.
\end{thm}

\paragraph{Duminy's Theorem and property $\pstar$ --} In lower regularity, the statement of Theorem~\ref{mthm:2} cannot hold, as one sees from 
the example of Thompson's group $T$. However, there is an intermediate result, on which Theorem~\ref{mthm:2} relies, that still 
holds for $C^r$ minimal non-expandable actions ($r\ge 3$):

\begin{thmA}\label{t:duminy}
	Let $G\subset \Diff_+^3(\T)$ be finitely generated subgroup of $C^r$ diffeomorphisms, such that the 
	action of $G$ is minimal, satisfies property $\pstar$ and has a non-expandable point $x_0\in \T$. Then
	the Schreier graph of the orbit of $x_0$ has infinitely many ends. 
\end{thmA}

The best plausible extension of the theorem above would be the following:

\begin{conj}\label{conj:groupoid}
	Under the hypotheses of Theorem~\ref{t:duminy}, the groupoid of germs $G_{x_0}$ has infinitely many ends. 
\end{conj}

In the statement of the conjecture, one could take for $G_{x_0}$ the groupoid of germs defined on a right or left neighbourhood of the orbit of $x_0$. A local $C^r$ diffeomorphism representing a germ in $G_{x_0}$ is defined on a right (or left) neighbourhood of a point in the orbit of $x_0$. In the following, we keep the convention of considering $G_{x_0}$ as the groupoid of \emph{right} germs. 

Despite our many efforts, we have not been able to prove Conjecture~\ref{conj:groupoid} in all its generality. However, we have the following result which will be enough for Theorem~\ref{mthm:2}:

\begin{thmA}\label{t:ends_germs}
	Let $G\subset \Diff^\omega_+(\T)$ be a finitely generated subgroup of $C^\omega$ diffeomorphisms, such that the action of $G$ is minimal, has property $\pstar$ and a non-expandable point $x_0\in \T$. Then $G$ has  infinitely many ends.
\end{thmA}

\begin{rem}
	It is important to stress that the assumption for $C^3$ regularity is unavoidable for our proof of Theorem~\ref{t:duminy}. Indeed, we are able to offer a proof only using control on the \emph{projective distortion} of the elements of the group, which classically uses the Schwarzian derivative and hence requires three derivatives.  However, we hope that Theorem \ref{t:duminy} can be generalized to actions of class~$C^2$.
\end{rem}

\paragraph{Proof of Theorem~\ref{mthm:2} from Theorems~\ref{mthm:1} and \ref{t:ends_germs} --}

Let $G\subset \Diff_+^\omega(\T)$ be a subgroup with property $\pstar$. If the set of non-expandable points $\NE=\NE(G)$ is empty, then Deroin's Theorem~\ref{t:Deroin2} implies that $G$ is $C^\omega$ conjugate to a finite central extension of a cocompact Fuchsian group.
If $G$ has an exceptional minimal set, then Ghys' Theorem~\ref{thm:Ghys} implies that $G$ is virtually free.
Therefore, we are left to suppose that $G$ acts minimally with non-expandable points. In this case we apply Theorem~\ref{t:ends_germs}: this gives that the $G$ has infinitely many ends. Since $G$ has property $\pstar$ and $\NE$ is not empty, we apply Theorem~\ref{mthm:1} and get that in the latter case $G$ is virtually free.

\subsection{Warm up: Duminy's theorem in analytic regularity}
\label{s:duminy}

In its full generality (namely, codimension-one foliations that are transversally of class $C^2$), the proof of Duminy's Theorem is a gemstone (a complete proof appears in \cite[\S{3}]{Navas2011}). Here we discuss the case of finitely generated groups of real-analytic diffeomorphisms. In this context, a similar proof was apparently already known to Hector. The proof is relatively simple because in $C^\omega$ regularity we can use Hector's lemma, but it is enlightening enough in view of the proof of Theorem~\ref{t:duminy}.

\begin{thm}[Duminy -- $C^\omega$ case]\label{t:duminy_analytic}
Let $G\subset \Diff_+^{\omega}(\T)$ be a finitely generated subgroup acting on $\T$ with an exceptional minimal set $\Lambda$. 
Let $J_0$ be a connected component of $\T\setminus \Lambda$ (a ``gap''). Then the Schreier graph $\Sch(X,\mathcal G)$ of the orbit of gaps $X=G\cdot J_0$ has infinitely many ends.

This implies that the group $G$ itself has infinitely many ends.
\end{thm}

\begin{proof}
We will prove that if the conclusion fails to be true, then $G$ preserves an \emph{affine structure} on $\T$. This is done by using control of the affine distortion 
of well chosen maps. The relevant tool to do this is the \emph{nonlinearity} of a diffeomorphism of the line: If $f:I\to J$ is a $C^2$ diffeomorphism of one 
dimensional manifolds, let
\[\mN(f)=\frac{f''}{f'}.\]
The nonlinearity of a map vanishes if and only if the map is affine. Moreover, this nonlinearity operator satisfies the cocycle relation 
\begin{equation}\label{eq:nonl-cocycle}
\mN(f\circ g)=g'\mN(f)\circ g +\mN(g).
\end{equation}

\smallskip

The first step of the proof is to use the nonlinearity to find a criterion for distinguishing different ends in the Schreier graph $\Sch(X,\mathcal G)$.
Recall that the stabilizer of $J_0$ is generated by some $h\in G$ (\textit{cf.}~Theorem \ref{t:hector}). We set $b=\int_{J_0}\mathcal N(h)$.
\begin{pdfn} Assume we are under the hypotheses of Theorem~\ref{t:duminy_analytic}.
The function
\begin{equation}\label{functionN}
\dfcn{N}{X}{\R/b\Z}{g(J_0)}{\int_{J_0}\mathcal N(g)}
\end{equation}
is well defined along the orbit $X$, and verifies
\begin{equation}\label{eq:nonlQ-increment}
N(f(J)) = N(J) + \int_{J} \mathcal N(f) \quad\text{for all }J\in X\text{ and all }f\in G.
\end{equation}
\end{pdfn}
\begin{proof}
If two elements $g_1$ and $g_2$ are such that $g_1(J_0)=g_2(J_0)$, then there exists some $k\in\Z$ such that $g_2=g_1h^k$. 
To verify that the function $N$ is well defined, we have to show that for a fixed $g\in G$, all the integrals $\int_{J_0}\mathcal N(gh^k)$ are equal modulo $b\Z$.

Using the cocycle relation \eqref{eq:nonl-cocycle} and the change of variable formula, we have
\begin{align*}
\int_{J_0}\mN(gh^k)&=\int_{J_0}(h^k)'\mN(g)\circ h^k+\sum_{i=0}^{k-1}\int_{J_0}(h^i)'\mN(h)\circ h^i\\
&=\int_{h^k(J_0)}\mN(g)+\sum_{i=0}^{k-1}\int_{h^i(J_0)}\mN(h),
\end{align*}
which is equal to $\int_{J_0}\mN(g)+k\int_{J_0}\mN(h)=\int_{J_0}\mN(g)+kb$. This proves the first assertion. The relation~\eqref{eq:nonlQ-increment} can be verified in a similar way.
\end{proof}

If $f$ is written in the form $f=g_n\cdots g_1$ in the generating system $\mathcal G$, then similarly to \eqref{eq:intermediate} we obtain the bound
\begin{equation}\label{eq:nonl-intermediate}
|N(f(J))-N(J)|\le C_\mathcal{G}\sum_{i=0}^{n-1}|g_i\cdots g_1(J)|
\end{equation}
with respect to the same constant $C_\mathcal G:=\max_{g\in \mathcal G\cup \mathcal G^{-1}}\sup_{\T} |g''/g'|$. 
From this fact it is not difficult to prove the following lemma 
which provides a criterion to distinguish ends of the Schreier graph of $J_0$. It is close to the original ideas of Duminy; see for example  \cite[Lemma 3.4.2]{Navas2011}.
Recall that the ends, and the fact that a sequence converges to a certain end, are independent of the finite system of generators of the group.

\begin{lem}
\label{nonldistinguishends}
Assume we are under the hypotheses of Theorem~\ref{t:duminy_analytic}.
\begin{enumerate}[\bf i.]
\item If $(J_n)_{n\in\N}$ is a sequence of gaps which goes to an end in the Schreier graph $\Sch(X,\mathcal G)$, then $\lim_{n\to\infty}N(J_n)$ exists.
\item If $(I_n)_{n\in\N}$ and $(J_n)_{n\in\N}$ determine the same end in the Schreier graph $\Sch(X,\mathcal G)$, then
\[\lim_{n\to\infty}N(I_n)=\lim_{n\to\infty}N(J_n).\]
\end{enumerate}
\end{lem}

\begin{proof}
It is enough to prove the first assertion for $J_n= g_{n}\cdots g_1(J_0)$, where $(g_n)_{n\in\N}$ is a sequence of elements of the (symmetric) 
system of generators of $\mathcal{G}$. To do this, notice that (\ref{eq:nonl-intermediate}) easily shows 
that the sequence $(N(J_n))_{n\in\N}$ is a Cauchy sequence, and hence converges.

To show the second assertion, given $\eps>0$, let $n_0$ be such that $\sum_{J\notin X({n_0})}|J|<\eps$, where $X({n_0})$ denotes the set of those $x\in X$ at distance no greater than $n_0$ to $J_0$ for the word distance in $X$. If $n$ is large enough, there exists a path linking $I_n$ and $J_n$ which avoids $X({n_0})$. A direct application of (\ref{eq:nonl-intermediate}) yields $|N(I_n)-N(J_n)| < \varepsilon C_{\mathcal{G}}$. Since $\eps$ is arbitrary, this concludes the proof.
\end{proof}

From now on, we suppose that $\Sch(X,\mathcal G)$ has only one end and look for a contradiction. The general case when the Schreier graph has finitely many ends can be treated similarly, as we detail in the proof of Lemma~\ref{projectiveholonomy}.

\smallskip

The second step relies on Sacksteder's theorem: there exists a local hyperbolic contraction, \textit{i.e.}~$f\in G$, $I\subset \T$ and $p\in I$ with $f'<1$ on $I$ and $f(p)=p$. Using Sternberg's (or in this case K\oe nigs-Poincar\'e's) linearization theorem, we can make a $C^\omega$ change of coordinates on $I$ and suppose that $f$ is a homothety of ratio $\mu=f'(p)$.

A way to describe an end of $\Sch(X,\mathcal G)$ is to pick some gap $J\subset I\cap G\cdot J_0$ and iterate it by $f$. Using the cocycle relation \eqref{eq:nonlQ-increment}, we find
\begin{equation}\label{nonl-limit}
\lim_{n\to\infty}N(f^n(J))=N(J),
\end{equation}
for $f$ is affine and thus its nonlinearity is $0$.

We want to prove that if there is one only end in $\Sch(X,\mathcal G)$, in this chart we have \emph{affine holonomy}: every element $\gamma\in G$ satisfying $I_\gamma=\gamma^{-1}(I)\cap I\neq \emptyset$ has to be an affine map. Note that by minimality of $\Lambda$, the union of gaps $I_{\gamma}\cap G \cdot J_0$ is dense in $I_{\gamma}$. So let $J\subset I_{\gamma}\cap G\cdot J_0$ be a gap. Since $J\subset I_{\gamma}$, we also have $\gamma(J)\subset I$.

If $\gamma\in G$  maps $J$ inside $I$ and is not a power of $f$, then the iterates of $\gamma(J)$ by $f$ also go towards the one only end of $\Sch(X,\mathcal G)$ and after \eqref{nonl-limit} we must have $N(J)=N(\gamma(J))$. Using \eqref{eq:nonlQ-increment} again, the latter implies $\int_J\mathcal N(\gamma)=0$ (supposing the gap $J$ sufficiently small, \textit{cf.}~the proof of Lemma~\ref{projectiveholonomy}).

We have just shown that the mean nonlinearity of $\gamma$ over every sufficiently small gap in $I_\gamma\cap G\cdot J_0$ vanishes. By continuity, there is a point $x_J$ in every such gap $J$, at which the nonlinearity $\mathcal N(\gamma)$ is zero. Observe that the points $x_J$ accumulate on $\Lambda\cap I_\gamma$. By the analytic continuation principle, $\gamma$ is affine on $\T$.

\smallskip

We remark that a non-abelian subgroup of automorphisms of some affine structure on $\T$ must have a finite number of globally periodic points and thus cannot  preserve a Cantor set, leading to a contradiction. Therefore, the Schreier graph $\Sch(X,\mathcal G)$ has infinitely many ends. 

\smallskip

It remains to show that the group itself has infinitely many ends. This requires some additional work: the map $\pi:g\in G\mapsto g(J_0)\in X$ defines a non-regular covering from the Cayley graph of $G$  to the the Schreier graph $\Sch(X,\mathcal G)$. The number of leaves usually bad-behaves when passing to covering spaces, unless the nontrivial monodromy of the covering is compactly supported.

The following lemma is classical in foliation theory (see \cite[Corollary~4.8]{CC-Duminy}):

\begin{lem}\label{l:nonl_compact_support}
Assume we are under the hypotheses of Theorem~\ref{t:duminy_analytic}. There exists $\eps>0$ such that the following holds.
Consider a gap $J$ in the orbit of $J_0$. Let $g\in G$ be an element that stabilizes $J$ and suppose that $g$ can be written in the form $g=g_n\cdots g_1$ in the generating system $\mathcal G$. Suppose that the intermediate images of the gap satisfy
\[
\sum_{i=0}^{n-1}|g_i\cdots g_1(J)|<\eps.
\]  
Then $g$ is the identity.
\end{lem}

Finally, arguing as in \cite[Corollaire 2.6]{euler}, we can deduce that the group $G$ has infinitely many ends. Indeed, consider the class of the loop defined by the stabilizer $h\in \stab{G}{J_0}$ in the fundamental group $\pi_1(\Sch(X,\mathcal{G}),J_0)$. After Lemma~\ref{l:nonl_compact_support}, it defines a nontrivial element in the image of the natural morphism $H^1_c(\Sch(X,\mathcal{G}),\Z)\to H^1(\Sch(X,\mathcal{G}),\Z)$. The covering $\pi:G\to \Sch(X,\mathcal{G})$ is exactly the covering associated with this element. We deduce that $G$ has infinitely many ends. 
\end{proof}

\subsection{Strategy of the proof of Theorem \ref{t:duminy}} In the setting of minimal actions with non-expandable points,   the strategy we adopt is similar to that of the proof of Duminy's theorem described above.

However, in our setting, it is not an invariant affine structure, but an invariant \emph{projective structure} that we intend to build. The relevant quantity is no longer the nonlinearity, but the \emph{Schwarzian derivative} of diffeomorphisms of one-dimensional manifolds. 

\smallskip

The first step of the proof will be to use a control of the \emph{projective} distortion. Instead of using gaps of Cantor sets,  we substitute them by considering the orbit $X$ of a non-expandable point $x_0\in\NE$. The control of the distortion is ensured by taking advantage of the Markov partition for groups acting minimally with property~$\pstar$, whose construction has been described in~\S~\ref{sc:Markovp}.

This allows to define a function $Q$ on the orbit $X$ of $x_0$, that we call the \emph{Schwarzian energy} and is analogue to \eqref{functionN}. As for the function $N$, the Schwarzian energy has a well-defined extension to the space of ends of the Schreier graph $\Sch(X,\mathcal G)$ (Lemma~\ref{distinguishends}).

\smallskip

Secondly, we first suppose that the Schwarzian energy takes only finitely many values on the space of ends of $\Sch(X,\mathcal G)$. We obtain an intermediate result, that it is interesting on its own: the group is $C^r$ conjugate to a subgroup of some finite covering of $\PSL(2,\R)$ (Theorem~\ref{t:rigidity}). The strategy follows the lines of our proof of Duminy's Theorem.
As above we take an element with a hyperbolic fixed point; using Sternberg's linearization theorem, this allows one to construct a chart with projective holonomy (see Lemma~\ref{projectiveholonomy}). Using the minimality of the action, we extend this chart to a projective structure: this is Lemma~\ref{projstructure}. Finally, relying on Kuiper-Goldman's classification of the automorphisms groups of a projective structure on $\T$, we find that the group is virtually a discrete subgroup of $\PSL(2,\R)$, with non-expandable points, and thus virtually free.

Finally, we put all the pieces together and prove Theorem~\ref{t:duminy} in \S~\ref{ssc:proofC}

\subsection{Distinguishing different ends: control of the projective distortion}

We assume that $G$ has property  $(\star)$ and that there exists $x_0\in\NE$. Our goal is to show that $G$ has infinitely many ends: here we present a criterion to distinguish two different ends.

\paragraph{Distortion control --}
From \cite[Lemma 5]{FK2012_C_eng} we have:
\begin{lem}\label{l:stab}
The stabilizer $\mathrm{Stab}_{G}(x_0)$ (in the $C^2$ setting, considered as the group of one-sided germs) is an infinite cyclic group, generated by some $h\in G$.
\end{lem}

We introduce a function $\cE:X\to (0,1]$, that we will call the \emph{energy} (and which is, in fact, the inverse of the function defined in \cite{FK2012_C_eng}), defined on the orbit $X=G\cdot x_0$ as 
\begin{equation}\label{eq:energy}
\cE(g(x_0))=g'(x_0) \quad \text{for every } g\in G.
\end{equation}
The map is well-defined. Indeed, assume that $x=g_1(x_0)=g_2(x_0)$ for $g_1,g_2\in G$. Then the element $g_2^{-1} g_1$ fixes $x_0$. Since this point is non-expandable, we must have $(g_2^{-1} g_1)'(x_0)=1$, hence $g_1'(x_0)=g_2'(x_0)$.

The energy is strongly related to the intervals appearing in the expansion sequence (Proposition~\ref{Expprocedure}), in the following precise way: 

\begin{lem}
\label{l:energy-exp}
Let $x\in X$, and let $\mathbf g_x$ be the map defined in Proposition \ref{Expprocedure}. Then the following properties hold.
\begin{enumerate}[\bf i.]
	\item We have $\cE(x)=\mathbf g'_x(x)^{-1}$.
	\item The ratio between $\cE(x)$ and $|J_x^+|$ is uniformly bounded away from $0$ and $\infty$.
\end{enumerate}  
\end{lem}
\begin{proof}
	The first statement follow directly from the definitions. The latter follows from the former and Lemma \ref{distortionlem}.
\end{proof}

\begin{lem}\label{l:series}
The series $\sum_{x\in X} \cE(x)^2$ converges.
\end{lem}
\begin{proof}
 After Lemma~\ref{l:energy-exp}, it is enough to prove that the series $\sum_{x\in X}|J_x^+|^2$ is convergent.
We can decompose this sum as

\begin{equation}
\label{sumsquares}
\sum_{k=0}^{\infty}\sum_{k(x)=k}|J_x^+|^2\le \sum_{k=0}^\infty \left ( \left (\max_{x\,:\,k(x)=k} |J_x^+|\right )\sum_{k(x)=k}|J_x^+|\right ).
\end{equation}
We first note that $|J_x^+|$ can be controlled by a term of the order of $\lambda^{-k(x)}$, because by construction we have $\mathbf g_x'(x)\geq\lambda^{k(x)}$.

Using Lemma \ref{distortionlem}, we get the following inequality holding for every $k\in\N$:
\[\sum_{k(x)=k}|J^+_x|\leq |\T|=1.\]
This suffices to prove that the upper bound in \eqref{sumsquares} is controlled by a converging geometric sum.
\end{proof}

\paragraph{The Schwarzian energy --} If $f\in\Diff^3_+(\T)$, we consider its \emph{Schwarzian derivative} given by the classical expression
\[
\mathcal S(f)=\left(\frac{f''}{f'}\right)'-\frac{1}{2}\left(\frac{f''}{f'}\right)^2.
\]
We have the following cocycle formula:
\begin{equation}\label{eq:S-comp}
\mathcal S(f\circ g)=(g')^2\cdot\mathcal S(f)\circ g + \mathcal S(g).
\end{equation}
Recall that the stabilizer of $x_0$ is generated by some $h\in G$, which moreover verifies $h'(x_0)=1$; we set $b=\mathcal S(h)(x_0)$. From this we can define a new function on the orbit $X$ of $x_0$:

\begin{pdfn}
The \emph{Schwarzian energy} is the function 
\begin{equation}\label{SchwarzianEnergy}
\dfcn{Q}{X}{\R/b\Z}{g(x_0)}{\mathcal S(g)(x_0)}
\end{equation}
(where the quotient $\R/b\Z$ can possibly be $\R$, if $b=0$).
\end{pdfn}
\begin{proof}
We follow the arguments previously given for the function $N$. We have to check that the function $Q$ is well-defined. Assume that $x=g_1(x_0)=g_2(x_0)$ for some $g_1,g_2\in G$. By Lemma \ref{l:stab}, we have $g_1=g_2 h^k$ for some $k\in\Z$. Using the cocycle relation \eqref{eq:S-comp} and the fact that $h'(x_0)=1$, we find
\[
\mathcal S(g_1)(x_0)=\mathcal S(g_2)(x_0)+k\, \mathcal S(h)(x_0).
\]
which is equal to $\mathcal S(g_2)(x_0)\pmod{b}$.
\end{proof}

An immediate corollary of~\eqref{eq:S-comp} is
\begin{equation}\label{eq:Q-increment}
Q(f(x))=\cE(x)^2\cdot \mathcal S(f)(x) + Q(x).
\end{equation}

\paragraph{Extension to the space of ends --} The following lemma provides a criterion to distinguish ends of the Schreier graph $\Sch(X,\mathcal G)$ of the orbit $X$ of $x_0$.

\begin{lem}
\label{distinguishends}
~

\begin{enumerate}[\bf i.]
\item Let $(x_n)_{n\in\N}$ be a sequence of points in $X$ which goes to an end in $\Sch(X,\mathcal G)$. Then $\lim_{n\to\infty}Q(x_n)$ exists.
\item If $(x_n)_{n\in\N},(y_n)_{n\in\N}$ go to the same end in $\Sch(X,\mathcal G)$, then
\[\lim_{n\to\infty}Q(x_n)=\lim_{n\to\infty}Q(y_n).\]
\end{enumerate}
\end{lem}

\begin{proof} The reasoning is analogous to the proof of Lemma~\ref{nonldistinguishends}, but we detail it for the sake of clarity.
Consider a sequence of the form $x_n= g_{n}\cdots g_1(x_0)$, where $(g_n)_{n\in\N}$ is a sequence of elements of the (symmetric) system of generators of $\mathcal{G}$.

Using \eqref{eq:Q-increment}, we get
\[Q(x_{n+1})-Q(x_n)=\cE(x_n)^2\cdot \mathcal S(g_{n+1})(x_n).\]
Using Lemma \ref{l:series} and an upper bound for the Schwarzian derivatives of the generators, we easily get that the sequence $(Q(x_n))_{n\in\N}$ is a Cauchy sequence, and hence converges.

\smallskip

We have the convergence of the sequence $(Q(x_n)-Q(y_n))_{n\in\N}$, and we have to prove that the limit is $0$ in the case where $x_n$ and $y_n$ converge to the same end. Let $\eps>0$ and $n_0$ be such that $\sum_{x\notin X({n_0})}\cE(x)^2<\eps$, where $X({n_0})$ denotes the set of those $x\in X$ at distance no greater than $n_0$ to $x_0$ for the word distance in $X$.

Assume that $x_n$ and $y_n$ go to the same end. When $n$ is large enough, there exists a path linking $x_n$ and $y_n$ in $\Sch(X,\mathcal G)$ which avoids $X({n_0})$. Using the same type of argument as above, we get that $|Q(x_n)-Q(y_n)|$ is smaller than $\eps$ times a uniform constant which only depends  on the system of generators. Since $\eps$ is arbitrary, this concludes the proof of the lemma.
\end{proof}

As a consequence, the function $Q$ defined in \eqref{SchwarzianEnergy} extends to the space of ends $e(X)$ of $\Sch(X,\mathcal G)$ (recall that the space of ends of a Schreier graph does not depend on the choice of the finite generating system). With abuse of notation, we also write $Q$ for this extension.

\subsection{Invariant projective structure}\label{ssc:invariant_proj}

Within this section, \emph{we will assume that the Schwarzian energy $Q$ takes finitely many values on the space of ends $e(X)$}. In particular, this holds if the Schreier graph of $x_0$ has finitely many ends, but we will show that this is never the case.
The goal is to produce a projective structure which is invariant for the action of $G$.

\begin{thm}\label{t:rigidity}
Let $G\subset \Diff^r_+(\T)$, $r\ge 3$, be a finitely generated subgroup of $C^r$ diffeomorphisms, such that the action of $G$ is minimal, has property $\pstar$ and a non-expandable point $x_0\in \T$. Suppose that the Schwarzian energy $Q$ defined on the Schreier graph $\Sch(X,\mathcal{G})$ of the orbit $X$ of the non-expandable point $x_0\in\NE$ takes finitely many values on the space of ends of $X$. Then $G$ is $C^r$ conjugate to a subgroup of some finite covering of $\PSL(2,\R)$.

In particular, the Schreier graph $\Sch(X,\mathcal G)$ has infinitely many ends and the group $G$ is virtually free.
\end{thm}

\paragraph{A projective chart --}
We begin by the construction of a single projective chart. We will next use the minimality of the action to construct a projective atlas.

The action of $G$ on $\T$ is at least $C^2$, minimal and does not preserve any probability  measure. Then Sacksteder's theorem (Theorem~\ref{t:sacksteder}) applies: the group $G$ acts on $\T$ with \emph{hyperbolic holonomy}.
More precisely, there exists a point $p\in\T$ and an element $f\in G$ with $f(p)=p$ and $\mu=f'(p)<1$. Sternberg's linearization theorem \cite[Section  3.6.1]{Navas2011} provides an interval $I$ about $p$, as well as a $C^r$-diffeomorphism $\fhi:(I,p)\to(\R,0)$, with $\fhi(p)=0$ and
\[\fhi\, f\,\fhi^{-1}=h_{\mu},\]
where $h_{\mu}$ denotes the homothety $x\mapsto\mu x$.

\begin{lem}[Projective holonomy]
\label{projectiveholonomy}
Assume that the Schwarzian energy $Q$ takes finitely many values on the space of ends of the Schreier graph of $x_0$. Then the chart $(I,\fhi)$ has projective holonomy. More precisely, for every $\gamma\in G$ such that $J=\gamma^{-1}(I)\cap I\neq\emptyset$, the following equality holds on $\fhi(J)$:
\[ \mathcal S(\fhi \gamma\fhi^{-1})=0.\]
\end{lem}

\begin{proof}
Assume that $Q$ takes finitely many values on the space of ends $e(X)$. By Lemma \ref{distinguishends}, for every $x\in I\cap X$, the limit $\lim_{n\to\infty} Q(f^n(x))$ exists and there is a finite set $\mathbf q=\{q_1,\ldots,q_\ell\}$ such that 
\[\lim_{n\to\infty} Q(f^n(x))\in \mathbf q+b\mathbb Z.\]
Now let $x=g(x_0)\in I\cap X$. Note that any homothety has zero Schwarzian derivative. Hence, the cocycle relation \eqref{eq:S-comp} implies the following equality:
\begin{align*}
Q(f^n(x))=\,&  \mathcal S(\fhi^{-1}\,  h_{\mu}^n\, \fhi  g)(x_0)\\
           =\,& \mu^{2n} (\fhi g)'(x_0)^2 \cdot\mathcal S(\fhi^{-1})(\mu^n\fhi  g(x_0))+\mathcal S(\fhi  g)(x_0).
\end{align*}
Letting $n$ go to infinity, we find $\lim_{n\to\infty} Q(f^n(x))=\mathcal S(\fhi g)(x_0)$. The latter shows that for every $g\in G$ satisfying $g(x_0)\in I$, we have that the Schwarzian derivative $\mathcal S(\fhi g)(x_0)$ belongs to the discrete set $\mathbf q+b\Z$.

Now consider a holonomy map of $I$, \textit{i.e.}~an element $\gamma\in G$ satisfying  $J=\gamma^{-1}(I)\cap I\neq\emptyset$. Note that by minimality, the set $J\cap X$ is dense in $J$. So let $x\in J\cap X$: we can write $x=g(x_0)$ for some $g\in G$. Since $x\in J$, we also have $\gamma g(x_0)=\gamma(x)\in I$. We deduce that both $\mathcal S(\fhi g)(x_0)$ and $\mathcal S(\fhi\, \gamma  g)(x_0)$ are in $\mathbf q+b\Z$. By \eqref{eq:Q-increment}, their difference is
\[\mathcal{S}(\fhi\,\gamma g)(x_0)-\mathcal S(\fhi g)(x_0)=\fhi'(x)^2\,\cE(x)^2\cdot \mathcal S (\fhi\, \gamma\fhi^{-1})(\fhi(x))\in \mathbf q-\mathbf q+b\Z.\]

The set $\mathbf q-\mathbf q+b\Z$ is discrete in $\R$ and contains $0$, so there is $\delta>0$ such that if
\[\left\vert\fhi'(x)^2\,\cE(x)^2\cdot \mathcal S (\fhi\, \gamma\fhi^{-1})(\fhi(x))\right\vert<\delta\]
then $\fhi'(x)^2\,\cE(x)^2\cdot \mathcal S (\fhi\, \gamma\fhi^{-1})(\fhi(x))=0$. Since $\fhi'(x)^2\,\cE(x)^2>0$, the latter condition implies $\mathcal S(\fhi\, \gamma\fhi^{-1})(\fhi(x))=0$.

\smallskip

By compactness, there is $M>0$ such that 
\[\sup_J\left\vert(\fhi')^2\cdot \mathcal S(\fhi\, \gamma\fhi^{-1})\circ\fhi\right\vert\le M.\]
Consider the set $X'$ of points $x\in X$ such that $\cE(x)^2<\frac{\delta}{M}$, which contains all but finitely many points of $X$.
The condition that points in $X'\cap J$ verify implies that $\mathcal{S}(\fhi\, \gamma\fhi^{-1})(\fhi(x))=0$ for every $x\in X'\cap J$. Since the orbit $X\cap J$ is dense in $J$, so is $X'\cap J$.
Hence, the Schwarzian derivative of $\fhi\, \gamma\,\fhi^{-1}$ vanishes on a dense set of $\fhi(J)$, which implies that $\fhi \,\gamma\,\fhi^{-1}$ is projective on $\fhi(J)$.
\end{proof}

\paragraph{Invariant projective structure --} By compactness of $\T$ and minimality of the action of $G$, there exists a finite number of open intervals $(I_j)_{j=1}^{m}$ and a finite number of elements of the group $(g_j)_{j=1}^m$ such that:
\begin{enumerate}
\item the family $(I_j)_{j=1}^m$ is an open cover of $\T$,
\item for every $j=1,\ldots,k$, we have $g_j(I_j)\dans I$.
\end{enumerate}

\begin{lem}[Invariant projective structure]
\label{projstructure}
For $j=1,\ldots,m$, we set $\fhi_j=\fhi\circ g_j:I_j\to\R$.
\begin{enumerate}[\bf i.]
\item The atlas $(I_j,\fhi_j)_{j=1}^{m}$ defines a projective structure  on $\T$, \textit{i.e.}~for every $j,k$ with $I_j\cap I_k\neq\emptyset$, we have:
\[\mathcal S(\fhi_k\fhi_j^{-1})=0.\]
\item The projective structure is $G$-invariant, \textit{i.e.}~for every $g\in G$ and $j,k$ satisfying $g^{-1}(I_k)\cap I_j\neq\emptyset$, we have:
\[\mathcal S(\fhi_k\, g\,\fhi_j^{-1})=0.\]
\end{enumerate}
\end{lem}

\begin{proof}
For every $g\in G$, when $g^{-1}(I_k)\cap I_j\neq\emptyset$, the map $g_k  g  g_j^{-1}$ is a holonomy map of $I$.

Hence, this lemma is a direct application of the fact that $(I,\fhi)$ has projective holonomy (see Lemma \ref{projectiveholonomy}).
\end{proof}

\paragraph{Projective structures on the circle --} On the circle, there is a canonical projective structure which is given by that of $\R\PP^1$, and whose group of automorphisms is $\mathrm{PSL}(2,\R)$.

For a general projective structure we have the following result originally due to Kuiper \cite{kuiper}, but whose proof contained a little mistake corrected by Goldman \cite{goldman1,goldman2} (\textit{cf.}~\cite{navas2006}; it also appears in \cite[Lemme 5.1]{Ghys1993}):
\begin{thm}[Kuiper--Goldman]\label{KuiperGoldman}
If the group of orientation preserving automorphisms of a $C^r$ projective structure is not abelian, then it is $C^r$ conjugate to some finite covering of $\PSL(2,\R)$.
\end{thm} 

Let us explain the main lines of the proof. The whole point is to show that the \emph{developing map} of the projective structure gives the $C^r$ diffeomorphism realizing the conjugacy.
 In what follows, we denote by $\Gamma$ the group of orientation preserving automorphisms of a projective structure on $\T$.
We also denote by $\widetilde{\mathbf{S}}^1$ and $\widetilde{\R\PP}^1$  the  universal covers of $\T$ and $\R\PP^1$ respectively. The central extension 
\begin{equation}\label{eq:central-extension}
0\to\Z\to^{\iota}\widetilde{\Gamma}\to\Gamma\to 1.
\end{equation}
defines the lift $\widetilde{\Gamma}$ of $\Gamma$ to the universal cover $\widetilde{\mathbf{S}}^1$. The injective homomorphism $\iota:\Z\to \widetilde\Gamma$ is such that the quotient $\widetilde {\mathbf S}^1/\iota(\Z)$ is diffeomorphic to $\T$.
 Similarly, we have that the universal cover $\widetilde{\PSL}(2,\R)$ of $\PSL(2,\R)$, defined by the central extension
\[
0\to\Z\to\widetilde{\PSL}(2,\R)\to\PSL(2,\R)\to 1,
\]
acts on $\widetilde{\R\PP}^1$. 

We defined a $C^r$ projective structure on $\T$ as an atlas $(I_j,\fhi_j)_{j=1}^m$ of projective charts. An equivalent way of defining it is by the data of a \emph{developing-holonomy pair} $(\mathsf{dev},\mathsf{hol})$. Here $\mathsf{hol}$ is an injective homomorphism $\mathsf{hol}:\widetilde{\Gamma}\to\widetilde{\PSL}(2,\R)$, called the \emph{holonomy representation}, and $\cD:\widetilde{\mathbf{S}}^1\to\widetilde{\R\PP}^1$ is a local diffeomorphism of class $C^r$, called the \emph{developing map}, which is $\widetilde\Gamma$-equivariant:  $\cD\circ\gamma=\mathsf{hol}(\gamma)\circ\cD$ for every $\gamma\in\widetilde{\Gamma}$.
The developing map, which is well-defined up to a post-composition by an element of $\widetilde{\PSL}(2,\R)$, globalizes the projective charts, and the holonomy representation globalizes the transition maps.

Observe that since $\iota(\Z)$ is central in $\widetilde\Gamma$, the centralizer of $\mathsf{hol}\circ \iota(\Z)$ in $\widetilde{\PSL}(2,\R)$ contains the whole image $\mathsf{hol}(\widetilde\Gamma)$. Moreover we have the following elementary fact:

\begin{lem}
The centralizer of a non-central element of $\widetilde{\PSL}(2,\R)$ is abelian. 
\end{lem}

One deduces that if $\widetilde{\Gamma}$ is not abelian, the element $\mathsf{hol}\circ\iota(1)$ is central in $\widetilde{\PSL}(2,\R)$ and so it must be an automorphism of the universal covering $\widetilde{\R\PP}^1\to\R\PP^1$. Finally one has $\cD(\widetilde{\mathbf{S}}^1)=\widetilde{\R\PP}^1$, and $\cD$ descends to a diffeomorphism between $\T$ and some $k$-fold covering of $\R\PP^1$ that conjugates $\Gamma$ to $\PSL^{(k)}(2,\R)$. In order to see that the conjugacy is $C^r$, notice that it is given by the developing map, which is $C^r$ because the projective charts are of class $C^r$.

\paragraph{Proof of Theorem~\ref{t:rigidity} --} 
The projective structure we constructed in Lemma \ref{projstructure} cannot have an abelian group of automorphism, since $G$ realizes as a subgroup and is not abelian.
Hence, the group of automorphism of our invariant projective structure has to be conjugate to some finite covering $\mathrm{PSL}^{(k)}(2,\R)$ of $\PSL(2,\R)$. 
We conclude that $G$ is $C^r$ conjugate to a subgroup of $\mathrm{PSL}^{(k)}(2,\R)$, and this subgroup is discrete in $\mathrm{PSL}^{(k)}(2,\R)$ for $G$ is locally discrete.
By definition of $\mathrm{PSL}^{(k)}(2,\R)$, there exists a $C^r$ diffeomorphism $\gamma$ of $\T$ of order $k$ that commutes with $G$. We denote by $\overline{G}$ the image of $G$ obtained by considering the action of $G$ on $\T/\langle \gamma \rangle$. Observe that $\overline G$ is $C^r$ conjugate to a discrete subgroup of $\PSL(2,\R)$.

By  assumption, there are non-expandable points, which means that there are elements in $G$ with parabolic fixed points. Hence there are elements in $\overline G$ that are $C^r$ conjugate to parabolic elements in $\PSL(2,\R)$. Hence $\overline G$ is $C^r$ conjugate to the fundamental group  of a hyperbolic surface with non-empty boundary and so it is virtually free  and with infinitely many ends. As $G$ is a finite central extension of $\overline G$, the same holds for $G$.

\subsection{Proof of Theorem~\ref{t:duminy}}\label{ssc:proofC}
Here we summarize all the work done so far in this section and prove Theorem~\ref{t:duminy}.
Consider a finitely generated subgroup $G\subset \Diff^3_+(\T)$, which acts minimally, possesses property $(\star)$, and has at least one non-expandable point $x_0$.
Consider the Schwarzian energy $Q$ defined on the Schreier graph $\Sch(X,\mathcal G)$ of the orbit of the non-expandable point $x_0$, as in \eqref{SchwarzianEnergy}. Recall that Lemma~\ref{distinguishends} ensures that the function $Q$ has a well-defined extension on the space of ends $e(X)$ of the Schreier graph of $X$.

If $Q$ takes only finitely many values on $e(X)$, we deduce from Theorem~\ref{t:rigidity} that $\Sch(X,\mathcal G)$ has infinitely many ends and the group $G$ is virtually free.
Otherwise, $Q$ takes infinitely many values and this implies that $\Sch(X,\mathcal G)$ has infinitely many ends.

%% file: Section6.tex
The main purpose of this section is to prove Theorem~\ref{t:ends_germs} about the number of ends of $G$.
Observe that it is enough to prove that the groupoid of germs $G_{x_0}$, for a non-expandable point $x_0\in \T$, has infinitely many ends: this is because in real-analytic regularity any element is uniquely determined by its germ at a given point.
Notice also that after Theorem~\ref{t:duminy} we know that the Schreier graph $\Sch(X,\cG)$ of the orbit of $x_0$ has infinitely many ends.

The proof of Theorem~\ref{t:ends_germs} relies on the following analogue to Lemma~\ref{l:nonl_compact_support} for groups with $\pstar$ (even though we have to ``discard'' some ends of $\Sch(X,\mathcal G)$). We will then deduce that if the Schreier graph of the orbit of $x_0$ has infinitely many ends, then also the groupoid of one-sided germs $G_{x_0}$ does.

Recall that points of the orbit $X$ are the vertices of the Schreier graph $\Sch(X,\cG)$.
 
\begin{prop}\label{p:nonc1}
Assume we are under the hypotheses of Theorem~\ref{t:ends_germs}. There exists a finite set $Y\subset X$ such that for at least an unbounded connected component $\cC$ of the complement $\Sch(X,\cG)\setminus Y$ the following holds. 

Consider a point $x\in \cC$ in the orbit of $x_0\in \NE$. Let $g\in G$ be an element that fixes $x$ and suppose that $g$ can be written in the form $g=g_n\cdots g_1$ with respect to the generating system $\mathcal G$. Suppose that the intermediate images of $x$ satisfy
$g_i\cdots g_1(x)\in \cC$, for any $i=1,\ldots, n$.
Then $g$ is the identity when restricted to a neighbourhood of $x$.
\end{prop}

\begin{rem}
It is possible that within the standing assumption of real-analytic regularity, the proof can be largely simplified. Here we want to provide a strategy that relies on this assumption as least as possible, hoping that Theorem~\ref{t:ends_germs} can be generalized to $C^2$ regularity (our Conjecture~\ref{conj:groupoid}). The essential property we use of $C^\omega$ regularity is given by Lemma~\ref{l:keyCw}, namely the magnification map $\mathcal R$ is always expanding (this is a consequence of the normal forms \eqref{eq:ParNormalForm} and \eqref{eq:ParNormalForm2}, which are not valid in $C^\infty$ regularity).
\end{rem}
We postpone the proof of Proposition~\ref{p:nonc1} to the end of this section, since we first need a few preliminary lemmas. Before that, let us explain how Proposition~\ref{p:nonc1} implies Theorem~\ref{t:ends_germs}.

\paragraph{Proof of Theorem~\ref{t:ends_germs} from Proposition~\ref{p:nonc1} --}

We proceed as for Theorem~\ref{t:duminy_analytic}. Consider the class of the loop defined by the stabilizer $h\in \stab{G}{x_0}$ in the fundamental group $\pi_1(\Sch(X,\mathcal{G}),x_0)$. The holonomy covering $\pi:\widehat{\Sch}(X,\cG)\to \Sch(X,\mathcal{G})$ (see Remark~\ref{r:holcover}) is exactly the covering associated with this element. 
After Proposition~\ref{p:nonc1}, there exists a finite subset $Y$ and an unbounded connected component of $\Sch(X,\cG)\setminus Y$ such that any loop contained in $\cC$ can be lifted to $\widehat{\Sch}(X,\cG)$. This implies that the pre-image of $\cC$ of the holonomy covering is homeomorphic to $\cC\times \Z$. Denote by $\phi$ the homeomorphism $\phi:\cC\times \Z\to \pi^{-1}(\cC)$.
For any $n>0$, let $Y_n$ be a finite subset contained in $\pi^{-1}(Y)$ such that $\widehat{\Sch}(X,\cG)\setminus Y_n$ contains $\phi(\cC\times [1,n])$. We deduce that $\widehat{\Sch}(X,\cG)\setminus Y_n$ contains at least $n$ unbounded connected components. Letting $n$ go to infinity, we deduce that $\widehat{\Sch}(X,\cG)$ has infinitely many ends, as desired.

\smallskip

The rest of the section is devoted to proving Proposition \ref{p:nonc1}.

\subsection{A particular case}
We begin by proving Proposition~\ref{p:nonc1} under the stronger assumption that the sum of the energies is sufficiently small: this condition is analogue to the condition of small sum of lengths in Lemma~\ref{l:nonl_compact_support}, but we replace the length by the energy $\cE$ defined as in \eqref{eq:energy}.
\begin{lem}\label{l:nonc1}
There exists $\delta>0$ with the following property:  let $g\in G$ be an element that fixes some $x\in X$ and suppose that $g$ can be written in the form $g=g_n\cdots g_1$ with respect to the generating system $\mathcal G$. Suppose that the intermediate images of $x$ satisfy
\begin{equation}\label{eq:linearsum}
\sum_{i=0}^{n-1}\cE(g_i\cdots g_1(x))<\delta.
\end{equation}
Then $g$ is the identity when restricted to a neighbourhood of $x$.
\end{lem}

\begin{proof}
Note that if~\eqref{eq:linearsum} is satisfied for some $g\in G$ and $x\in X$
then for the sum of the intermediate derivatives we have
\[
\sum_{i=0}^{n-1}(g_i\cdots g_1)'(x)=\frac{\sum_{i=0}^{n-1}\cE(g_i\cdots g_1(x))}{\cE(x)}<\frac{\delta}{\cE(x)}.
\]

We consider the restriction of $g$ to the right neighbourhood $J_x^+$ in the partition of level $\kappa(x)$ given by Proposition \ref{Expprocedure}.\ref{ass:level} (we can proceed in a similar way with the left neighbourhood $J_x^-$). 

Recall from Lemma~\ref{l:energy-exp} that the length $|J_x^+|$ is of the same order of magnitude as $\cE(x)$: there exists $C>1$, which does not depend on $x$, such that
\begin{equation}
\label{eq:energy-interval}
C^{-1}|J_x^+|\le \cE(x)\le C|J_x^+|.
\end{equation}
We can apply Proposition~\ref{l:schwartz0} to have a good control of distortion for $g$ on the interval $J_x^+$: for any $\delta<\log 2/4C_{\mathcal G}C$, Proposition~\ref{l:schwartz0} gives 
\[
\varkappa(g;J_x^+)\le 4C_{\mathcal{G}}C\delta.
\]

Now, $g'(x)=1$ because $x$ is in the orbit of a non-expandable point. Hence $g'$ is close to $1$ uniformly on $J_x^{+}$. In particular, there exists a constant $K>0$, which does not depend on $x$, such that for every $z\in J_x^{+}$ we have
\[
|g(z)-z|\le K |J_x^+|\delta.
\]
Similarly to Proposition~\ref{Expprocedure}.\ref{ass:dist}, for a sufficiently small $\delta$ we have an arbitrarily good control of distortion for the map $\mathbf g_x$ given by Proposition~\ref{Expprocedure}.\ref{ass:level}, on the interval $[z,g(z)]$; hence the ratio
\[
\frac{\mathbf g_x'(g(z))}{\mathbf g_x'(z)}
\]
is uniformly close to $1$. We conclude that the element $\widetilde g=\mathbf g_x\,g\,\mathbf g_x^{-1}$ has derivative close to $1$ on $\mathbf{g}_x(J_x^+)$. By the definition of the expanding map $\mathbf g_x$, the interval $\mathbf{g}_x(J_x^+)$ is in the finite collection $\mathcal{I}$ of Theorem~\ref{t:Markov}. The element $\widetilde g$ fixes the leftmost point of this interval.

For any interval $I^+_i$ in the collection $\mathcal I$, we know from Lemma~\ref{l:stab} that the stabilizer of its leftmost point is cyclic, generated by some element $h_i$. Choosing the constant $\delta$ such that $\widetilde{g}$ is closer to the identity on the macroscopic interval $\mathbf{g}_x(J_x^+)$ than all the $h_i^{\pm 1}$ on their corresponding $I_i^+$, we can conclude that $\widetilde g$ must be the identity on a right neighbourhood of $\mathbf g_x(x)$, and so is $g$ on a right neighbourhood of $x$, as desired.
\end{proof}

If we had that $\sum_{x\in X}\cE(x)<\infty$ we could easily use Lemma \ref{l:nonc1} in order to prove Proposition \ref{p:nonc1}, the problem is that we only know that $\sum_{x\in X}\cE(x)^2<\infty$. The idea is to lift the study of stabilizers to a ``macrocospic'' level where we control the sum of the energies so Lemma \ref{l:nonc1} applies. The lift is made by the magnification procedure described in \S~\ref{sc:Markovp}.

\subsection{Crowns}

In the following, we will take the freedom to enlarge the finite generating system when needed. This changes the graph structure of the Schreier graph, by adding edges, but by Remark~\ref{r:indep_generators}, this does not affect the assumption that $\Sch(X,\cG)$ has infinitely many ends. However increasing the number of generators leads to worse bounds on dynamical quantities, and we keep track of this by setting, given a finite generating set $\cG$, 
\begin{equation}\label{eq:chooseconstants}
M=M_{\cG}:=\max_{g\in\cG\cup\cG^{-1}}\|g'\|_0.
\end{equation}
Also, given a finite generating set $\cG$ and $\eps>0$, we define the two following nested subsets of $X$, to which we refer as \emph{crowns}:
\begin{equation}\label{eq:Xeps}
X'=X'_{\eps,\cG}:=\left\{ x\in X \mid M^{-3}\eps\le \cE(x)\le \eps\right\},
\end{equation}
and 
\begin{equation}\label{eq:Xsegeps}
X''=X''_{\eps,\cG}:=\left\{ x\in X \mid M^{-2}\eps\le \cE(x)\le M^{-1} \eps\right\}.
\end{equation}
The dependence on $\cG$ and/or $\eps$ will be omitted when there is no ambiguity.  Note that as a consequence of Lemma~\ref{l:series}, these are finite sets.
 
\begin{lem}
	\label{l:ccXeps}
Fix a finite generating set $\cG$ and assume that $\Sch(X,\cG)$ has infinitely many ends.	For every $K\in\N$ there exists $\eps_0>0$ such that for any $\eps\le \eps_0$, the complement $\Sch(X,\cG)\moins X'_\eps$ has at least $K$ unbounded connected components.

The same holds also for the minor crown $X''_\eps$.
\end{lem}

\begin{proof}
Recall that by Lemma \ref{l:series} the sum $\sum_{x\in X}\cE(x)^2$ is finite. Take the constant $M>0$ as in \eqref{eq:chooseconstants}, and for any $\eps>0$ define
\[Y_\eps:=\left\{ x\in X \mid M^{-3}\eps\le \cE(x)\right\}.\]
The family $(Y_\eps)_{\eps>0}$ is an increasing family (as $\eps$ decreases to $0$) of compact subsets exhausting $X$. Thus, as $\Sch(X,\cG)$ has infinitely many ends, for any $K\in\N$ there exists $\eps_0>0$ such that for any $\eps\le \eps_0$ the complement $\Sch(X,\cG)\setminus Y_\eps$ has at least $K$ unbounded connected components. It is then enough to prove the following.

\begin{claim*}
Every connected component of $\Sch(X,\cG)\moins Y_\eps$ is a connected component of $\Sch(X,\cG)\moins X'_\eps$.	
\end{claim*}

\begin{proof}[Proof of claim]
The subset $Z_\eps:=\left\{ x\in X \mid  \cE(x)>\eps \right\}$ is also finite and $Y_\eps$ is the disjoint union $X'_\eps\sqcup Z_\eps$.
The claim is a consequence of two following assertions.
\begin{enumerate}[\bf i.]
\item\label{i:noedge1} Every continuous path in $\Sch(X,\cG)$ whose extremities belong respectively to $Z_\eps$ and to $X\moins Y_\eps$ intersects $X'_\eps$.
\item\label{i:noedge2} Every continuous path in $\Sch(X,\cG)$ whose extremities belong to different connected components of $\Sch(X,\cG)\moins Y_\eps$ intersects $X'_\eps$.
\end{enumerate}
(Here by a \emph{path}, we always mean a path whose endpoints are vertices of $\Sch(X,\cG)$.)

Assertion \ref{i:noedge1} follows from the fact that there is no edge in $\Sch(X,\cG)$ between a point of $Z_\eps$ and a point of $X\moins Y_\eps$. The existence of such  an edge is equivalent to the existence of $y\in X$ and $g\in\cG\cup\cG^{-1}$ such that $\cE(g(y))>\eps$ and $\cE(y)<M^{-3}\eps$. This is impossible because $\cE(g(y))=g'(y)\cE(y)$ and because from the definition~\eqref{eq:chooseconstants} we must have $M\geq||g'||_0$.

Assertion \ref{i:noedge2} follows from \ref{i:noedge1}, and the fact that $Y_\eps=X'_\eps\sqcup Z_\eps$. Indeed a continuous path between two connected components of $\Sch(X,\cG)\moins Y_\eps$ must intersect $Y_\eps$. If it intersects $Z_\eps$ then it must intersect $X'_\eps$ by \ref{i:noedge1}. If not, it must intersect $X'_\eps$ anyway.
\end{proof}

This gives the result for the major crown $X'_\eps$; the proof for the minor crown $X''_\eps$ is analogous.
\end{proof}

Crowns have been introduced for two reasons. First, as we say above, they disconnect the Schreier graph into a large number of unbounded connected components, and as $\eps>0$, they exhaust $\Sch(X,\cG)$. They also enjoy a very nice dynamical property, that is we have a bound on the sum of energies.

\begin{lem}
\label{l:macroucrouille}
Let $\cG$ be a finite system of generator and $M$ the associated constant defined by \eqref{eq:chooseconstants}. Assume that $\Sch(X,\cG)$ has infinitely many ends. There exists a constant $C_0$ such that for every $\eps>0$ we have
\[\sum_{x\in X'_\eps}\cE(x)\leq C_0,\]
where $X'_\eps$ is defined as in \eqref{eq:Xeps}.
\end{lem}

\begin{proof}
%We now proceed to the proof of \ref{i:nonc2}: let us estimate $\sum_{x\in X'_\eps}\cE(x)$. We will show that it does not exceed some universal constant $C_0$ that does not depend on $\eps$. 
We consider the intervals $J_x^+$, $x\in X$, given by Proposition~\ref{Expprocedure}. 
Any two distinct intervals $J_x^{+}$, $J_y^{+}$ either  are disjoint, or one is contained into the other. As in the proof of Lemma~\ref{l:series}, we observe that in the latter case, the ratio of the lengths is larger than  $\lambda>1$ (due to control of distortion).

We claim that there exists a uniform $C_1$ such that for any $\eps>0$, any point of the circle is covered by at most $C_1$ intervals $J_x^+$, with $x\in X'_\eps$.

Indeed, let $z\in\T$ be any point and denote by $J_{x_1}^+\subset \ldots \subset J_{x_d}^+$ all the intervals containing $z$, given by points $x_i\in X'_\eps$, ordered by inclusion. On the one hand, we must have
\[
\frac{|J_{x_d}^+|}{|J_{x_1}^+|}\ge \lambda^d.
\]
On the other hand by Lemma~\ref{l:energy-exp} there exists a constant $c>1$ such that for any $x\in X$ one has
\[
c^{-1}\cdot\cE(x)\le |J_x^+|\le c\cdot\cE(x),
\]
whence
\[
\frac{|J_{x_d}^+|}{|J_{x_1}^+|}\le \frac{c\cdot \cE(x_d)}{c^{-1}\cdot\cE(x_1)}\le c^2M^3,
\]
for $x_1,x_d\in X'_\eps$.
Thus we have a uniform bound for the number of overlaps $d$ given by $\lambda^d\le c^2\,M^3$. Therefore it is enough to take $C_1$ such that
$
\lambda^{C_1}>c^2\,M^3.
$
We deduce the inequality
$
\sum_{x\in X'_\eps}|J_x^{+}|\le C_1,
$
Thus, by Lemma~\ref{l:energy-exp}, there exists $C_0$ such that
$
\sum_{x\in X'_\eps}\cE(x)\le C_0,
$
as desired.
\end{proof}

%%%%%%%%%%%%%
%%%%%%%%%%%%%
%%%%%%%%%%%%%
%%%%%%%%%%%%%

\subsection{Magnification}

\paragraph{Increased systems of generators --} We consider the Markov partition $\mathcal{I}$ consisting of atoms $I$ and associated maps $g_I$ given by Theorem \ref{t:Markov2}. We choose the Markov partition so that the \emph{magnification map} $\cR:\T\moins\Delta_0\to\T$, introduced in \eqref{eq:expR},
 is expanding: $\cR'(x)>1$ for every $x\in\T\moins\Delta_0$ (see Lemma \ref{l:keyCw}).

 \begin{conv}\label{convention}
 	Recall that $\cR$ is defined to be equal to a certain element $g_I\in G$ on every atom  $I$. To avoid imprecisions in what follows, we consider the \emph{right-continuous extension} of $\cR$ to $\T$, and keep denoting it by $\cR$. As a consequence, the conclusions of Theorem~\ref{t:thompsonlike} hold for a partition of $\T$ (we do not need to remove a finite subset $\Sigma_g\subset\T$).
 \end{conv}
 
After Theorem~\ref{t:thompsonlike}, given the Markov partition, there is a finite collection of elements $h_i\in G$ of $G$ that describe all elements upon conjugation which allows to describe elements of $G$ through sufficient magnification.
With this, we enlarge the finite generating set $\cG$:
\begin{equation}\label{eq:newgen}
\mathcal{G}_1=\mathcal G\cup \{g_I\}\cup\{h_i\}.
\end{equation}
This is also a finite generating set so by Remark~\ref{r:indep_generators} $\Sch(X,\cG_1)$ has infinitely many ends.

Apply again Theorem \ref{t:thompsonlike} to elements of $\cG_1$. For an element $g\in\cG_1\cup\cG_1^{-1}$ there exists a partition $J_1,...,J_q$ of $\T$ and integers $n_p,n_p'$ such that
\[\cR^{n_p'}g\vert_{J_p}=h_{i_p}\cR^{n_p}\vert_{J_p}.\]

The elements $\gamma\in G$ such that there exist an interval $U\subset \T$ and a power $i\le \max\{n_p,n'_p;p=1,\ldots,q\}$ such that $\gamma\vert_U=\cR^i\vert_U$, constitute a finite collection $\mathcal F$ of compositions of maps $g_I$. We define the enlarged system of generators of $G$ by
\begin{equation}\label{eq:ext_gen_set}\widetilde{\cG}=\cG_1\cup \mathcal F\cup \mathcal F^{-1},\end{equation}
which is also a finite generating set.
In what follows we will always consider the constant $M=M_{\widetilde{\mathcal G}}$ introduced in \eqref{eq:chooseconstants} with respect to this generating system $\widetilde{\cG}$. We keep working with these distinct sets of generators, as they play different roles in the course of the proof.

\begin{rem}
\label{r:increased_system}
The graph $\Sch(X,\widetilde{\cG})$ is obtained from $\Sch(X,\cG_1)$ by adding the shortcuts represented by elements of $\mathcal F$. Thus $\Sch(X,\cG_1)$ is naturally a subgraph $\Sch(X,\widetilde{\cG})$.
The set of vertices does not change, but distances decrease and different connected components of some subset $\Sch(X,\cG_1)\setminus Y$ may be merged into a unique connected component of $\Sch(X,\widetilde{\cG})\setminus Y$. In particular, every unbounded connected component of $\Sch(X,\widetilde{\cG})\moins X'$ contains an unbounded connected component of $\Sch(X,\cG_1)\moins X'$.
\end{rem}

\paragraph{Lifting paths to the crown --} We now prove how to use the magnification procedure in order to lift paths of $\Sch(X,\widetilde{\cG})$ to the macroscopic level.

\begin{lem}
\label{l:nonc2}
Let $\eps>0$. Given a point $y$ belonging to an unbounded connected component of $\Sch(X,\widetilde{\cG})\moins X_{\eps}''$ and $g\in\cG_1$,  there exist $n,m\in\N$ and $h\in\cG_1$ such that
\[
\cR^n(y), \cR^{m}(g(y))\in X_\eps',\text{ and }h(\cR^n(y))=\cR^{m}(g(y)).
\]
\end{lem}

\begin{proof}
Let $y$ be a point in an unbounded connected component of $\Sch(X,\widetilde{\cG})\setminus X''_\eps$. In particular we have
$\cE(y)<M^{-2}\eps$.
Let $g\in\cG_1$. Using the magnification procedure, there exist $n_1,m_1,i_1$ such that $\cR^{m_1}g(y)=h_{i_1}\cR^{n_1}(y)$. Note that by definition $\cR^{n_1}$ (resp. $\cR^{m_1}$) coincides in a neighbourhood of $y$ (resp. $g(y)$) with an element of the augmented set $\widetilde{\cG}$, and that $h_{i_1}\in\cG_1$.

As a consequence, performing the magnification of the element $h_{i_1}$ on an interval containing $\cR^{n_1}(y)$ provides $n_2,m_2,i_2$ such that
 \begin{align*}
 \cR^{m_1+m_2}g(y) & = \cR^{m_2}\cR^{m_1}g(y) \\
 & = \cR^{m_2} h_{i_1}\cR^{n_1}(y) \\
 & = h_{i_2} \cR^{n_2}\cR^{n_1}(y) =h_{i_2}\cR^{n_2+n_1}(y).
 \end{align*}
 Here again $\cR^{n_2}$ (resp. $\cR^{m_2}$) coincides in a neighbourhood of $\cR^{n_1}(y)$ (resp. $\cR^{m_1}(g(y))$) with an element of the augmented set $\widetilde{\cG}$, and we also have that $h_{i_1}\in\cG_1$.

Iterating this process, we obtain a tower of magnifications that we schematically represent by the following diagram:
\[
\xymatrix{
\cR^n(y) \ar[r]^{h} & \cR^m(g(y))\\
\cdots \ar[u] & \cdots \ar[u]\\
\cR^{n_1+n_2}(y) \ar[r]^{h_{i_2}} \ar[u] & \cR^{m_1+m_2}(g(y)) \ar[u] \\ 
\cR^{n_1}(y)\ar[r]^{h_{i_1}} \ar[u]^{\cR^{n_2}} & \cR^{m_1}(g(y)) \ar[u]_{\cR^{m_2}}\\
y \ar[u]^{\cR^{n_1}} \ar[r]_g & g(y) \ar[u]_{\cR^{m_1}} 
}
\]
where the ascending arrows $\cR^{n_k},\cR^{m_k}$ are restrictions of elements of $\widetilde{\cG}$ and horizontal arrows $h_j$ are elements of $\cG_1$.
The final step is chosen so that for
\[n=n_1+n_2+\ldots,\quad m=m_1+m_2+\ldots\]
we have  $\cE(\cR^n(y))\in [M^{-2}\eps,M^{-1}\eps]$, meaning that $\cR^n(y)\in X''_\eps\dans X'_\eps$. This choice is always possible because on the one hand the energy increases when applying $\cR$ (since in real-analytic regularity $\mathcal R$ is expanding, Lemma~\ref{l:keyCw}),
 but on the other it cannot grow too fast. Indeed, set $z=\cR^{n_k+\ldots +n_1}(y)$. The map $\cR^{n_{k+1}}$ coincides with an element of $\widetilde{\cG}$ in restriction to a neighbourhood of $z$ so
\[
\cE(\cR^{n_{k+1}}(z))=(\cR^{n_{k+1}})'(z)\cdot\cE(z)\leq M\cE(z).
\]
We also have
\begin{equation}\label{eq:control_magn}
\cE(\cR^m(g(y)))\in [M^{-3}\eps,\eps].
\end{equation}
Indeed, on the one hand we know that $\cR^{m}g(y)=h\cR^{n}(y)$ and  $\cE(\cR^n(y))\in [M^{-2}\eps,M^{-1}\eps]$; on the other hand we have $M^{-1}\le \|h'\|_0\le M$, by our choice of $M$, therefore from the equality
\[
\cE(\cR^m(g(y)))=\cE(h\cR^{n}(y))=h'(\cR^n(y))\cdot \cE(\cR(y))
\]
we can easily deduce \eqref{eq:control_magn}.
In particular, both $\cR^n(y)$ and $\cR^m(g(y))$ are in $ X'_\eps$. This concludes the proof of the lemma.
\end{proof}

\paragraph{Unbounded components and adjacent sets --} Let $\cC$ be an unbounded connected of complement of $\Sch(X,\widetilde{\cG})\setminus X''$. We let $X_{\cC}$ denote the set of vertices of $X''$ that are adjacent to $\cC$ in $\Sch(X,\widetilde{\cG})$, \textit{i.e.}~the set of vertices $ \partial \cC\cap X''$ in $\Sch(X,\widetilde{\cG})$. Consider also the \emph{augmented adjacent set} defined by
\begin{equation}
\label{eq:augmented_adjacent}
X^+_{\cC}=\overline{\cC}\cap X',
\end{equation}
where $\overline{\cC}=\cC\cup \partial\cC$ denotes the closure of $\cC$ in $\Sch(X,\widetilde{\cG})$.
Note that if $\cC_1$ and $\cC_2$ are disjoint unbounded connected components of $\Sch(X,\widetilde{\cG})\moins X''$ we have
\begin{equation}
\label{eq:intersectouille}
X^+_{\cC_1}\cap X^+_{\cC_2}=X_{\cC_1}\cap X_{\cC_2},
\end{equation}

An adaptation of the proof of Lemma \ref{l:nonc2} gives the following strengthening of that lemma.
\begin{lem}
\label{l:strength}
Let $\cC$ be an unbounded component of $\Sch(X,\widetilde{\cG})\moins X''$. Consider $y\in\cC$ and $g\in\cG_1$ such that $g(y)\in\cC$. 
Then there exist $n,m\in\N$ and $h\in\cG_1$ such that
\[
\cR^n(y), \cR^{m}(g(y))\in X^+_\cC,\text{ and }h(\cR^n(y))=\cR^{m}(g(y)).
\]
\end{lem}

\begin{proof}
Consider the tower defined in the proof of Lemma \ref{l:nonc2}. Note first that by definition of $\cC$ and of the system of generators $\widetilde{\cG}$ we have $\cR^i(y),\cR^j(g(y))\in\cC$ until they eventually enter $X''$. Hence the first time one of these two sequences enters  $X''$, it must belong to $X_{\cC}$.

There are three possibilities. Either the sequence $\cR^{n_k+\ldots+n_1}(y)$ first enters $X''$, or the sequence $\cR^{m_k+\ldots+m_1}(y)$ does so, or the two sequences enter $X''$ at the same time.

If the first possibility occurs, then $\cR^m(g(y))$ belongs to $\cC$ (by the same reasoning given to show \eqref{eq:control_magn}), to $X'$ (by Lemma \ref{l:nonc2}), and the $h$ obtained in Lemma \ref{l:nonc1} belongs to $\cG_1$. In that case, we then have $\cR^n(y)\in X_{\cC}\dans X^+_{\cC}$ and $\cR^m(g(y))\in X^+_{\cC}$.

If the second possibility occurs, say $\cR^{m'}(g(y))\in X_{\cC}$ and $\cR^{n'}(y)\in\cC$ where $n'=n_k+\ldots n_1$ and $m'=m_k+\ldots m_1$, then conclude as in the previous case.

If the third possibility occurs, then the two points $\cR^n(y)$ and $\cR^m(g(y))$ belong to $X_{\cC}$. That concludes the proof of the lemma.
\end{proof}

\subsection{Trivial stabilizers and conclusion}

Observe that any element $g\in G$ fixing a point of the orbit $X$ defines a loop in the Schreier graph (and \textit{viceversa}). By a local ``conjugation'' by a sufficiently large power of the magnification $\cR$, we bring any loop contained in an unbounded connected component of $\Sch(X,\widetilde{\cG})\setminus X''$, to the finite set $X'$.

\begin{lem}\label{l:nonc3}
Let $\cC$ be an unbounded connected component of $\Sch(X,\widetilde{\cG})\setminus X''$.
Consider a vertex $x\in \cC$ and an element $g\in G$ that fixes $x$ and that can be written in the form $g=g_m\cdots g_1$ with respect to the generating set $\mathcal{G}_1$, such that $g_i\cdots g_1(x)\in \cC$ for any $i\le m$.

Then there exists $N\in\N$ and an element $h\in G$ such that the following properties are satisfied:
\begin{enumerate}[\bf i.]
\item there exists a right neighbourhood $U$ of $x$ such that $\cR^{N}g\vert_U=h\cR^N\vert_U$,
\item $\cR^{N}(x)\in X^+_{\cC}$,
\item with respect to the generating system $\mathcal{G}_1$, $h$ can be written in the form $h=s_k\cdots s_1$, with $s_i\cdots s_1(\cR^N(x))\in X^+_{\cC}$ for any $i\le  k$.
\end{enumerate}
\end{lem}

\begin{proof}
%	\note{referee: I don't see why this is necessarily true for the proof, right?}
	We consider the partition $J_1,\ldots,J_p$ of $\T$ for the finite subset $\cG_1\cup \cG_1^{-1}$ given by Theorem~\ref{t:thompsonlike} (see also Convention~\ref{convention}).
Let $x\in\cC$ and $g=g_m\cdots g_1\in G$ be such as in the statement. We let $x_k$ denote the point $g_k\cdots g_1(x)$, which belongs to $\cC$ by hypothesis (we also write $x_0=x_m=x$). We apply Lemma \ref{l:strength} to every $x_{k-1},g_{k}$:
 there exist sequences $n_k,n'_k,i_k$ as well as intervals $J_{i_k}\ni x_{k-1}$ given by the magnification procedure such that
\[\cR^{n'_k}g_k\vert_{J_{i_k}}=h_{i_k}\cR^{n_k}\vert_{J_{i_k}},\]
and $\cR^{n_k}(x_{k-1}),\cR^{n'_{k}}(x_k)\in X^+_{\cC}$. 
We are tempted to locally rewrite $g=g_m\cdots g_1$ in the following way:
\begin{equation}\label{eq:formal}
\cR^{n_1}g\vert_U=\text{``}\left(\cR^{n_1}\cR^{-n_m'}h_{i_m}\cR^{n_m}\cdots\cR^{-n_2'}h_{i_2}\cR^{n_2}\cR^{-n_1'}h_{i_1}\right)\cR^{n_1}\vert_U\text{''},
\end{equation}
for some sufficiently small right neighbourhood $U$ of $x$. The reason for the quotes is that the map $\cR$ is not invertible. Nonetheless, $\cR$ is \emph{locally} invertible, in the sense that for any point $y\in \T$, there exist an element $t\in G$ and a right neighbourhood $V\ni y$ such that $t\vert_V=\cR^{-1}\vert_V$.

For this, we can use the elements $\mathbf g_y\in G$ from Proposition~\ref{Expprocedure}, which locally define the ``maximal expansion'' of a non-expandable point $y\in\NE$. There exists a right neighbourhood $U$ of $x$ such that, writing $U_k=g_k\cdots g_1(U)$ for $k=0,\ldots,m$, then $U_k\subset J_{i_k}$ and
\[\cR^{n_k}\vert_{U_{k-1}}=\mathbf g_{\cR^{n_k}(x_{k-1})}^{-1}\mathbf g_{x_{k-1}}\vert_{U_{k-1}},\quad \cR^{n'_k}\vert_{U_k}=\mathbf g_{\cR^{n'_k}(x_{k})}^{-1}\mathbf g_{x_{k}}\vert_{U_k}.
\]
In the formal equality in \eqref{eq:formal}, we  replace inverse powers  $\cR^{-n'_k}$ by the inverses of the elements representing it.
After some simplifications (and using that $x_m=x_0= x$) we find:
\[ \cR^{n_1}g\vert_U= 
\left (
\mathbf g_{\cR^{n_1}(x_{m})}^{-1} \mathbf g_{\cR^{n'_m}(x_{m})}
h_{i_m}
\mathbf g_{\cR^{n_m}(x_{m-1})}^{-1}
\cdots
 \mathbf g_{\cR^{n'_2}(x_{2})}
h_{i_2}
\mathbf g_{\cR^{n_2}(x_{1})}^{-1} \mathbf g_{\cR^{n'_1}(x_{1})}
h_{i_1}
\right )
\cR^{n_1}\vert_U.
\]
In the expression above, the compositions
\[
t_m = \mathbf g_{\cR^{n_1}(x_{m})}^{-1}\mathbf g_{\cR^{n'_{m}}(x_{m})}
\quad\ldots\quad  t_1=\mathbf g_{\cR^{n_{2}}(x_{1})}^{-1}\mathbf g_{\cR^{n'_{1}}(x_{1})}
\]
correspond to the formal compositions $\cR^{n_{k+1}-n_k'}$ in \eqref{eq:formal} (which makes sense only if $n_{k+1}\ge n_k'$), and we can write $t_k$ as a composition of maps $g_I\in\cG_1$ (if $n_{k+1}\ge n_k'$) or inverses of $g_I$ (otherwise). 
Then writing
\begin{equation}\label{eq:h}
h=
t_m
h_{i_m}
t_{m}-1
h_{i_{m-1}}
\cdots
h_{i_2}
t_1
h_{i_1},
\end{equation}
we clearly have $\cR^{n_1}g\vert_U= 
h
\cR^{n_1}\vert_U$.

We claim that $N=n_1$ and  $h$ as in \eqref{eq:h}
satisfy the conclusions of the lemma. The first property holds by construction. The second one holds by definition of $n_1$. Let us prove that the third one holds as well.

Set $y_{k-1}=\cR^{n_k}(x_{k-1})$ and $z_{k}=h_{i_k}(y_{k-1})=\cR^{n'_k}(x_k)$, so that $t_k = \mathbf g^{-1}_{y_{k}}\mathbf g_{z_{k}}$.
 By construction, we have $y_{k-1},z_{k}\in X^+_{\cC}$. Let us describe the loop based at $y_0$ defined by $h$ in the Schreier graph $\Sch(X,\cG_1)$. It links $y_{k-1}$ to $z_{k}$ via $h_{i_k}\in\cG_1$ and $z_{k}$ to $y_{k}$ via the path corresponding to $t_k$, that is a composition of maps $g_I$ or $g_I^{-1}$. 
We claim that the vertices belonging to such a path are actually in $X^+_{\cC}$. To see this, observe that the element $g_I$ that has to be applied to a vertex of such a path is locally expanding, as it locally defines $\cR$ (Lemma~\ref{l:keyCw}). We deduce that the energy is monotone along such a path. As its two endpoints are in $X'$, all the vertices in this path are in $X'$.

Now, assume without loss of generality that $n_k'<n_{k+1}$ so that $y_{k}$ is the image of $z_{k-1}$ by a positive power of $\cR$. Recall that $y_{k}$ is defined from $x_{k}$ by an inductive process. Let $y_{k}'\in\cC$ be the point obtained by the previous step of the induction, so that it belongs to $\cC$ and is at distance $1$ from $y_{k}$ in $\Sch(X,\widetilde \cG)$. There are two possibilities for a vertex $z$ on the path between $z_{k}$ and $y_{k}$. Either it belongs to $\cC$, and in that case it belongs to $X^+_{\cC}$. Or it belongs to $X''$ and is at distance $1$ from $y_{k}'\in\cC$, so it belongs to $\partial \cC$. In that case again $z\in X^+_{\cC}=\overline{\cC}\cap X''$.
\end{proof}

\paragraph{Simple loops --} Next, we decompose a loop, whose vertices are contained in $X^+_{\cC}$ into finitely many \emph{simple} loops, whose vertices are contained in $X^+_{\cC}$:
\begin{lem}\label{l:nonc4}
Let $h\in G$ be an element fixing a point $y\in X^+_{\cC}$, such that $h$ can be written in the form $h=h_m\cdots h_1$ with respect to the generating system $\mathcal G_1$, and such that
\[
h_i\cdots h_1(y)\in X^+_{\cC}\quad\text{for every }i\le m.
\]
Then there exist elements $\gamma_1,\ldots, \gamma_\alpha$ and $\beta_1,\ldots,\beta_\alpha$ such that the following properties are satisfied:
\begin{enumerate}[\bf i.]
\item $h=\beta^{-1}_\alpha\gamma_\alpha\beta_\alpha\cdots\beta_2^{-1}\gamma_2\beta_2 \beta^{-1}_1\gamma_1\beta_1$,
\item any $\gamma_j$ fixes the point $x_j=\beta_j(y)$ in $X^+_{\cC}$,
\item any $\gamma_j$ can be written in the form $\gamma_j=t_k\cdots t_1$ with respect to the generating system $\cG_1$, in such a way that for any $i=1,\ldots,k$, the points $t_i\cdots t_1(x_j)$ belong to $X^+_{\cC}$ and all are distinct. These are called \emph{simple} loops.
%THE NEXT ITEM IS ALSO TRUE BUT NOT NEEDED
%\item any $\beta_j$ can be written in the form $\gamma_j=t'_{m}\cdots t'_1$ with respect to the generating system $\mathcal{G}$, in such %a way that for any $i=1,\ldots,k$, the points $t'_i\cdots t'_1(y)$ belong to $X_{\cC}$ and all are distinct and different from $y$. %These are called {\em simple chains\/} at $y$.
\end{enumerate}
The maps $\beta^{-1}_i\gamma_i\beta_i$ are called \emph{basic loops} at $y$. Observe also that $k$ depends on $\gamma_j$ but we shall omit this dependence in the notation for the sake of readability.
\end{lem}

\begin{proof}
%As we said, we take the loop based at $y$ representing $h$, whose vertices are contained in $X_{\cC}$, and we decompose it into simple loops, whose vertices are contained in $X_{\cC}$.
The decomposition is obtained by means of an inductive process on the word length $\|h\|$ (with respect to $\cG_1$). If $\|h\|=0$ then $h=id$ and there is nothing to prove. Let us set $y_i=h_i\cdots h_1(y)$ and $y_0=y=y_m$, that belong to $X^+_\cC$ by hypotheses. Let $i_1>1$ be the first index so that $y_{i_1}=y_{j_1}$ for some $0\leq j_1< i_1$. Let $\gamma_1=h_{i_1}\cdots h_{j_1+1}$ and $\beta_1=h_{j_1}\cdots h_1$ (if $j_1\neq 0$ and the identity otherwise). By construction $\gamma_1$ induces a simple loop in $X_\cC$ based at $y_{j_1}=\beta_1(y)$. 
%One of the following situations must hold: either $h=h_m\cdots h_{i_1+1}\beta_1\beta_1^{-1}\gamma_1\beta_1$ or $h$ is itself a simple loop (and $\beta_1$ is the identity). In this later case the result is done.

Observe now that the map $\hat{h}_1=h_m\cdots h_{i_1+1}\beta_1$ satisfies the same hypotheses as $h$ but with $\|\hat h\|<\|h\|$. By the inductive hypothesis, we can write $\hat h=\beta^{-1}_\alpha\gamma_\alpha\beta_\alpha\cdots\beta_2^{-1}\gamma_2\beta_2$, with all the elements satisfying the required statements. As $h=\hat h\, \beta^{-1}_1\gamma_1\beta_1$, this gives the proof.
%
%
%If the length of $\hat{h}_1$ is $0$ then it is the identity map and $h=\beta^{-1}_1\gamma_1\beta_1$ as desired (and $h$ will be a basic loop). In any other case we can apply the same method as above to find $i_2$, $j_2$, $\gamma_2$ and $\beta_2$ so that $\beta_2(y)=y_{j_2}$ is the base point of the simple loop in $X_\cC$ induced by $\gamma_2$ and we have $\hat{h}_1=h_m\cdots h_{i_2+1}\beta_2\beta^{-1}_2\gamma_2\beta_2$ or $\hat{h}_1$ is a simple loop, this latter case trivially implies the result. Combining both decompositions we have $h=h_m\cdots h_{i_2+1}\beta_2\beta^{-1}_2\gamma_2\beta_2\beta^{-1}_1\gamma_1\beta_1$.
%
%By finite induction, it must exist a finite ordinal $\alpha$ so that $\hat{h}_{\alpha+1}=h_m\cdots h_{i_\alpha+1}\beta_\alpha$ is the identity or $\hat{h}_\alpha$ is a simple loop, and both cases give the desired result.
\end{proof}

%{\color{red}Observe that, by construction, the $\beta_i$ induce \emph{simple chains} at $y$, i.e., $\beta_i=t'_{k}\cdots t'_{1}$ with $t'_j\in\cG$ (where $k$ and the $t'_j$'s dependending on $\beta_i$) so that for all $0\leq i\leq k$ the points $t'_i\cdots t'_1(y)\in X_\cC$ are all distinct (recall $i=0$ means just $y$). This interesting fact will not be necessary for our purposes.}

\paragraph{Proof of Proposition \ref{p:nonc1} --} 
We will prove that, for an appropriate choice of $\eps>0$, the finite set $Y=X''_\eps$ defined as in \eqref{eq:Xsegeps} satisfies the required properties.

First of all, recall that we are working with different generating sets, for which we have the inclusions $\cG\subset \cG_1 \subset \widetilde{\cG}$. Hence $\Sch(X,\cG)$ is naturally a subgraph of $\Sch(X,\cG')$, which is on its turn a subgraph of $\Sch(X,\widetilde{\cG})$. In particular, a connected component of $\Sch(X,\widetilde{\cG})\setminus Y$ is the disjoint union (modulo extra edges) of connected components of $\Sch(X,\cG_1)\setminus Y$ and hence of $\Sch(X,\cG)\setminus Y$ (see Remark \ref{r:increased_system}). Therefore it is enough to prove the statement for the Schreier graph $\Sch(X,\widetilde{\cG})$.
 
By Theorem~\ref{t:duminy}, the Schreier graph $\Sch(X,\widetilde{\cG})$ has infinitely many ends. 
We take $\delta>0$ given by Lemma~\ref{l:nonc1}.
% Using Lemma \ref{l:ccXeps}, for every $K\in\N$ there exists $\eps>0$, which can be chosen arbitrarily small, such that in the complement of the set $X''=X''_\eps$,
%% defined as in \eqref{eq:Xeps}, 
%there are at least $K$ unbounded connected components. As before, for any such component $\cC$ we denote by $X_{\cC}$ the subset of $X''$ that is adjacent to $\cC$.
Using Lemma \ref{l:ccXeps}, for every $K$ there exists $\eps>0$, which can be chosen arbitrarily small, such that in the complement of the set $X''_\eps$ inside $\Sch(X,\widetilde{\cG})$, there are at least $K$ unbounded connected components. As above, for any such component $\cC$ we let $X_{\cC}$ denote the subset of $X''_\eps$ that is adjacent to it, and $X^+_{\cC}\dans X'_\eps$ the associated augmented adjacent set defined by \eqref{eq:augmented_adjacent}.

Now using Lemma \ref{l:macroucrouille} we have that $\sum_{x\in X'_\eps}\cE(x)\leq C_0$. Moreover a point $x$ can be adjacent to at most $\kappa=|\widetilde{\cG}\cup \widetilde{\cG}|$ disjoint connected components of $\Sch(X,\widetilde{\cG})\moins X''_\eps$. Finally, for two disjoint connected components $\cC_1,\cC_2$, we have $X^+_{\cC_1}\cap X^+_{\cC_2}=X_{\cC_1}\cap X_{\cC_2}$ (see \ref{eq:intersectouille}). Hence we deduce that for at most $\left \lfloor\frac{\kappa C_0}{\delta}\right \rfloor$ connected components $\cC$, the sum of energies $S_{\cC}$ satisfies
\[
S_{\cC}:=\sum_{x\in X^+_{\cC}}\cE(x)\ge\delta.
\]
Take $K> \frac{\kappa C_0}{\delta}$ and $\eps<\delta$ such that $\Sch(X,\widetilde{\cG})\moins X''_\eps$ has at least $K$ unbounded components. In this way we ensure at least one unbounded connected component $\widetilde{\cC}$ of $\Sch(X,\widetilde{\cG})\setminus X''_\eps$ satisfying $S_{\widetilde{\cC}}<\delta$.
%\emr{We show that with this condition, the component $\widetilde{\cC}$ contains an unbounded connected component of $\Sch(X,\cG)\moins X''_\eps$ that satisfies the requirements in the statement.}

\smallskip

%\emr{Consider now $\cC\dans\widetilde{\cC}$ an unbounded component of $\Sch(X,\cG)\moins X''_\eps$ and
Take $x\in{\widetilde{\cC}}$ and $g\in \stab{G}{x}$ as in the statement.
Using Lemma \ref{l:nonc3}, we find a power $N$ of $\cR$ that locally conjugates $g$ to an element $h$ that fixes a point $y$ in $X^+_{\widetilde{\cC}}$ and defines a loop whose vertices are contained in $X^+_{\widetilde{\cC}}$.

By applying Lemma~\ref{l:nonc4}, we decompose $h$ into a product $h=\beta_\alpha^{-1}\gamma_\alpha\beta_\alpha\cdots \beta_1^{-1}\gamma_1\beta_1$, with every $\gamma_i$ defining a simple loop  whose vertices are in $X^+_{\widetilde{\cC}}$. Therefore there exists a right neighbourhood $U$ of $x$ such that we have the decomposition
\begin{equation}\label{eq:nonc3}
\cR^{N}g\vert_U=\beta_\alpha^{-1}\gamma_\alpha\beta_\alpha\cdots \beta_1^{-1}\gamma_1\beta_1\cR^N\vert_U.
\end{equation}
The loops defined by the $\gamma_j$'s are simple: we can write $\gamma_j=t_k\cdots t_1$, such that all the points $t_i\cdots t_1(x_j)$ belong to $X^+_{\widetilde{\cC}}$ and all are distinct, where $x_j=\beta_j(y)$. Hence we have the upper bound
\[
\sum_{i=0}^{k-1}\cE(t_i\cdots t_1(x_j))\le S_{\widetilde{\cC}}<\delta,
\]
Then Lemma~\ref{l:nonc1} implies that the $\gamma_j$'s are trivial. We get that the decomposition \eqref{eq:nonc3} equals $\cR^Ng\vert_U = \cR^N\vert_U$. Hence also $g$ is trivial, as desired.
The proof of Proposition~\ref{p:nonc1} is now over, and with that also that of Theorem~\ref{t:ends_germs}. 